\documentclass[11pt]{article}
\usepackage{latexsym}

\usepackage[hidelinks]{hyperref}

\usepackage{amsmath}
\usepackage{amsthm}
\usepackage{amssymb}
\usepackage{mathrsfs}
\usepackage{wasysym,color}
\usepackage{graphicx}
\usepackage[title]{appendix}
\usepackage{comment}
\usepackage{bm}
\usepackage{amsfonts}
\usepackage{graphicx}   
\usepackage{relsize}    
\usepackage{float}

\def\EE{\mathbb{E}}

\def\c\delta{\mathcal{\delta}}

\newtheorem{theorem}{Theorem}
\newtheorem{Proposition}[theorem]{Proposition}
\newtheorem{corollary}[theorem]{Corollary}
\newtheorem{lemma}[theorem]{Lemma}

\theoremstyle{definition}

\allowdisplaybreaks[4]

\numberwithin{equation}{section}
\numberwithin{theorem}{section}

\begin{document}
\centerline{\Large\bf  Rough differential equations driven by TFBM }

\centerline{\Large\bf with Hurst index $H\in (\frac{1}{4}, \frac{1}{3})$
\footnote{This work was supported by the National Natural Science Foundation of China under grant 12371198.}}

\vskip 6pt \centerline{{\bf Lijuan Zhang,}~~  ~~{\bf Jianhua Huang\footnote{Corresponding author: Jianhua Huang.\\
\indent
E-mail addresses: zhanglj@nau.edu.cn,  jhhuang32@nudt.edu.cn},}
} \vskip 3pt

\centerline{\footnotesize\it School of Mathematics, Nanjing Audit University, Nanjing 210000, China}
\centerline{\footnotesize\it $^{\dag}$College of Science, National University of Defense Technology, Changsha 410073, China}

\bigbreak \noindent{\bf Abstract}{
We consider the rough differential equations driven by tempered fractional Brownian motion with Hurst index $H\in (\frac{1}{4}, \frac{1}{3})$ and tempered parameter $\lambda>0$.
First, by means of piecewise linear approximation, we canonically lift the tempered fractional Brownian motion to a three-step geometric rough path in an almost sure sense. Subsequently, employing the Doss-Sussmann technique in conjunction with a greedy sequence of stopping times, we construct a suitable transformation that establishes a bijection between the solution of the rough differential equation and that of an associated ordinary differential equation. This yields the existence and uniqueness of a solution to the original equation. Based on this result and appealing to Gronwall's lemma, we further derive an upper bound for the solution norm, thereby providing a quantitative control on its growth.



\noindent{\section{Introduction}}

This paper is devoted to studying the following rough differential equation:
\begin{align}\label{eq0.1}
dy_{t}=f(y_{t})dt+g(y_{t})dB^{H,\lambda}_{t},~~t\in [a,b],~~y_{a}\in \mathbb{R}^{d},
\end{align}
where $B^{H,\lambda}_{t}$ denotes a tempered fractional Brownian motion with Hurst index $H\in (\frac{1}{4}, \frac{1}{3})$ and tempered parameter $\lambda>0$.
This model possesses considerable application value in simulation studies. Gatheral et al. \cite{Gatheral-Jaisson} referred to it as the rough fractional stochastic volatility model and noted that it can characterize the volatility characteristics of financial markets with greater accuracy when \(H<\frac{1}{3}\); Boniece et al. \cite{Boniece-Sabzikar} demonstrated that \(H=\frac{1}{3}\) corresponds exactly to the Davenport spectrum in the inertial range of turbulence. This implies that for $H<\frac{1}{3}$, this model can not only effectively capture the rougher volatility behavior in financial markets but also reveal the low-frequency dynamical characteristics of turbulence beyond the classical inertial range, thus endowing it with profound physical implications and broad application prospects.

Rough path theory was formally introduced by T. Lyons in his pioneering work \cite{Lyons-1998} in 1998.
This theory provides a systematic mathematical framework for analyzing the dynamical behaviors of differential equations driven by low-regularity paths.
By appropriately enriching low-regularity driving paths with ``additional structural information'', Lyons lifts the original trajectory to a rough path in a higher-dimensional tensor space. Based on this lift, he defines the rough path integral and establishes a well-posedness theory for solutions of rough differential equations.
This theory not only solves the well-posedness issue of differential equations driven by irregular paths but also plays a crucial role in the research of various singular equations, such as the Kardar-Parisi-Zhang equation and the Navier-Stokes equations driven by singular external forces.

As research progresses, the theoretical framework concerning the dynamical behavior of stochastic differential equations driven by rough paths has become increasingly refined. The existence and uniqueness of solutions to rough differential equations in the sense of Friz-Victoir were first established in \cite{Riedel-Scheutzow}. Local exponential stability of solutions to differential equations perturbed by rough noise was proved in \cite{Garrido-Atienza-Schmalfu?}. The existence of stochastic attractors for rough partial differential equations driven by nonlinear multiplicative rough paths was explored in \cite{Yang-Lin-Zeng}. A. Neamtu et al. \cite{Neamtu-2021} analyzed invariant manifolds for rough differential equations using an appropriately discretized Lyapunov-Perron-type method. L. Duc \cite{Due-2022-1} investigated the existence and upper semicontinuity of global pullback attractors for rough differential equations, and further explored asymptotic stability and global stability criteria in \cite{Due-2022-2}. The upper semicontinuity of stochastic attractors for evolution equations driven by rough noise was studied in \cite{Cao-Gao-2025}.

In the aforementioned studies on stochastic differential equations driven by rough noise, most results have focused on the case where the H\"{o}lder exponent of the rough noise satisfies $\alpha\in (\frac{1}{3}, \frac{1}{2})$. However, research on systems driven by rough noise with lower regularity ($\alpha<\frac{1}{3}$) remains relatively underdeveloped. The aim of this paper is to establish an existence and uniqueness theory for solutions to rough differential equations driven by tempered fractional Brownian motion with Hurst index $H\in (\frac{1}{4}, \frac{1}{3})$ and tempered parameter $\lambda>0$. This work consists of two parts:

\begin{itemize}
\item [(1)]First, we lift the tempered fractional Brownian motion to the framework of three-step geometric rough paths. The core method is based on piecewise linear approximation, and by appealing to the Borel-Cantelli lemma, we prove the convergence of the approximating sequence in an almost sure sense, thereby completing the construction of the lift. In this process, the key lies in a deep exploitation of the covariance structure of the tempered fractional Brownian motion and the relevant fundamental properties of the modified Bessel function of the second kind $K_{H}(t)$
 -- including differentiation formulas, recurrence relations, and boundedness estimates -- in order to overcome the divergence of the series appearing in the Borel-Cantelli lemma caused by the regime $H\in (\frac{1}{4}, \frac{1}{3})$. This ensures the almost sure convergence of the piecewise linear approximation.

\item [(2)]Subsequently, by applying the Doss-Sussmann technique, we construct a transformation $y_t=\varphi(t,\mathbf{B}^{H, \lambda},z_t)$, which establishes a one-to-one correspondence between the solution $y_t$ of the rough differential equation \eqref{eq0.1} and the solution $z_t$ of an associated ordinary differential equation. We first prove the existence and uniqueness of the solution to this ordinary differential equation on a small local time interval, thereby obtaining the local existence and uniqueness of the solution to the original rough differential equation \eqref{eq0.1} on that interval under the following assumptions:

$\mathrm{(H_1)}$ $f: \mathbb{R}^d \to \mathbb{R}^d$ is globally Lipschitz continuous with the Lipschitz constant $C_f$;

$\mathrm{(H_2)}$ $g$ either belongs to $C^3_b(\mathbb{R}^d, \mathcal{L}(\mathbb{R}^m, \mathbb{R}^d))$ such that
\begin{align}\label{eq0.2}
    \|g\|_{\infty, C_g} := \max\left\{ \|g\|_\infty, \|Dg\|_\infty, \|D^2_g\|_\infty, \|D^3_g\|_\infty \right\} < \infty.
\end{align}
To extend this result to arbitrary intervals, we construct a greedy sequence of stopping times that partitions the time axis. On each small subinterval between consecutive stopping times, we repeatedly apply the Doss-Sussmann technique and, by concatenating the solutions, establish the existence and uniqueness of the solution to \eqref{eq0.1} on the whole interval. Building on this, we further derive an upper bound for the norm of the solution by means of Gronwall's lemma.
\end{itemize}

The paper is organized as follows. In Section 2, we collect some necessary preliminaries on TFBM and rough path theory. In Section 3, we lift the tempered fractional Brownian motion to the framework of three-step geometric rough paths. In Section 4, we establish the existence and uniqueness theory for solution to rough differential equations driven by TFBM, as well as some important estimates for the solution. In Appendix, we give some very
useful properties which will be frequently used throughout the paper.

\noindent{\section{Preliminaries}}

\noindent{\subsection{TFBM}}

A tempered fractional Brownian motion (TFBM), denoted by $\{B^{H, \lambda}_{t}\}_{ t\geq0}$,
with Hurst index $H\in(0,1)$ and tempering parameter $\lambda>0$,
was first introduced by Meerschaert and Sabzikar \cite{Meerschaert-1}.
It is a centered Gaussian stochastic process, satisfying
$\EE[B^{H,\lambda}_{t}]=0$ for any $t\geq0$,
and possessing the covariance structure
\begin{align}\label{eq2.1}
\EE[B^{H,\lambda}_{s}B^{H,\lambda}_{t}]=
\frac{1}{2}\bigg[C_{t}^{2}|t|^{2H}
+C_{s}^{2}|s|^{2H}-C_{t-s}^{2}|t-s|^{2H}\bigg]
\end{align}
for any $t,s \geq0$. Here the function $C_{t}^{2}$ is defined for $t>0$ by
\begin{align}\label{eq2.1-1}
C_{t}^{2}
&=\int_{\mathbb{R}}\bigg[e^{-\lambda |t|(1-x)_{+}}(1-x)^{-\alpha}_{+}-e^{-\lambda |t|(-x)_{+}}(-x)^{-\alpha}_{+}\bigg]^{2}dx\nonumber\\
&=\frac{2\Gamma(2H)}
{(2\lambda)^{2H}|t|^{2H}}-\frac{2\Gamma(H+\frac{1}{2})}
{\sqrt{\pi}(2\lambda)^{H}|t|^{H}}K_{H}(\lambda |t|),
\end{align}
with $C_{0}^{2}=0$. In the above expressions,
$(x)_{+}=xI_{(x>0)}$, $\alpha=\frac{1}{2}-H$, and
$K_{v}(t)=\int_{0}^{\infty}e^{-t \cosh x}\cosh (vx)dx$ denotes the modified Bessel function \cite{Gaunt} of the second kind for $v>0$ and $t>0$.

\noindent{\subsection{ Rough paths}}

Let $I = [\min I, \max I] \subset \mathbb{R}$ be a compact time interval, and denote $|I| := \max I - \min I$ and $I^2 := I \times I$. For any finite-dimensional vector space $W$, let $C(I,W)$ be the space of continuous paths $y: I \to W$ equipped with the supremum norm $\| \cdot \|_{\infty,I}$ given by $\| y \|_{\infty,I} = \sup_{t \in I} \| y_t \|$, where $\| \cdot \|$ is the norm in $W$. We write $y_{s,t} := y_t - y_s$.

For $p \geq 1$, define $C^{p-\text{var}}(I,W) \subset C(I,W)$ as the space of continuous paths with finite $p$-variation, i.e.,
\[||| y |||_{p-\text{var},I} := \left( \sup_{\Pi(I)} \sum_{i=1}^n \| y_{t_i,t_{i+1}} \|^p \right)^{1/p} < \infty,\]
where the supremum is taken over all finite partitions $I$. It is well known that the map $||| y |||^p_{p-\text{var},I}$
 is a control, i.e., it satisfies
\begin{align}\label{eq0-2.1}
    ||| y |||^p_{p-\text{var},[s,s]} = 0,
    \quad
    ||| y |||^p_{p-\text{var},[s,u]} +
    ||| y |||^p_{p-\text{var},[u,t]} \leq
    ||| y |||^p_{p-\text{var},[s,t]}, \quad \forall s \leq u \leq t.
\end{align}
Endowed with the norm $||| y |||_{p-\text{var},I} := ||| y_{\min I} ||| + ||| y |||_{p-\text{var},I}$, the space $C^{p-\text{var}}(I,W)$ is a nonseparable Banach space; see \cite[Theorem 5.25, p. 92]{Friz-Victoir}.
For $0 < \alpha < 1$, let $C^\alpha(I,W)$ denote the space of $\alpha$-H\"older continuous functions on $I$, equipped with the norm
\begin{align}\label{eq0-2.2}
||| y |||_{\alpha,I} := ||| y_{\min I} ||| + ||| y |||_{\alpha,I}, \quad \text{where } ||| y |||_{\alpha,I} := \sup_{\substack{s,t \in I, s < t}} \frac{||| y_{s,t} |||}{(t - s)^\alpha} < \infty.
\end{align}
Now fix $\alpha \in \left( \frac{1}{4}, \frac{1}{2} \right)$. A triple $\mathbf{x} = (x,X^{2}, X^{3})$,
where $x\in C^\alpha(I,\mathbb{R}^m)$ and
\begin{align}\label{eq0-2.3}
    X^{2} \in C^{2\alpha}(I^2,\mathbb{R}^m \otimes \mathbb{R}^m) := \bigg\{ X^{2} \in C(I^2,\mathbb{R}^m \otimes \mathbb{R}^m): \sup_{\substack{s,t \in I, s < t}} \frac{\| X^{2}_{s,t} \|}{(t - s)^{2\alpha}} < \infty \bigg\},
\end{align}
\begin{align}\label{eq0-2.4}
    X^{3} \in C^{3\alpha}(I^2,\mathbb{R}^m \otimes \mathbb{R}^m  \otimes \mathbb{R}^m) := \bigg\{ X^{3} \in C(I^2,\mathbb{R}^m \otimes \mathbb{R}^m): \sup_{\substack{s,t \in I, s < t}} \frac{\| X^{3}_{s,t} \|}{(t - s)^{3\alpha}} < \infty \bigg\},
\end{align}
is called a rough path if it satisfies Chen's relation
\begin{align}\label{eq0-2.5}
    X^{2}_{s,t} -
    X^{2}_{s,u} -
    X^{2}_{u,t} = x_{s,u} \otimes x_{u,t},
\end{align}
\begin{align}\label{eq0-2.6}
    X^{3}_{s,t}=
    X^{3}_{s,u} +
    X^{3}_{u,t} +
    X^{2}_{s,u}\otimes x_{u,t}+
    x_{s,u}\otimes X^{2}_{u,t}.
\end{align}
$X^{2}$ is called a \textit{L\'evy area} for $x$ and is viewed as postulating the value of the quantity
\begin{align}
\int_s^t x_{s,r} \otimes dx_r :=
X^{2}_{s,t},\quad
\int_s^t\int_s^r
x_{s,u}\otimes dx_{u} \otimes dx_{r}
=X^{3}_{s,t},
\nonumber
\end{align}
where the right hand side is taken as a definition for the left hand side.
Let $\mathscr{C}^\alpha(I,\mathbb{R}^m \oplus (\mathbb{R}^m \otimes \mathbb{R}^m)
\oplus
(\mathbb{R}^m \otimes \mathbb{R}^m\otimes \mathbb{R}^m))
\subset
C^\alpha(I,\mathbb{R}^m) \oplus C^{2\alpha}(I^2,\mathbb{R}^m \otimes \mathbb{R}^m)
\oplus
C^{2\alpha}(I^2,\mathbb{R}^m \otimes \mathbb{R}^m\otimes \mathbb{R}^m)$ be the set of all rough paths on $I$ (or in short $\mathscr{C}^\alpha(I)$),
then $\mathscr{C}^\alpha(I)$ is a closed set (but not a linear space), equipped with the rough path semi-norm
\begin{align}\label{eq0-2.7}
    ||| \mathbf{x} |||_{\alpha,I} := ||| x |||_{\alpha,I} + ||| X^{2} |||_{2\alpha,I^2}
    +||| X^{3} |||_{3\alpha,I^2},
\end{align}
where
\begin{align}
||| X^{2} |||_{2\alpha,I^2} := \sup_{\substack{s,t \in I, s < t}} \frac{\| X^{2}_{s,t} \|}{(t - s)^{2\alpha}} < \infty,\nonumber
\end{align}
\begin{align}
||| X^{3} |||_{3\alpha,I^2} := \sup_{\substack{s,t \in I, s < t}} \frac{\| X^{3}_{s,t} \|}{(t - s)^{3\alpha}} < \infty.\nonumber
\end{align}
Throughout this paper, we will fix parameters $\frac{1}{4} < \alpha < \nu < \frac{1}{2}$ and $p = \frac{1}{\alpha}$ so that $C^\alpha(I,W) \subset C^{p-\text{var}}(I,W)$. We also consider the $p$-var semi-norm
\begin{align}\label{eq0-2.8}
    ||| \mathbf{x} |||_{p-\text{var},I} := \bigg( ||| x |||^p_{p-\text{var},I} + ||| X^{2} |||^{\frac{p}{2}}_{\frac{p}{2}-\text{var},I^2}
    + ||| X^{3} |||^{\frac{p}{3}}_{\frac{p}{3}-\text{var},I^2}
    \bigg)^{\frac{1}{p}},
\end{align}
where
\begin{align}\label{eq0-2.9}
||| X^{2} |||_{\frac{p}{2}-\text{var},I^2} := \bigg( \sup_{\Pi(I^2)} \sum_{i=1}^n \| X^{2}_{t_i,t_{i+1}} \|^{\frac{p}{2}} \bigg)^{2/p},
\end{align}
\begin{align}\label{eq0-2.10}
||| X^{3} |||_{\frac{p}{3}-\text{var},I^2} := \bigg( \sup_{\Pi(I^2)} \sum_{i=1}^n \| X^{3}_{t_i,t_{i+1}} \|^{\frac{p}{3}} \bigg)^{3/p}.
\end{align}

\noindent{\subsection{Rough integrals}}

A continuous path $y \in C^\alpha(I,W)$ is said to be controlled by $x \in C^\alpha(I,\mathbb{R}^m)$ if there exists a triple $(y', y'', R^y)$ with $y' \in C^\alpha(I,\mathcal{L}(\mathbb{R}^m,W))$,
$y'' \in C^{\alpha}(I^{2},\mathcal{L}(\mathbb{R}^m,W))$,
$R^{y\sharp} \in C^{3\alpha}(I^2,W)$,
$R^{y\sharp \sharp} \in C^{2\alpha}(I^2,W)$, such that the following decomposition holds for all $s,t\in I$:
\begin{align}\label{eq0-2.11}
    y_{s,t} = y'_s x_{s,t} + y''_s X^{2}_{s,t}+
    R^{y\sharp}_{s,t},
\end{align}
\begin{align}\label{eq0-2.11-1}
    y'_{t}-y'_{s} = y''_s x_{s,t}+
    R^{y\sharp \sharp}_{s,t}.
\end{align}
The paths $y'$ and $y''$ are referred to as the \textit{Gubinelli derivatives} of $y$.

Let $\mathscr{D}^{2\alpha}_x(I)$ denote the space of all triples $(y,y',y'')$ controlled by $x$. Endowed with the norm
\begin{align}\label{eq0-2.12}
    \| (y,y',y'') \|_{x,2\alpha,I} := \| y_{\min I} \| + \| y'_{\min I} \| + \| y''_{\min I} \| +||| (y,y',y'') |||_{\mathscr{D}^{2\alpha}_x(I),2\alpha,I},
\end{align}
where
\begin{align}\label{eq0-2.13}
|||(y,y',y'') |||_{\mathscr{D}^{2\alpha}_x(I),2\alpha,I} := ||| y'' |||_{\alpha,I} +||| R^{y\sharp } |||_{3\alpha,I^{2}}+ \| R^{y\sharp \sharp} \|_{2\alpha,I^2}.
\end{align}
Given a fixed rough path $\mathbf{x} = (x,X^{2},X^{3})$ and a controlled path $(y,y',y'') \in \mathscr{D}^{2\alpha}_x(I)$, the rough integral $\int_s^t y_u dx_u$ is defined as the limit of Riemann-Stieltjes type sums:
\begin{align}\label{eq0-2.14}
    \int_s^t y_u dx_u := \lim_{| \Pi | \to 0} \sum_{[u,v] \in \Pi} \bigg( y_u \otimes x_{u,v} + y'_u X^{2}_{u,v}+ y''_u X^{3}_{u,v}  \bigg),
\end{align}
where the limit is taken over all finite partitions $\Pi$ of $I$ with $| \Pi | := \max_{[u,v] \in \Pi} |v - u|$. Moreover, there exists a constant $C_\alpha = C_{\alpha,|I|} > 1$ such that
\begin{align}\label{eq0-2.15}
&\bigg\| \int_s^t y_u dx_u - y_s \otimes x_{s,t} -
y'_s X^{2}_{s,t}
-y''_s X^{3}_{s,t}
\bigg\| \nonumber\\
&\leq C_\alpha |t - s|^{4\alpha} \bigg( ||| x |||_{\alpha,[s,t]} ||| R^{y\sharp} |||_{3\alpha,[s,t]^2} \nonumber\\
&+
||| R^{y\sharp \sharp} |||_{2\alpha,[s,t]^2}
||| X^{2}|||_{2\alpha,[s,t]}
+
||| y'' |||_{\alpha,[s,t]} ||| X^{3} |||_{3\alpha,[s,t]^2} \bigg).
\end{align}

In this paper, we frequently employ the $p$-variation norm. For a controlled path $(y,y',y'')$, we set
\begin{align}\label{eq0-2.16}
    \| (y,y',y'') \|_{x,p,I} := \| y_{\min I} \| + \| y'_{\min I} \| + \| (y,y',y'') \|_{x,p,I},
\end{align}
where
\begin{align}\label{eq0-2.17}
||| (y,y',y'') |||_{x,p,I} := ||| y'' |||_{p-\text{var},I} +||| R^{y\sharp}|||_{\frac{p}{3}-\text{var},I^2}
+||| R^{y\sharp\sharp}|||_{\frac{p}{2}-\text{var},I^2}.
\end{align}
A corresponding estimate in the $p$-variation setting, analogous to
\eqref{eq0-2.16} , is given by
\begin{align}\label{eq0-2.17}
    &\bigg\| \int_s^t y_u dx_u - y_s \otimes x_{s,t} - y'_s X^{2}_{s,t}-
    y''_s X^{3}_{s,t}\bigg\|\nonumber\\
    &\leq C_p \bigg( ||| x |||_{p-\text{var},[s,t]} ||| R^{y\sharp} \|_{\frac{p}{3}-\text{var},[s,t]^2} +
    ||| X^{2} |||_{\frac{p}{2}-\text{var},[s,t]} ||| R^{y\sharp\sharp} \|_{\frac{p}{2}-\text{var},[s,t]^2}
    \nonumber\\
    &+
    ||| y'' \|_{p-\text{var},[s,t]} \| X^{3} \|_{\frac{p}{3}-\text{var},[s,t]^2} \bigg),
\end{align}
where the constant $C_p > 1$ is independent of the rough path $\mathbf{x}$ and the controlled path $(y,y',y')$.

\subsection{Greedy sequence of stopping times}
Throughout this paper, we will employ the concept of a greedy sequence of stopping times, as introduced in  \cite{Cass-Litterer-Lyons, Cong-Duc-Hong, Duc-Hong}. Given an exponent $\frac{1}{p} \in \bigg( \frac{1}{4}, \nu \bigg)$  and a parameter $\nu\in (0,1)$, we construct a sequence $\{ \tau_i(\gamma, I, p) \}_{i \in \mathbb{N}}$ of greedy times with respect to the $p$-variation norm by setting
\begin{align}\label{eq0-2.18}
    \tau_0 = \min I, \quad \tau_{i+1} := \inf \bigg\{ t > \tau_i : ||| \mathbf{x} |||_{p-\text{var}, [\tau_i, t]} = \gamma \bigg\} \wedge \max I.
\end{align}
Let $N_{\gamma, I, p}(\mathbf{x}) := \sup\{ i \in \mathbb{N} : \tau_i \leq \max I \}$ denote the number of intervals generated by this procedure. Then the following elementary estimate holds:
\begin{align}\label{eq0-2.19}
    N_{\gamma, I, p}(\mathbf{x}) \leq 1 + \gamma^{-p} ||| \mathbf{x} |||^p_{p-\text{var}, I}.
\end{align}
For the H\"{o}lder setting, given $\alpha \in \bigg( \frac{1}{4}, \nu \bigg)$, we construct an alternative greedy sequence
$\{ \bar{\tau}_i(\gamma, I, \alpha) \}_{i \in \mathbb{N}}$ with respect to the $\alpha$-H\"older norm via
\begin{align}\label{eq0-2.20}
    \bar{\tau}_0 = \min I, \quad \bar{\tau}_{i+1} := \inf \bigg\{ t > \bar{\tau}_i : (t - \bar{\tau}_i)^{1-2\alpha} + ||| \mathbf{x} |||_{\alpha-\text{Hol}, [\bar{\tau}_i, t]} = \gamma \bigg\} \wedge \max I.
\end{align}
Defining $N_{\gamma, I, \alpha}(\mathbf{x}) := \sup\{ i \in \mathbb{N} : \bar{\tau}_i \leq \max I \}$, we obtain the bound
\begin{align}\label{eq0-2.21}
    N_{\gamma, I, \alpha}(\mathbf{x}) \leq 1 + |I| \gamma^{-\frac{1}{\nu-\alpha}} \left( 1 + ||| \mathbf{x} |||^{\frac{1}{\nu-\alpha}}_{\nu-\text{Hol}, I} \right).
\end{align}

\noindent{\section{Geometric rough path over TFBM}}

In this section, we show that, almost surely on the canonical Wiener space, the sample paths of tempered fractional Brownian motion (TFBM) admit a canonical enhancement up to the third-level, thereby inducing a geometric rough path of roughness $ p \in (2,4)$.

Without loss of generality, we restrict our attention to the time interval $[0,1]$ in what follows.
For each $ n \geq 1 $ and $ k = 0, 1, \dots, 2^n $,
define the dyadic partition of $[0,1]$ by $ t_k^n = \frac{k}{2^n }$.
Let $ B_t^{H,\lambda;n} $ be the piecewise linear approximation of $ B_t^{H,\lambda} $ along this partition.
Since the sample paths of $ B_t^{H,\lambda;n} $ are smooth (i.e., of bounded variation),
they possess a unique canonical lift to a geometric rough path:
\begin{align}
\mathbf{B}_{s,t}^{H, \lambda;n} = (1,B_{s,t}^{H, \lambda;n,1},B_{s,t}^{H, \lambda;n,2},
B_{s,t}^{H, \lambda;n,3}), \quad 0 \leq s < t \leq 1,
\nonumber
\end{align}
where the higher-order components are given by iterated integrals.
This lift belongs to the space $\mathbf{\Omega}_p(\mathbb{R}^d)$ of geometric rough paths of roughness $ p$ (in fact, for any $ p \geq 1 $).

We now aim to prove that the sequence $\{\mathbf{B}_{s,t}^{H, \lambda;n}\}_{n\geq1}$ converges almost surely in the $ p $-variation  metric $ d_p$.
This convergence implies that, almost surely, the sample paths of $ B_{t}^{H, \lambda}$ admit a third-level enhancement to a geometric rough path with roughness $ p$, defined as the limit of $\mathbf{B}_{s,t}^{H, \lambda;n}$ under $ d_p$. Such an enhancement can be interpreted as the canonical lift of TFBM obtained via dyadic approximation.
Before establishing this convergence, we first derive the following key estimates for each level.

\begin{lemma}
For any $ m, n \geq 1$, $ k = 1, 2, \cdots, 2^n $, $H\in (\frac{1}{4}, \frac{1}{2}]$, $\lambda>0$ and $p\in (2,4)$,  the following estimate holds at the first level paths:
\[
[\mathbb{E} | B_{t_{k-1}^{n}, t_{k}^{n}}^{H, \lambda; m, 1} |^p ]^{\frac{1}{p}}\leq
\begin{cases}
c[C^{2}_{\frac{1}{2^n}}]^{\frac{1}{2}}
\frac{1}{2^{nH}}, & n \leq m; \\
c[C^{2}_{\frac{1}{2^m}}]^{\frac{1}{2}}
\frac{2^{m(1-H)}}{2^n}, & n > m,
\end{cases}
\]
\[
[\mathbb{E} |B_{t_{k-1}^{n}, t_{k}^{n}}^{H, \lambda;m+1, 1} - B_{t_{k-1}^{n}, t_{k}^{n}}^{H, \lambda; m, 1}|^p ]^{\frac{1}{p}} \leq
\begin{cases}
0, & n \leq m; \\
c[C^{2}_{\frac{1}{2^m}}]^{\frac{1}{2}}
\frac{2^{m(1-H)}}{2^n}, & n > m,
\end{cases}
\]
where $c$ is some positive constant not depending on $m, n, k$.
\end{lemma}

\begin{proof}

For $ n \leq m $, we have
\begin{align}
B_{t_{k-1}^{n}, t_{k}^{n}}^{H, \lambda; m, 1} = B_{t_{k}^{n}}^{H, \lambda} - B_{t_{k-1}^{n}}^{H, \lambda}.\nonumber
\end{align}
Owing to the equivalence between the $L^{2}$-norm and the
$L^{p}$-norm for a Gaussian random variable, it follows from \eqref{eq2.1} that
\begin{align}
\mathbb{E}  | B_{t_{k-1}^{n}, t_{k}^{n}}^{H, \lambda; m, 1} |^p
&\leq c_{p}
[\mathbb{E} | B_{t_{k-1}^{n}, t_{k}^{n}}^{H, \lambda; m, 1} |^2 ]^{\frac{p}{2}}
= c_{p}\bigg[C^{2}_{\frac{1}{2^n}}
\frac{1}{2^{2nH}}\bigg]^{\frac{p}{2}}.\nonumber
\end{align}
Moreover, note that
\begin{align}
B_{t_{k-1}^{n}, t_{k}^{n}}^{H, \lambda; m+1, 1} - B_{t_{k-1}^{n}, t_{k}^{n}}^{H, \lambda; m, 1} = (B^{H, \lambda}_{t_{k}^{n}} - B^{H, \lambda}_{t_{k-1}^{n}}) - (B^{H, \lambda}_{t_{k}^{n}} - B^{H, \lambda}_{t_{k-1}^{n}}) = 0.\nonumber
\end{align}

For $n > m$, it holds that
\begin{align}\label{eq2.2-0}
B_{t_{k-1}^{n}, t_{k}^{n}}^{H, \lambda; m, 1} = \frac{2^m}{2^n} ( B_{t_{l}^{m}}^{H, \lambda} - B_{t_{l-1}^{m}}^{H, \lambda} ),
\end{align}
where $l$ is the unique integer such that $ [t_{k-1}^{n}, t_{k}^{n}] \subset [t_{l-1}^{m}, t_{l}^{m}] $. Hence
\begin{align}
[\mathbb{E}  | B_{t_{k-1}^{n}, t_{k}^{n}}^{H, \lambda; m, 1} |^{p}]^{\frac{1}{p}}
&= \frac{2^m}{2^n} [\mathbb{E}  |  B_{t_{l}^{m}}^{H, \lambda} - B_{t_{l-1}^{m}}^{H, \lambda}|^{p}]^{\frac{1}{p}}\nonumber\\
&
\leq c_{p}
\frac{2^m}{2^n} [\mathbb{E}  |  B_{t_{l}^{m}}^{H, \lambda} - B_{t_{l-1}^{m}}^{H, \lambda} |^{2}]^{\frac{1}{2}}
=c_{p} \frac{2^{m}}{2^n}
\bigg[C^{2}_{\frac{1}{2^m}}
\frac{1}{2^{2mH}}\bigg]^{\frac{1}{2}}.\nonumber
\end{align}

For $n>m$, let $l$ be the unique integer such that $ [t_{k-1}^{n}, t_{k}^{n}] \subset [t_{l-1}^{m}, t_{l}^{m}] $. Then the interval $[t^n_{k-1}, t^n_k]$ must lie in  either $[t^{m+1}_{2l-2}, t^{m+1}_{2l-1}]$ or $[t^{m+1}_{2l-1}, t^{m+1}_{2l}]$.
In the case $[t^n_{k-1}, t^n_k] \subset [t^{m+1}_{2l-2}, t^{m+1}_{2l-1}]$, it follows from \eqref{eq2.2-0} that
\begin{align}
B^{H, \lambda;m+1,1}_{t^n_{k-1}, t^n_k} - B^{H, \lambda;m,1}_{t^n_{k-1}, t^n_k}
&= \frac{2^m}{2^n}[(B^{H, \lambda}_{t^{m+1}_{2l-1}} - B^{H, \lambda}_{t^{m+1}_{2l-2}}) -
(B^{H, \lambda}_{t^{m+1}_{2l}} - B^{H, \lambda}_{t^{m+1}_{2l-1}})],\nonumber
\end{align}
and consequently
\begin{align}
&\mathbb{E}
|B^{H, \lambda;m+1,1}_{t^n_{k-1}, t^n_k} - B^{H, \lambda;m,1}_{t^n_{k-1}, t^n_k}|^{p}\nonumber\\
&\leq c_{p}\bigg(\frac{2^m}{2^n}\bigg)^{p}
[\mathbb{E}|
(B^{H, \lambda}_{t^{m+1}_{2l-1}} - B^{H, \lambda}_{t^{m+1}_{2l-2}}) -
(B^{H, \lambda}_{t^{m+1}_{2l}} - B^{H, \lambda}_{t^{m+1}_{2l-1}})
|^{2}]^{\frac{p}{2}}
\nonumber\\
&\leq c_{p}\bigg(\frac{2^m}{2^n}\bigg)^{p}
\bigg[2
C^{2}_{\frac{1}{2^{m+1}}}
\frac{1}{2^{2(m+1)H}}\bigg]^{\frac{p}{2}}\leq c_{p}\bigg(\frac{2^m}{2^n}\bigg)^{p}
\bigg[
C^{2}_{\frac{1}{2^{m}}}
\frac{1}{2^{2mH}}\bigg]^{\frac{p}{2}},
\nonumber
\end{align}
where in the last inequality we have used the monotonically increasing property of $C^{2}_{t}t^{2H}$ with respect to $t$. For the case $[t^n_{k-1}, t^n_k] \subset [t^{m+1}_{2l-1}, t^{m+1}_{2l}]$,
an analogous argument yields the same estimate.
\end{proof}

\begin{lemma}
For any $ m, n \geq 1$, $ k = 1, 2, \cdots, 2^n $, $H\in (\frac{1}{4}, \frac{1}{2}]$, $\lambda>0$ and $p\in (2,4)$, the following estimate holds at the second level paths:
\[
[\mathbb{E}
|B_{t_{k-1}^{n}, t_{k}^{n}}^{H,\lambda;m, 2}
|^{\frac{p}{2}}]^{\frac{2}{p}}
\leq
\begin{cases}
\frac{c}{2^{2Hn}}
, & n \leq m; \\
cC^{2}_{\frac{1}{2^m}}
\frac{2^{2m(1-H)}}{2^{2n}}, & n > m,
\end{cases}
\]
\[
[\mathbb{E}
|B_{t_{k-1}^{n}, t_{k}^{n}}^{H, \lambda; m+1, 2} - B_{t_{k-1}^{n}, t_{k}^{n}}^{H, \lambda; m, 2}|^{\frac{p}{2}}]^{\frac{2}{p}} \leq
\begin{cases}
\frac{c}{2^{\frac{(4H-1)m}{ 2}}2^{\frac{n}{2}}}, & n \leq m; \\
cC^{2}_{\frac{1}{2^m}}
\frac{2^{2m(1-H)}}{2^{2n}}, & n > m,
\end{cases}
\]
where $c$ is some positive constant not depending on $m, n, k$.
\end{lemma}

\begin{proof}

The condition $p \in (2, 4)$ implies $\frac{p}{2} < 2$.
Hence, by the H\"{o}lder inequality, in order to obtain the second-level estimate in the $L^{\frac{p}{2}}$-norm,  it suffices to establish it in the $L^2$-norm.

For $ n > m $, the following holds:
\begin{align}\label{eq2.4-1200}
B_{t_{k-1}^n, t_k^n}^{H, \lambda; m,2;i,j}
&
= \int_{t_{k-1}^n}^{t_k^n} B_{t_{k-1}^n,v}^{H, \lambda; m,1;i} dB_{v}^{H, \lambda; m;j}\nonumber\\
&
= \frac{\Delta_l^m B^{H, \lambda;i} \Delta_l^m B^{H, \lambda;j}}{(\Delta t^m)^2} \int_{t_{k-1}^n}^{t_k^n} (v - t_{k-1}^n) dv\nonumber\\
&
=  2^{2(m-n)-1} \Delta_l^m B^{H, \lambda;i} \Delta_l^m B^{H, \lambda;j},
\end{align}
where $l$ is the unique integer such that $[t_{k-1}^n, t_k^n] \subset [t_{l-1}^m, t_l^m]$.
Then applying the notation of tensor products yields
\begin{align}\label{eq2.4}
B_{t_{k-1}^n, t_k^n}^{H, \lambda;m,2} = 2^{2(m-n)-1} (\Delta_l^m B^{H, \lambda})^{\otimes 2}.
\end{align}
It follows that
\begin{align}
&B_{t_{k-1}^n, t_k^n}^{H, \lambda;m+1,2}
- B_{t_{k-1}^n, t_k^n}^{H, \lambda;m,2} \nonumber\\
&=
\begin{cases}
2^{2(m-n)+1}(\Delta_{2l-1}^{m+1}B^{H, \lambda})^{\otimes 2} - 2^{2(m-n)-1}(\Delta_l^m B^{H, \lambda})^{\otimes 2},
& [t_{k-1}^n, t_k^n] \subset [t_{2l-2}^{m+1}, t_{2l-1}^{m+1}]; \\
2^{2(m-n)+1}(\Delta_{2l}^{m+1}B^{H, \lambda})^{\otimes 2} - 2^{2(m-n)-1}(\Delta_l^m B^{H, \lambda})^{\otimes 2}, &
[t_{k-1}^n, t_k^n] \subset [t_{2l-1}^{m+1}, t_{2l}^{m+1}]. \nonumber
\end{cases}
\end{align}
Since the components of the TFBM are independent and identically distributed,
we deduce by the Minkowski inequality that for the case  $[t_{k-1}^n, t_k^n] \subset
[t_{2l-2}^{m+1}, t_{2l-1}^{m+1}]$,
\begin{align}\label{eq2.5}
&\bigg[\mathbb{E} \bigg(
B_{t_{k-1}^n, t_k^n}^{H, \lambda;m+1,2}-
B_{t_{k-1}^n, t_k^n}^{H, \lambda;m,2}\bigg)^{2}\bigg]^{\frac{1}{2}}
\nonumber\\
&= 2^{2(m-n)}
\bigg[\mathbb{E} \bigg(
2(\Delta_{2l-1}^{m+1}B^{H, \lambda})^{\otimes 2} -
\frac{1}{2}(\Delta_l^m B^{H, \lambda})^{\otimes 2}
\bigg)^{2}\bigg]^{\frac{1}{2}}
\nonumber\\
&\leq   2^{2(m-n)}
\sum_{i,j=1}^{d}
\bigg[\mathbb{E}\bigg(
2\Delta_{2l-1}^{m+1}B^{H, \lambda;i}
\Delta_{2l-1}^{m+1}B^{H, \lambda;j} -
\frac{1}{2}\Delta_l^m B^{H, \lambda;i}
\Delta_l^m B^{H, \lambda;j}
\bigg)^{2}\bigg]^{\frac{1}{2}}\nonumber\\
&\leq   2^{2(m-n)}
\bigg[ \sum_{i,j=1}^{d}2
\mathbb{E}(
\Delta_{2l-1}^{m+1}B^{H, \lambda;i}
)^{2}
+
\sum_{i,j=1}^{d}\frac{1}{2}
\mathbb{E}(
\Delta_l^m B^{H, \lambda;i})^{2}\bigg]
\nonumber\\
&\leq c 2^{2(m-n)}
 C^{2}_{\frac{1}{2^{m}}}
 \bigg(\frac{1}{2^{m}}\bigg)^{2H}
 = c C^{2}_{\frac{1}{2^{m}}}
\frac{2^{2m(1-H)}}{2^{2n}}.
\end{align}
Similarly, for the case $[t_{k-1}^n, t_k^n] \subset [t_{2l-1}^{m+1},
 t_{2l}^{m+1}]$, we arrive at the identical estimate.

For $n \leq m$, we have
\begin{align}
B_{t_{k-1}^{n}, t_{k}^{n}}^{H, \lambda; m, 2; i, j}
&= \int_{t_{k-1}^{n}}^{t_{k}^{n}}
B_{t_{k-1}^{n}, v}^{H, \lambda;m, 1; i} dB^{H, \lambda;m; j}_v
\nonumber\\
&= \sum_{l=2^{(m-n)}(k-1)+1}^{2^{(m-n)}k}
\frac{\Delta_l^m B^{H, \lambda;j}}{\Delta t^m} \int_{t_{l-1}^m}^{t_l^m} \bigg( \frac{v - t_{l-1}^m}{\Delta t^m} B_{t_{l}^{m}}^{H, \lambda;i} + \frac{t_l^m - v}{\Delta t^m} B_{t_{l-1}^m}^{H, \lambda;i} - B_{t_{k-1}^n}^{H, \lambda;i}\bigg) dv\nonumber\\
&= \sum_{l=2^{(m-n)}(k-1)+1}^{2^{(m-n)}k} \bigg(
\frac{B_{t_{l-1}^{m}}^{H, \lambda;i} + B_{t_l^m}^{H, \lambda;i}}{2} - B_{t_{k-1}^n}^{H, \lambda;i} \bigg) \Delta_l^m B^{H, \lambda;j}.
\nonumber
\end{align}
Hence
\begin{align}
&B_{t_{k-1}^{n}, t_{k}^{n}}^{H, \lambda;m+1,2;i,j} - B_{t_{k-1}^{n}, t_{k}^{n}}^{H, \lambda;m,2;i,j}\nonumber\\
&= \sum_{l=2^{(m+1-n)}(k-1)+1}^{2^{(m+1-n)}k} \bigg(
\frac{B_{t_{l-1}^{m+1}}^{H, \lambda;i} + B_{t_l^{m+1}}^{H, \lambda;i}}{2} - B_{t_{k-1}^n}^{H, \lambda;i} \bigg) \Delta_l^{m+1} B^{H, \lambda;j}
\nonumber\\
&~~~~- \sum_{l=2^{(m-n)}(k-1)+1}^{2^{(m-n)}k} \bigg(
\frac{B_{t_{l-1}^{m}}^{H, \lambda;i} + B_{t_l^m}^{H, \lambda;i}}{2} - B_{t_{k-1}^n}^{H, \lambda;i} \bigg) \Delta_l^m B^{H, \lambda;j}\nonumber\\
&= \sum_{l=2^{(m-n)}(k-1)+1}^{2^{(m-n)}k}
\bigg[\bigg(
\frac{B_{t_{2l-2}^{m+1}}^{H, \lambda;i} + B_{t_{2l-1}^{m+1}}^{H, \lambda;i}}{2} - B_{t_{k-1}^n}^{H, \lambda;i} \bigg) \Delta_{2l-1}^{m+1} B^{H, \lambda;j}\nonumber\\
&~~~~+
\bigg(
\frac{B_{t_{2l-1}^{m+1}}^{H, \lambda;i} + B_{t_{2l}^{m+1}}^{H, \lambda;i}}{2} - B_{t_{k-1}^n}^{H, \lambda;i} \bigg) \Delta_{2l}^{m+1} B^{H, \lambda;j}\nonumber\\
&~~~~
-
\bigg(
\frac{B_{t_{2l-2}^{m+1}}^{H, \lambda;i} + B_{t_{2l}^{m+1}}^{H, \lambda;i}}{2} - B_{t_{k-1}^n}^{H, \lambda;i} \bigg)
(\Delta_{2l-1}^{m+1} B^{H, \lambda;j}+
\Delta_{2l}^{m+1} B^{H, \lambda;j})\bigg]
\nonumber\\
&= \frac{1}{2}
\sum_{l=2^{(m-n)}(k-1)+1}^{2^{(m-n)}k}
\bigg[
\Delta_{2l-1}^{m+1}B^{H,\lambda;i}
\Delta_{2l}^{m+1}B^{H,\lambda;j} - \Delta_{2l}^{m+1}B^{H,\lambda;i}
\Delta_{2l-1}^{m+1}B^{H,\lambda;j}\bigg].
\nonumber
\end{align}
By using the notation of tensor products, we obtain that
\begin{align}
&B_{t_{k-1}^{n}, t_{k}^{n}}^{H, \lambda;m+1,2} - B_{t_{k-1}^{n}, t_{k}^{n}}^{H, \lambda;m,2}\nonumber\\
&= \frac{1}{2}
\sum_{l=2^{(m-n)}(k-1)+1}^{2^{(m-n)}k}
\bigg[
\Delta_{2l-1}^{m+1}B^{H,\lambda}\otimes
\Delta_{2l}^{m+1}B^{H,\lambda} - \Delta_{2l}^{m+1}B^{H,\lambda}
\otimes\Delta_{2l-1}^{m+1}B^{H,\lambda}\bigg],
\nonumber
\end{align}
and thus
\begin{align}\label{eq2.2}
&\mathbb{E}\bigg(B_{t_{k-1}^{n}, t_{k}^{n}}^{H, \lambda;m+1,2} - B_{t_{k-1}^{n}, t_{k}^{n}}^{H, \lambda;m,2}\bigg)^{2}\nonumber\\
&\leq c\sum_{\substack{
    i\neq j \\
    i,j=1,\ldots, d
}}\sum_{r,l=2^{(m-n)}(k-1)+1}^{2^{(m-n)}k}
\mathbb{E}\bigg[
(
\Delta_{2l-1}^{m+1}B^{H,\lambda;i}
\Delta_{2l}^{m+1}B^{H,\lambda;j} - \Delta_{2l}^{m+1}B^{H,\lambda;i}
\Delta_{2l-1}^{m+1}B^{H,\lambda;j})\nonumber\\
&~~\times
(
\Delta_{2r-1}^{m+1}B^{H,\lambda;i}
\Delta_{2r}^{m+1}B^{H,\lambda;j} - \Delta_{2r}^{m+1}B^{H,\lambda;i}
\Delta_{2r-1}^{m+1}B^{H,\lambda;j})
\bigg]\nonumber\\
&=c\sum_{\substack{
    i\neq j \\
    i,j=1,\ldots, d
}}\sum_{r,l=2^{(m-n)}(k-1)+1}^{2^{(m-n)}k}
\bigg[[
\mathbb{E}
(
\Delta_{2l-1}^{m+1}B^{H,\lambda;i}
\Delta_{2r-1}^{m+1}B^{H,\lambda;i})]^{2}
\nonumber\\
&~~-
\mathbb{E}(
\Delta_{2l-1}^{m+1}B^{H,\lambda;j}
\Delta_{2r}^{m+1}B^{H,\lambda;j})
\mathbb{E}(
\Delta_{2l}^{m+1}B^{H,\lambda;i}
\Delta_{2r-1}^{m+1}B^{H,\lambda;i})\bigg]\nonumber\\
&=c\sum_{\substack{
    i\neq j \\
    i,j=1,\ldots, d
}}
\sum_{
\substack{
    l>r \\
r,l=2^{(m-n)}(k-1)+1}
}^{2^{(m-n)}k}
\bigg[[
\mathbb{E}
(
\Delta_{2l-1}^{m+1}B^{H,\lambda;i}
\Delta_{2r-1}^{m+1}B^{H,\lambda;i})]^{2}
\nonumber\\
&~~-
\mathbb{E}(
\Delta_{2l-1}^{m+1}B^{H,\lambda;j}
\Delta_{2r}^{m+1}B^{H,\lambda;j})
\mathbb{E}(
\Delta_{2l}^{m+1}B^{H,\lambda;i}
\Delta_{2r-1}^{m+1}B^{H,\lambda;i})\bigg]\nonumber\\
&~~+c\sum_{\substack{
    i\neq j \\
    i,j=1,\ldots, d
}}
\sum_{
l=2^{(m-n)}(k-1)+1}^{2^{(m-n)}k}
\bigg[[
\mathbb{E}
(
\Delta_{2l-1}^{m+1}B^{H,\lambda;i}
\Delta_{2l-1}^{m+1}B^{H,\lambda;i})]^{2}
\nonumber\\
&~~-
\mathbb{E}(
\Delta_{2l-1}^{m+1}B^{H,\lambda;j}
\Delta_{2l}^{m+1}B^{H,\lambda;j})
\mathbb{E}(
\Delta_{2l}^{m+1}B^{H,\lambda;i}
\Delta_{2l-1}^{m+1}B^{H,\lambda;i})\bigg]
:=I_{1}+I_{2},
\end{align}
where in the first equality we have used the independent and identically distributed property of the components of the TFBM.

According to Lemma A.1, we find that
\begin{align}\label{eq2.2-1}
I_{1}&\leq \sum_{\substack{
    i\neq j \\
    i,j=1,\ldots, d
}}
\sum_{
\substack{
    l>r \\
r,l=2^{(m-n)}(k-1)+1}
}^{2^{(m-n)}k}
\frac{c_{H, \lambda}}{2^{(4H+2)m}}
\leq
\frac{c_{H, \lambda}}{2^{4Hm}2^{2n}}.
\end{align}
By making use of the H\"{o}lder inequality, we have
\begin{align}\label{eq2.2-2}
I_{2}
&\leq \sum_{\substack{
    i\neq j \\
    i,j=1,\ldots, d
}}
\sum_{
l=2^{(m-n)}(k-1)+1}
^{2^{(m-n)}k}
\bigg[[
\mathbb{E}
(
\Delta_{2l-1}^{m+1}B^{H,\lambda;i})^{2}]^{2}
\nonumber\\
&~~~~+\mathbb{E}
(
\Delta_{2l-1}^{m+1}B^{H,\lambda;j})^{2}
\mathbb{E}
(
\Delta_{2l}^{m+1}B^{H,\lambda;i})^{2}\bigg]\nonumber\\
&\leq c
\sum_{
l=2^{(m-n)}(k-1)+1}^{2^{(m-n)}k}
\bigg[
C^{2}_{\frac{1}{2^{m+1}}}
\frac{1}{2^{2H(m+1)}}\bigg]^{2}\nonumber\\
&\leq c
\bigg[
C^{2}_{\frac{1}{2^{m}}}\bigg]^{2}
\frac{1}{2^{(4H-1)m}2^{n}},
\end{align}
where in the last inequality we have used the monotonically
increasing property of $C^{2}_{t}t^{2H}$ with respect to $t$. Inserting \eqref{eq2.2-1}-\eqref{eq2.2-2} into \eqref{eq2.2} results in
\begin{align}
&\mathbb{E}\bigg(B_{t_{k-1}^{n}, t_{k}^{n}}^{H, \lambda;m+1,2} - B_{t_{k-1}^{n}, t_{k}^{n}}^{H, \lambda;m,2}\bigg)^{2}
\leq
\frac{c_{H, \lambda}}{2^{(4H-1)m}2^{n}}.\nonumber
\end{align}

Now we consider the term $B_{t_{k-1}^{n}, t_{k}^{n}}^{H, \lambda;m, 2}$.

For $n \geq m$, in view of \eqref{eq2.4},
by using some similar arguments as in \eqref{eq2.5},
we have
\begin{align}
[\mathbb{E}|B_{t_{k-1}^{n},
t_{k}^{n}}^{H, \lambda;m, 2}|^{2}]^{\frac{1}{2}} \leq c C^{2}_{\frac{1}{2^{m}}}
\frac{2^{2m(1-H)}}{2^{2n}}.\nonumber
\end{align}
For $n < m$, it is clear that
\begin{align}
B_{t_{k-1}^{n}, t_{k}^{n}}^{H, \lambda;m, 2}
= \sum_{l=n+1}^{m}\bigg(B_{t_{k-1}^{n},t_{k}^{n}}^{H, \lambda;l, 2} - B_{t_{k-1}^{n},t_{k}^{n}}^{H, \lambda;l-1, 2}\bigg) + B_{t_{k-1}^{n},t_{k}^{n}}^{H, \lambda;n, 2}. \nonumber
\end{align}
Subsequently, application of the Minkowski inequality yields that
\begin{align}
[\mathbb{E}|B_{t_{k-1}^{n},
t_{k}^{n}}^{H, \lambda;m, 2}|^{2}]^{\frac{1}{2}}
&\leq \sum_{l=n+1}^m
[\mathbb{E}|B_{t_{k-1}^{n},t_{k}^{n}}^{H, \lambda;l,2} - B_{t_{k-1}^{n},t_{k}^{n}}^{H, \lambda;l-1,2}|^{2}]^{\frac{1}{2}}  + [\mathbb{E}|B_{t_{k-1}^{n},t_{k}^{n}}^{H, \lambda;n,2}|^{2}]^{\frac{1}{2}} \nonumber\\
&
\leq
\sum_{l=n+1}^m
\frac{c_{H,\lambda}}{2^{\frac{(4H-1)l}{ 2}}2^{\frac{n}{2}}}
+cC^{2}_{\frac{1}{2^{n}}}
\frac{1}{2^{2H n}}
\leq \frac{c_{H,\lambda}}{2^{2nH}}.
\end{align}
We complete the proof.
\end{proof}

\begin{lemma}
For any $ m, n \geq 1$, $ k = 1, 2, \cdots, 2^n $, $H\in (\frac{1}{4}, \frac{1}{2}]$, $\lambda>0$ and $p\in (2,4)$, the following estimate holds at the third level paths:
\[
[\mathbb{E}|B_{t_{k-1}^{n}, t_{k}^{n}}^{H,\lambda;m, 3}|^{\frac{p}{3}}]^{\frac{3}{p}} \leq
\begin{cases}
\frac{ c}{2^{3H n}}
, & n < m; \\
c [C^{2}_{\frac{1}{2^{m}}}]^{\frac{3}{2}}
\frac{2^{3m(1-H)}}{2^{3n}}, & n \geq m,
\end{cases}
\]
\[
[\mathbb{E}|B_{t_{k-1}^{n}, t_{k}^{n}}^{H, \lambda; m+1, 3} - B_{t_{k-1}^{n}, t_{k}^{n}}^{H, \lambda; m, 3}|^{\frac{p}{3}}]^{\frac{3}{p}} \leq
\begin{cases}
\frac{c}
{2^{\frac{(6H-1)m}{ 2}}2^{\frac{n}{2}}}, & n < m; \\
c
[C^{2}_{\frac{1}{2^{m}}}]^{\frac{3}{2}}
\frac{2^{3m(1-H)}}{2^{3n}}  , & n \geq m,
\end{cases}
\]
where $c$ is some positive constant not depending on $m, n, k$.
\end{lemma}
\begin{proof}
With $p \in (2, 4)$ in mind, to derive the $L^{\frac{p}{3}}$-norm of the third-level estimate,
it suffices to establish its corresponding $L^2$-norm.

In a similar way as in \eqref{eq2.4-1200},
we derive that for $n\geq m$,
\begin{align}
B_{t_{k-1}^{n}, t_{k}^{n}}^{H, \lambda; m, 3; i, j,\xi}
&=\int_{t_{k-1}^{n}}^{t_{k}^{n}}
\int_{t_{k-1}^{n}}^{r}
B_{t_{k-1}^{n},v}^{H, \lambda; m, i}
dB_{v}^{H, \lambda; m, j}
dB_{r}^{H, \lambda; m, \xi}\nonumber\\
&=\frac{\Delta_{l}^{m} B^{H, \lambda;i}
\Delta_{l}^{m} B^{H, \lambda;j}
\Delta_{l}^{m} B^{H, \lambda;\xi}
}{(\Delta t^{m})^{3}}
\int_{t_{k-1}^{n}}^{t_{k}^{n}}
\int_{t_{k-1}^{n}}^{r}
(v-t_{k-1}^{n})
dvdr\nonumber\\
&=\frac{1}{3}
2^{3(m-n)-1}
\Delta_{l}^{m} B^{H, \lambda;i}
\Delta_{l}^{m} B^{H, \lambda;j}
\Delta_{l}^{m} B^{H, \lambda;\xi},\nonumber
\end{align}
where $l$ is the unique integer such that $[t_{k-1}^n, t_k^n] \subset [t_{l-1}^m, t_l^m]$.
It follows that
\begin{align}
&B_{t_{k-1}^{n}, t_{k}^{n}}^{H, \lambda; m, 3}
=\frac{1}{3}
2^{3(m-n)-1}
(\Delta_{l}^{m} B^{H, \lambda})^{\otimes 3}.\nonumber
\end{align}
As a result,
\begin{align}
&B_{t_{k-1}^n, t_k^n}^{H, \lambda;m+1,3}
- B_{t_{k-1}^n, t_k^n}^{H, \lambda;m,3} \nonumber\\
&=
\begin{cases}
\frac{1}{3}
2^{3(m+1-n)-1}(\Delta_{2l-1}^{m+1}B^{H, \lambda})^{\otimes 3} -
\frac{1}{3}
2^{3(m-n)-1}(\Delta_l^m B^{H, \lambda})^{\otimes 3},
& [t_{k-1}^n, t_k^n] \subset [t_{2l-2}^{m+1}, t_{2l-1}^{m+1}]; \\
\frac{1}{3}
2^{3(m+1-n)-1}(\Delta_{2l}^{m+1}B^{H, \lambda})^{\otimes 3} -
\frac{1}{3}
2^{3(m-n)-1}(\Delta_l^m B^{H, \lambda})^{\otimes 3}, &
[t_{k-1}^n, t_k^n] \subset [t_{2l-1}^{m+1}, t_{2l}^{m+1}]. \nonumber
\end{cases}
\end{align}
In the case
$[t_{k-1}^n, t_k^n] \subset [t_{2l-2}^{m+1}, t_{2l-1}^{m+1}]$,
based on the fact that the components of the TFBM are independent and identically distributed, we have
\begin{align}
&
\mathbb{E}\bigg(B_{t_{k-1}^n, t_k^n}^{H, \lambda;m+1,3}
- B_{t_{k-1}^n, t_k^n}^{H, \lambda;m,3}\bigg)^{2} \nonumber\\
&=\frac{2^{6(m-n)-2}}{9}
\sum_{i,j,\xi}^{d}
\mathbb{E}\bigg(2^{3}
\Delta_{2l-1}^{m+1}B^{H, \lambda;i}
\Delta_{2l-1}^{m+1}B^{H, \lambda;j}
\Delta_{2l-1}^{m+1}B^{H, \lambda;\xi}
\nonumber\\
&~~~~-
\Delta_{l}^{m}B^{H, \lambda;i}
\Delta_{l}^{m}B^{H, \lambda;j}
\Delta_{l}^{m}B^{H, \lambda;\xi}
\bigg)^{2} \nonumber\\
&\leq c
2^{6(m-n)-2}
\sum_{i,j,\xi}^{d}
\bigg[
\mathbb{E}(\Delta_{2l-1}^{m+1}B^{H, \lambda;i})^{2}
\mathbb{E}(\Delta_{2l-1}^{m+1}B^{H, \lambda;j})^{2}
\mathbb{E}(\Delta_{2l-1}^{m+1}B^{H, \lambda;\xi})^{2}
\nonumber\\
&~~~~+
\mathbb{E}(\Delta_{l}^{m}B^{H, \lambda;i})^{2}
\mathbb{E}(\Delta_{l}^{m}B^{H, \lambda;j})^{2}
\mathbb{E}(\Delta_{l}^{m}B^{H, \lambda;\xi})^{2}\bigg]\nonumber\\
&\leq c
[C^{2}_{\frac{1}{2^{m}}}]^{3}
\frac{2^{6m(1-H)}}{2^{6n}}.\nonumber
\end{align}
For the case $[t_{k-1}^n, t_k^n] \subset [t_{2l-1}^{m+1},t_{2l}^{m+1}]$, an analogous approach can be employed to arrive at the same result.

For $n < m$, according to \cite[Lemma 11]{Coutin-Qian} that
\begin{align}\label{eq2.6-1}
&B_{t_{k-1}^{n}, t_{k}^{n}}^{H, \lambda; m+1, 3}
-B_{t_{k-1}^{n}, t_{k}^{n}}^{H, \lambda; m, 3}
\nonumber\\
&=\frac{1}{2}
\sum_l\bigl(B_{t_{2l-1}^{m+1}}^{H, \lambda}-
B_{t_{k-1}^n}^{H, \lambda}\bigr)
\otimes \bigl(\triangle_{2l-1}^{m+1} B^{H, \lambda}
\otimes \triangle_{2l}^{m+1}B^{H, \lambda}
- \triangle_{2l}^{m+1} B^{H, \lambda}
\otimes \triangle_{2l-1}^{m+1}B^{H, \lambda}
\bigr) \nonumber\\
&+\frac{1}{2}
\sum_l \bigl(\triangle_{2r-1}^{m+1} B^{H, \lambda}
\otimes \triangle_{2r}^{m+1} B^{H, \lambda}
- \triangle_{2r}^{m+1} B^{H, \lambda}
\otimes \triangle_{2r-1}^{m+1} B^{H, \lambda}
\bigr) \otimes \bigl(
B^{H, \lambda}_{t_k^n}-B^{H, \lambda}
_{t_{2l+1}^{m+1}}\bigr) \nonumber\\
&+\frac{1}{3}
\sum_l \triangle_{2l-1}^{m+1}
B^{H, \lambda}
\otimes \bigl(\triangle_{2l}^{m+1}
B^{H, \lambda} \otimes \triangle_{2l}^{m+1}
B^{H, \lambda} + \triangle_{2l-1}^{m+1} B^{H, \lambda} \otimes \triangle_{2l}^{m+1} B^{H, \lambda}\bigr) \nonumber\\
&-
\frac{1}{6} \sum_l \triangle_{2l}^{m+1} B^{H, \lambda}
\otimes \bigl(\triangle_{2l}^{m+1}B^{H, \lambda}
\otimes \triangle_{2l-1}^{m+1} B^{H, \lambda}
+ \triangle_{2l-1}^{m+1} B^{H, \lambda}
\otimes \triangle_{2l}^{m+1} B^{H, \lambda}\bigr) \nonumber\\
&-\frac{1}{6} \sum_l \bigl(\triangle_{2l-1}^{m+1}
B^{H, \lambda}
\otimes \triangle_{2l}^{m+1} B^{H, \lambda} + \triangle_{2l}^{m+1} B^{H, \lambda} \otimes \triangle_{2l-1}^{m+1} B^{H, \lambda}\bigr) \otimes \triangle_{2l-1}^{m+1} B^{H, \lambda},
\end{align}
where all the summations run over \(2^{m-n}(k-1)+1 \leq l \leq 2^{m-n}k\).
We first estimate the term
\begin{align}\label{eq2.6-2}
I_3 &\equiv \bigg| \sum_{l=2 + 2^{m-n}(k-1)}^{2^{m-n}k}
 \bigg\{ \bigl(B^{H, \lambda}_{t_{2l-2}^{m+1}} - B^{H, \lambda}_{t_{k-1}^n}\bigr)\nonumber\\
&\otimes \bigl(\triangle_{2l-1}^{m+1} B^{H, \lambda} \otimes \triangle_{2l}^{m+1} B^{H, \lambda} - \triangle_{2l}^{m+1} B^{H, \lambda} \otimes \triangle_{2l-1}^{m+1} B^{H, \lambda}\bigr) \bigg\} \bigg|^2,
\end{align}
Set \(A_l = B^{H, \lambda}_{t_{2l-1}^{m+1}} - B^{H, \lambda}_{t_{k-1}^n}\) and
\(A_l^i = B^{H, \lambda;i}_{t_{2l-1}^{m+1}} - B^{H, \lambda;i}_{t_{k-1}^n}\). Then
\begin{align}\label{eq2.6-3}
I_3
&= \sum_{i\neq j,\xi}
\bigg(
\sum_{l=2 + 2^{m-n}(k-1)}^{2^{m-n}k}
A_l^\xi\bigg(
\triangle_{2l-1}^{m+1}
B^{H, \lambda;i}
\triangle_{2l}^{m+1}B^{H, \lambda;j}
- \triangle_{2l}^{m+1}
B^{H, \lambda;i}
\triangle_{2l-1}^{m+1}
B^{H, \lambda;j}
\bigg) \bigg)^2 \nonumber\\
&= \sum_{i\neq j,\xi}
\sum_{l=2 + 2^{m-n}(k-1)}^{2^{m-n}k}
(A_l^\xi)^2
\bigg(\triangle_{2l-1}^{m+1}
B^{H, \lambda;i}
\triangle_{2l}^{m+1} B^{H, \lambda;j}
- \triangle_{2l}^{m+1} B^{H, \lambda;i}
\triangle_{2l-1}^{m+1} B^{H, \lambda;i}\bigg)^2
\nonumber\\
&~~~~+ 2 \sum_{i\neq j, \xi}
\sum_{r<l}
\bigg\{ A_r^\xi A_l^\xi
\triangle_{2r-1}^{m+1}
B^{H, \lambda;i}
\triangle_{2r}^{m+1} B^{H, \lambda;j}
\triangle_{2l-1}^{m+1} B^{H, \lambda;i}
\triangle_{2l}^{m+1} B^{H, \lambda;j}
\nonumber\\
&~~~~- A_r^\xi A_l^\xi
\triangle_{2l}^{m+1}
B^{H, \lambda;i}
\triangle_{2l-1}^{m+1} B^{H, \lambda;j}
\triangle_{2r-1}^{m+1}
B^{H, \lambda;i}
\triangle_{2r}^{m+1} B^{H, \lambda;j}
\nonumber\\
&~~~~- A_r^\xi A_l^\xi
\triangle_{2l-1}^{m+1} B^{H, \lambda;i} \triangle_{2l}^{m+1} B^{H, \lambda;j} \triangle_{2r}^{m+1} B^{H, \lambda;i} \triangle_{2r-1}^{m+1} B^{H, \lambda;j}
\nonumber\\
&~~~~ + A_r^\xi A_l^\xi \triangle_{2l}^{m+1}
B^{H, \lambda;i} \triangle_{2l-1}^{m+1} B^{H, \lambda;j} \triangle_{2r}^{m+1} B^{H, \lambda;i} \triangle_{2r-1}^{m+1} B^{H, \lambda;j} \bigg\}.
\end{align}
For \(i \neq j\), using the Cauchy-Schwarz inequality,
we obtain that
\begin{align}\label{eq2.6-4}
&\bigg|\mathbb{E}\bigg(A_r^\xi A_l^\xi
\triangle_{2l-1}^{m+1}B^{H, \lambda;i}
\triangle_{2r-1}^{m+1} B^{H, \lambda;i}
\triangle_{2l}^{m+1} B^{H, \lambda;j} \triangle_{2r}^{m+1} B^{H, \lambda;j}\bigg)\bigg| \nonumber\\
&= [\mathbb{E}(A_r^\xi)^{2}]^{\frac{1}{2}}
[\mathbb{E}(A_l^\xi)^{2}]^{\frac{1}{2}}
\bigg|\mathbb{E}\bigg(
\triangle_{2l-1}^{m+1} B^{H, \lambda;i}
\triangle_{2r-1}^{m+1} B^{H, \lambda;i}\bigg)\bigg| \bigg|\mathbb{E}\bigg(\triangle_{2l}^{m+1}
B^{H, \lambda;j} \triangle_{2r}^{m+1}
B^{H, \lambda;j}\bigr)\bigg| \nonumber\\
&\leq c_{H,\lambda}\bigg(\frac{2l - 2 - 2^{m-n}(k-1)}{2^{m+1}}\bigg)^H \bigg(\frac{2r - 2 - 2^{m-n}(k-1)}{2^{m+1}}\bigg)^H \nonumber\\
&\times\bigg(\frac{2l-2r}{2^{m+1}}\bigg)^{4H-4}
\bigg(\frac{1}{2^{m+1}}\bigg)^{4}\nonumber\\
&\leq c_{H,\lambda} \bigl(2^{m-n}\bigr)^{2H} \left(\frac{1}{2^m}\right)^{6H} \frac{1}{(r-l)^{4-4H}}.
\end{align}
Hence
\begin{align}
\mathbb{E} I_3
&\leq c_{H,\lambda} 2^{m-n} \left(\frac{1}{2^m}\right)^{4H} + c_{H,\lambda} \bigl(2^{m-n}\bigr)^{1+2H} \left(\frac{1}{2^m}\right)^{6H} \nonumber\\
&\leq c_{H,\lambda} \bigl(2^{m-n}\bigr)^{1+2H} \left(\frac{1}{2^m}\right)^{6H}.\nonumber
\end{align}
Similarly, we may estimate the second summation appeared in the right side of \eqref{eq2.6-1} and eventually we get the same upper bound as the above.

Next we estimate the remaining terms in \eqref{eq2.6-1}. By symmetry, consider for example the following term:
\begin{align}
&\mathbb{E}\left| \sum_{l=2 + 2^{m-n}(k-1)}^{2^{m-n}k}
\triangle_{2l-1}^{m+1} B^{H, \lambda}
\otimes \triangle_{2l}^{m+1} B^{H, \lambda}
\otimes \triangle_{2l}^{m+1} B^{H, \lambda} \right|^2 \nonumber\\
&= \sum_{i,j,\xi}
\sum_{l=2 + 2^{m-n}(k-1)}^{2^{m-n}k}
\mathbb{E}\bigl(\triangle_{2l-1}^{m+1} B^{H, \lambda;i}\bigr)^2 \mathbb{E}\bigl(\triangle_{2l}^{m+1} B^{H, \lambda;j}\bigr)^2 \mathbb{E}\bigl(\triangle_{2l}^{m+1} B^{H, \lambda;\xi}\bigr)^2 \nonumber\\
&+ 2 \sum_{ i,j,\xi}
\sum_{
\substack{
    l>r \\
r,l=2^{(m-n)}(k-1)+1}
}^{2^{(m-n)}k}
\mathbb{E}\bigl(\triangle_{2r-1}^{m+1} B^{H, \lambda;i} \triangle_{2l-1}^{m+1} B^{H, \lambda;i}\bigr) \nonumber\\
&\times\mathbb{E}\bigl(\triangle_{2r}^{m+1} B^{H, \lambda;j} \triangle_{2l}^{m+1} B^{H, \lambda;j}\bigr)
\mathbb{E}\bigl(\triangle_{2r}^{m+1} B^{H, \lambda;\xi} \triangle_{2l}^{m+1} B^{H, \lambda;\xi}\bigr) \nonumber\\
&+ 2 \sum_i
\sum_{
\substack{
    l>r \\
r,l=2^{(m-n)}(k-1)+1}
}^{2^{(m-n)}k}
\mathbb{E}\bigg[\bigl(\triangle_{2l}^{m+1} B^{H, \lambda;i}\bigr)^2 \bigl(\triangle_{2r}^{m+1} B^{H, \lambda;i}\bigr)^2 \bigl(\triangle_{2r-1}^{m+1} B^{H, \lambda;i} \triangle_{2l-1}^{m+1} B^{H, \lambda;i}\bigr)\bigg].\nonumber
\end{align}
In view of Lemma A.2 and \cite[Lemma 13]{Coutin-Qian}, we deduce that
\begin{align}
&\mathbb{E}\left| \sum_{l=2 + 2^{m-n}(k-1)}^{2^{m-n}k}
\triangle_{2l-1}^{m+1} B^{H, \lambda}
\otimes \triangle_{2l}^{m+1} B^{H, \lambda}
\otimes \triangle_{2l}^{m+1} B^{H, \lambda} \right|^2
&\leq c_{H,\lambda} 2^{m-n} \frac{1}{2^{6mH}}.
\end{align}
Similarly we may estimate the other terms appeared in \eqref{eq2.6-1} and finally get that
\begin{align}
\mathbb{E}\left| B_{t_{k-1}^{n}, t_{k}^{n}}^{H, \lambda; m+1, 3}
-B_{t_{k-1}^{n}, t_{k}^{n}}^{H, \lambda; m, 3} \right|^2
\leq c_{H,\lambda} \bigl(2^{m-n}\bigr)^{1+2H} \frac{1}{2^{6mH}}.
\end{align}

Now we consider the term $B_{t_{k-1}^{n}, t_{k}^{n}}^{H, \lambda;m, 3}$.

For $n \geq m$, in view of \eqref{eq2.4},
by using some similar arguments as in \eqref{eq2.5},
we have
\begin{align}
[\EE|B_{t_{k-1}^{n}, t_{k}^{n}}^{H, \lambda;m, 3}|^2]^{\frac{1}{2}} \leq c \bigg(C^{2}_{\frac{1}{2^{m}}}\bigg)^{\frac{3}{2}}
\frac{2^{3m(1-H)}}{2^{3n}}.\nonumber
\end{align}
For $n < m$, it is clear that
\begin{align}
B_{t_{k-1}^{n}, t_{k}^{n}}^{H, \lambda;m, 3}
= \sum_{l=n+1}^{m}\bigg(B_{t_{k-1}^{n},t_{k}^{n}}^{H, \lambda;l, 3} - B_{t_{k-1}^{n},t_{k}^{n}}^{H, \lambda;l-1, 3}\bigg) + B_{t_{k-1}^{n},t_{k}^{n}}^{H, \lambda;n, 3}. \nonumber
\end{align}
Then by the Minkowski inequality, we see that
\begin{align}
[\EE|B_{t_{k-1}^{n},t_{k}^{n}}^{H, \lambda;m, 3}|^2]^{\frac{1}{2}}
&\leq \sum_{l=n+1}^m
\bigg[\EE\bigg|B_{t_{k-1}^{n},t_{k}^{n}}^{H, \lambda;l,3} - B_{t_{k-1}^{n},t_{k}^{n}}^{H, \lambda;l-1,3}\bigg|^2\bigg]^{\frac{1}{2}} + \bigg[\EE\bigg|B_{t_{k-1}^{n},t_{k}^{n}}^{H, \lambda;n,3}\bigg|^2\bigg]^{\frac{1}{2}}\nonumber\\
&
\leq
\sum_{l=n+1}^m
c_{H,\lambda} \bigl(2^{l-1-n}\bigr)^{1+2h} \frac{1}{2^{6H(l-1)}}
+c\bigg(C^{2}_{\frac{1}{2^{n}}}\bigg)^{\frac{3}{2}}
\frac{1}{2^{3H n}}
\nonumber\\
&
\leq \frac{c}{2^{\frac{n}{2}}}
\frac{1}{2^{\frac{(6H-1)n}{ 2}}}
+c\bigg(C^{2}_{\frac{1}{2^{n}}}\bigg)^{\frac{3}{2}}
\frac{1}{2^{3H n}}
\leq c
\frac{1}{2^{3H n}}.
\end{align}

We complete the proof.

\end{proof}

In order to study the behavior of $ \mathbf{B}^{H,\lambda;m}$ in the space $ \mathbf{\Omega}_p(\mathbb{R}^d) $, we may need to control the $ p $-variation distance $ d_p$ in a suitable way. For $ \bm{X}, \tilde{\bm{X}} \in \mathbf{\Omega}_p(\mathbb{R}^d)$, define
\begin{align}\label{eq2.6}
\rho_j(\bm{X}, \tilde{\bm{X}})
:= \bigg(\sum_{n=1}^\infty n^\gamma \sum_{k=1}^{2^n} |X_{t_{k-1}^{n},t_k^n}^j - \tilde{X}_{t_{k-1}^{n},t_k^n}^j|^{\frac{p}{j}}
\bigg)^{\frac{j}{p}}, \quad j = 1, 2,3,
\end{align}
where $\gamma > p - 1$ is some fixed universal constant. For notational simplicity,
we will use $ \rho_j(\bm{X}) $ to denote $ \rho_j(\bm{X}, \tilde{\bm{X}})$ with $ \tilde{\bm{X}} = (1, 0, 0,0)$.

The following result, which is important for us, is proved in
\cite{Lyons-Qian}.

\begin{Proposition}
There exists some positive constant $ R = R(p, \gamma) $, such that for any $ \bm{X}, \tilde{\bm{X}} \in \mathbf{\Omega}_p(\mathbb{R}^d) $,
\begin{align}
d_p(\bm{X}, \tilde{\bm{X}}) &\leq R
\max\{\rho_1(\bm{X}, \tilde{\bm{X}}),
\rho_2(\bm{X}, \tilde{\bm{X}}),
\rho_3(\bm{X}, \tilde{\bm{X}}),
\rho_1(\bm{X}, \tilde{\bm{X}})
(\rho_1(\bm{X}) + \rho_1(\tilde{\bm{X}})),
\nonumber\\
&
\rho_2(\bm{X}, \tilde{\bm{X}})
(\rho_1(\bm{X}) + \rho_1(\tilde{\bm{X}})),
\rho_1(\bm{X}, \tilde{\bm{X}})
(\rho_2(\bm{X}) + \rho_2(\tilde{\bm{X}}))
\}. \nonumber
\end{align}
\end{Proposition}

Now let
\begin{align}\label{eq2.7}
I(\bm{X}, \tilde{\bm{X}}) &:=
\max\{\rho_1(\bm{X}, \tilde{\bm{X}}),
\rho_2(\bm{X}, \tilde{\bm{X}}),
\rho_3(\bm{X}, \tilde{\bm{X}}),
\rho_1(\bm{X}, \tilde{\bm{X}})
(\rho_1(\bm{X}) + \rho_1(\tilde{\bm{X}})),
\nonumber\\
&
\rho_2(\bm{X}, \tilde{\bm{X}})
(\rho_1(\bm{X}) + \rho_1(\tilde{\bm{X}})),
\rho_1(\bm{X}, \tilde{\bm{X}})
(\rho_2(\bm{X}) + \rho_2(\tilde{\bm{X}}))
\}.
\end{align}
and observe that
\begin{align}\label{eq2.8}
&\{\omega : \mathbf{B}^{H,\lambda;m} \text{ is not Cauchy under } d_p\}\nonumber\\
&\subset \{\omega : \sum_{m=1}^{\infty} d_p(\mathbf{B}^{H,\lambda;m}, \mathbf{B}^{H,\lambda;m+1}) = \infty\}\nonumber\\
&\subset \limsup_{m \to \infty} \{\omega : d_p(\mathbf{B}^{H,\lambda;m}, \mathbf{B}^{H,\lambda;m+1}) > \frac{R}{2^{m\beta}}\}\nonumber\\
&
\subset \limsup_{m \to \infty} \{\omega : I(\mathbf{B}^{H,\lambda;m}, \mathbf{B}^{H,\lambda;m+1}) > \frac{1}{2^{m\beta}}\},
\end{align}
where $\beta$ is some positive constant to be chosen. Notice that the right hand side of \eqref{eq2.8} is
$\mathcal{B}(\Omega)$-measurable so its capacity is well-defined. Therefore, in order to prove that for almost-surely,
$\mathbf{B}^{H, \lambda; m}$ is a Cauchy sequence under $d_p$, it suffices to show that the right hand side of \eqref{eq2.8} has capacity zero. This can be shown by using the Borel-Cantelli lemma.

According to \eqref{eq2.7},
we may first need to establish estimates for
\begin{align}
P(\rho_j(\mathbf{B}^{H,\lambda;m}, \mathbf{B}^{H,\lambda;m+1}) > \eta), \ j = 1, 2,3,
\end{align}
and
\begin{align}
P(\rho_1(\mathbf{B}^{H,\lambda;m}) > \eta),~~~~
P(\rho_2(\mathbf{B}^{H,\lambda;m}) > \eta),
\end{align}
where $ m \geq 1 $ and $\eta > 0$. They are contained in the following lemma.

\begin{lemma}
For $H\in (\frac{1}{4}, \frac{1}{2})$, $\lambda > 0$,
$ m \geq 1$ and $\eta> 0$, we have the following estimates:
\begin{itemize}
\item[(1)]
\[P(\rho_1(\mathbf{B}^{H,\lambda;m})> \eta)
 \leq c\eta^{-p},\]
 \[P(\rho_2(\mathbf{B}^{H,\lambda;m})> \eta)
 \leq c\eta^{-p}. \]
\item[(2)] Let $\theta \in (0, Hp - 1)$ be some constant such that
\begin{align}
n^{\gamma+1} \leq c \frac{2^{n(p-1)}}{2^{n(p-\theta-1)}}, \quad
\forall n \geq 1, \nonumber
\end{align}
Then we have \begin{align}
P(\rho_1(\mathbf{B}^{H,\lambda;m}, \mathbf{B}^{H,\lambda;m+1}) > \eta) \leq c\eta^{-p}
[C^{2}_{\frac{1}{2^{m}}}]^{\frac{p}{2}}
\frac{1}{2^{m(Hp-\theta-1)}},
\nonumber
\end{align}

\begin{align}
P(\rho_2(\mathbf{B}^{H,\lambda;m}, \mathbf{B}^{H,\lambda;m+1}) > \eta) \leq c\eta^{-\frac{p}{2}}
[C^{2}_{\frac{1}{2^{m}}}]^{\frac{p}{2}}
\frac{1}{2^{m(Hp-\theta-1)}},\nonumber
\end{align}

\begin{align}
P(\rho_3(\mathbf{B}^{H,\lambda;m}, \mathbf{B}^{H,\lambda;m+1}) > \eta) \leq c\eta^{-\frac{p}{3}}
\frac{1}{2^{m(Hp-\theta-1)}}.\nonumber
\end{align}

\end{itemize}

\end{lemma}

\begin{proof}
First consider
\begin{align}
P(\rho_1(\mathbf{B}^{H,\lambda;m}) > \eta) = P\bigg(\sum_{n=1}^\infty n^\gamma \sum_{k=1}^{2^n} |B_{t_{k-1}^{n},t_{k}^{n}}^{H,\lambda;m,1}|^p > \eta^p\bigg).
\nonumber
\end{align}
Define
\begin{align}
A_N := \bigg\{\omega : \sum_{n=1}^N n^\gamma \sum_{k=1}^{2^n} |B_{t_{k-1}^{n},t_{k}^{n}}^{H,\lambda;m,1}|^p > \eta^p\bigg\} \in \mathcal{B}(\Omega),\nonumber
\end{align}
and
\begin{align}
A := \bigg\{\omega : \sum_{n=1}^\infty n^\gamma \sum_{k=1}^{2^n} |B_{t_{k-1}^{n},t_{k}^{n}}^{H,\lambda;m,1}|^p > \eta^p\bigg\} \in \mathcal{B}(\Omega). \nonumber
\end{align}
It is obvious that $A_N \uparrow A$. By the properties of the capacity $P$, we have
\begin{align}
P(A) = \lim_{N \to \infty} P(A_N).   \nonumber
\end{align}
On the other hand, by the Chebyshev inequality for the capacity $P$ and Lemma 2.1, we have
\begin{align}
P(A_N) &\leq \eta^{-p} \sum_{n=1}^N n^\gamma \sum_{k=1}^{2^n} \mathbb{E}\bigg[ \bigg| B_{t_{k-1}^{n},t_{k}^{n}}^{H,\lambda;m,1}\bigg|^p \bigg]
\nonumber\\
&
\leq c\eta^{-p} \bigg[ \sum_{n=1}^m n^\gamma 2^n
\bigg(C^{2}_{\frac{1}{2^{n}}}
\bigg(\frac{1}{2^{n}}\bigg)^{2H}\bigg)^{\frac{p}{2}}
+ \sum_{n=m+1}^{\infty} n^\gamma 2^n
\bigg[
\frac{2^{m}}{2^{n}}
\bigg(
C^{2}_{\frac{1}{2^{m}}}
\bigg(\frac{1}{2^{m}}\bigg)^{2H}\bigg)^{\frac{1}{2}}
\bigg]^{p}\bigg]\nonumber\\
&\leq c\eta^{-p}
\bigg[ \sum_{n=1}^m n^\gamma
[C^{2}_{\frac{1}{2^{n}}}]^{\frac{p}{2}}
\frac{1}{2^{n(Hp-1)}}
+ 2^{mp}
\bigg[C^{2}_{\frac{1}{2^{m}}}
\bigg(\frac{1}{2^{m}}\bigg)^{2H}
\bigg]^{\frac{p}{2}}
\sum_{n=m+1}^{\infty} n^\gamma
\frac{1}{2^{n(p-1)}}\bigg]\nonumber\\
&
\leq c\eta^{-p}. \nonumber
\end{align}
It follows that
\begin{align}
P(\rho_1(\mathbf{B}^{H,\lambda;m}) > \eta) = P(A) \leq c\eta^{-p}.
\nonumber
\end{align}
Similarly,
\begin{align}
P(\rho_2(\mathbf{B}^{H,\lambda;m}) > \eta) \leq c\eta^{-p}.
\nonumber
\end{align}

Now consider
\begin{align}
P(\rho_1(\mathbf{B}^{H,\lambda;m}, \mathbf{B}^{H,\lambda;m+1}) > \eta) = P\bigg( \sum_{n=1}^{\infty} n^\gamma \sum_{k=1}^{2^n} \bigg| B_{t^{n}_{k-1}, t^{n}_{k}}^{H,\lambda;m+1,1} - B_{t^{n}_{k-1}, t^{n}_{k}}^{H,\lambda;m,1} \bigg|^p > \eta^p \bigg).
\nonumber
\end{align}
By similar reasons we will have
\begin{align}
&P(\rho_1(\mathbf{B}^{H,\lambda;m}, \mathbf{B}^{H,\lambda;m+1}) > \eta)\nonumber\\
&\leq \eta^{-p} \sum_{n=1}^{\infty}
n^{\gamma} \sum_{k=1}^{2^n} \mathbb{E}\bigg[ \bigg| B_{t^{n}_{k-1}, t^{n}_{k}}^{H,\lambda;m+1,1} - B_{t^{n}_{k-1}, t^{n}_{k}}^{H,\lambda;m,1}  \bigg|^p \bigg]\nonumber\\
&
\leq c\eta^{-p}  \sum_{n=m+1}^{\infty} n^{\gamma} 2^n \frac{2^{mp}}{2^{np}}
\bigg[C^{2}_{\frac{1}{2^{m}}}\bigg(\frac{1}{2^{m}}\bigg)^{2H}
\bigg]^{\frac{p}{2}}
\nonumber\\
&
= c\eta^{-p}
2^{mp}
\bigg[C^{2}_{\frac{1}{2^{m}}}\bigg(\frac{1}{2^{m}}\bigg)^{2H}
\bigg]^{\frac{p}{2}}
\sum_{n=m+1}^{\infty} n^{\gamma} \frac{1}{2^{n(p-1)}}.\nonumber
\end{align}
Since $\theta \in (0, Hp-1)$ is such that
\begin{align}
n^{\gamma+1} \leq c \frac{2^{n(p-1)}}{2^{n(p-\theta-1)}}, \quad \forall n \geq 1, \nonumber
\end{align}
we arrive at
\begin{align}
P(\rho_1(\mathbf{B}^{H,\lambda;m}, \mathbf{B}^{H,\lambda;m+1}) > \eta)
\leq c\eta^{-p}
[C^{2}_{\frac{1}{2^{m}}}]^{\frac{p}{2}}
\frac{1}{2^{m(Hp-\theta-1)}}.\nonumber
\end{align}
Consider the second level part. By similar reasons, we have
\begin{align}
&P(\rho_2(\mathbf{B}^{H,\lambda;m}, \mathbf{B}^{H,\lambda;m+1}) > \eta)\nonumber\\
&\leq \eta^{-\frac{p}{2}} \sum_{n=1}^{\infty}
n^{\gamma} \sum_{k=1}^{2^n} \mathbb{E}\bigg[ \bigg|
B_{t^{n}_{k-1}, t^{n}_{k}}^{H,\lambda;m+1,2} -
B_{t^{n}_{k-1}, t^{n}_{k}}^{H,\lambda;m,2}  \bigg|^{\frac{p}{2}}
\bigg]\nonumber\\
&\leq c\eta^{-\frac{p}{2}}
\bigg[ \sum_{n=1}^m n^{\gamma} 2^n
\bigg(
C^{2}_{\frac{1}{2^{m}}}
\frac{1}{2^{\frac{(4H-1)m}{ 2}}2^{\frac{n}{2}}}
\bigg)^{\frac{p}{2}}
+ \sum_{n=m+1}^\infty
 n^{\gamma} 2^n
\bigg(
C^{2}_{\frac{1}{2^m}}
\frac{2^{2m(1-H)}}{2^{2n}}
\bigg)^{\frac{p}{2}}\bigg]\nonumber\\
&\leq
c\eta^{-\frac{p}{2}}
\bigg[
[C^{2}_{\frac{1}{2^{m}}}]^{\frac{p}{2}}
\frac{1}{2^{\frac{(4H-1)mp}{ 4}}}
\sum_{n=1}^m
n^{\gamma} 2^{n(1-\frac{p}{4})}
+
[C^{2}_{\frac{1}{2^{m}}}]^{\frac{p}{2}}
2^{m(1-H)p}
\sum_{n=m+1}^\infty n^{\gamma}
2^{n(1-p)}
\bigg]\nonumber\\
&\leq
c\eta^{-\frac{p}{2}}
\bigg[
[C^{2}_{\frac{1}{2^{m}}}]^{\frac{p}{2}}
\frac{1}{2^{\frac{(4H-1)mp}{ 4}}}
2^{m(1-\frac{p}{4})}m^{\gamma+1}
+
[C^{2}_{\frac{1}{2^{m}}}]^{\frac{p}{2}}
2^{m(1-H)p}
\frac{1}{2^{m(p-1-\theta)}}
\bigg]\nonumber\\
&\leq
c\eta^{-\frac{p}{2}}
\bigg[
[C^{2}_{\frac{1}{2^{m}}}]^{\frac{p}{2}}
\frac{1}{2^{\frac{(4H-1)mp}{ 4}}}
2^{m(1-\frac{p}{4}+\theta)}
+
[C^{2}_{\frac{1}{2^{m}}}]^{\frac{p}{2}}
\frac{1}{2^{m(Hp-\theta-1)}}
\bigg]\nonumber\\
&\leq
c\eta^{-\frac{p}{2}}
[C^{2}_{\frac{1}{2^{m}}}]^{\frac{p}{2}}
\frac{1}{2^{m(Hp-\theta-1)}}.\nonumber
\end{align}
Finally, for the third level part, we have
\begin{align}
&P(\rho_3(\mathbf{B}^{H,\lambda;m}, \mathbf{B}^{H,\lambda;m+1}) > \eta)\nonumber\\
&\leq c\eta^{-\frac{p}{3}}
\bigg[
\sum_{n=1}^m n^{\gamma} 2^n
\bigg(
\frac{1}
{2^{m(4H-1)}2^{n(1+2H)}}\bigg)^{\frac{p}{6}}
+\sum_{n=m+1}^\infty
 n^{\gamma} 2^n
\bigg(C^{2}_{\frac{1}{2^{m}}}\bigg)^{\frac{p}{2}}
\bigg(\frac{2^{3m(1-H)}}{2^{3n}} \bigg)^{\frac{p}{3}}
\bigg]\nonumber\\
&\leq
c\eta^{-\frac{p}{3}}
\frac{1}{2^{m(Hp-\theta-1)}}.\nonumber
\end{align}

\end{proof}

Now we are in position to prove the main result of this section.

\begin{theorem}
Outside a $\mathcal{B}(\Omega)$-measurable set of capacity zero, $\mathbf{B}^{H,\lambda;m}$ is a Cauchy sequence under the $p$-variation distance $d_p$. In particular, for almost-surely, the sample paths of $B_{t}^{H,\lambda}$ can be enhanced to be geometric rough paths
\begin{align}
\mathbf{B}_{s,t}^{H,\lambda} =
(1, B_{s,t}^{H,\lambda;1}, B_{s,t}^{H,\lambda;2},
B_{s,t}^{H,\lambda;3}
), \quad 0 \leq s < t \leq 1, \nonumber
\end{align}
with roughness $p$, which are defined as the limit of sample (geometric rough) paths of $\mathbf{B}^{H,\lambda;m}$ in $\mathbf{\Omega}_p(\mathbb{R}^d)$ under the $p$-variation distance $d_p$.
\end{theorem}

\begin{proof}
By Lemma 2.3, we have
\begin{align}
&P\bigg(I(\mathbf{B}^{H,\lambda;m}, \mathbf{B}^{H,\lambda;m+1}) > \frac{1}{2^{m\beta}}\bigg)\nonumber\\
&\leq \sum_{j=1}^3
P\bigg(\rho_j(
\mathbf{B}^{H,\lambda;m}, \mathbf{B}^{H,\lambda;m+1}
) > \frac{1}{2^{m\beta}}\bigg) \nonumber\\
&~~~~ + P
\bigg(\rho_1(
\mathbf{B}^{H,\lambda;m}, \mathbf{B}^{H,\lambda;m+1})(
\rho_1(\mathbf{B}^{H,\lambda;m}) +
\rho_1(\mathbf{B}^{H,\lambda;m+1}))
> \frac{1}{2^{m\beta}}\bigg)\nonumber\\
&~~~~+ P
\bigg(\rho_1(
\mathbf{B}^{H,\lambda;m}, \mathbf{B}^{H,\lambda;m+1})(
\rho_2(\mathbf{B}^{H,\lambda;m}) +
\rho_2(\mathbf{B}^{H,\lambda;m+1}))
> \frac{1}{2^{m\beta}}\bigg)\nonumber\\
&~~~~+ P
\bigg(\rho_2(
\mathbf{B}^{H,\lambda;m}, \mathbf{B}^{H,\lambda;m+1})(
\rho_1(\mathbf{B}^{H,\lambda;m}) +
\rho_1(\mathbf{B}^{H,\lambda;m+1}))
> \frac{1}{2^{m\beta}}\bigg)\nonumber\\
&\leq
3 P\bigg(\rho_1(
\mathbf{B}^{H,\lambda;m}, \mathbf{B}^{H,\lambda;m+1})
> \frac{1}{2^{2m\beta}}\bigg) \nonumber\\
&~~~~+2
P\bigg(\rho_2(
\mathbf{B}^{H,\lambda;m}, \mathbf{B}^{H,\lambda;m+1})
> \frac{1}{2^{2m\beta}}\bigg)\nonumber\\
&~~~~+
P\bigg(\rho_3(
\mathbf{B}^{H,\lambda;m}, \mathbf{B}^{H,\lambda;m+1})
> \frac{1}{2^{m\beta}}\bigg)\nonumber\\
&~~~~+2 P\bigg(\rho_1(\mathbf{B}^{H,\lambda;m})
> \frac{2^{m\beta}}{2}\bigg)
+2 P\bigg(\rho_1(
\mathbf{B}^{H,\lambda;m+1}
) > \frac{2^{m\beta}}{2}\bigg)\nonumber\\
&~~~~+P\bigg(\rho_2(\mathbf{B}^{H,\lambda;m})
> \frac{2^{m\beta}}{2}\bigg)
+P\bigg(\rho_2(\mathbf{B}^{H,\lambda;m+1})
> \frac{2^{m\beta}}{2}\bigg)\nonumber\\
&\leq c\bigg[
[C^{2}_{\frac{1}{2^{m}}}]^{\frac{p}{2}}
\frac{1}{2^{m(Hp-1-\theta-2\beta p)}}
+
[C^{2}_{\frac{1}{2^{m}}}]^{\frac{p}{2}}
\frac{1}{2^{m(Hp-1-\theta-\frac{\beta p}{2})}}
+\frac{1}{2^{m\beta p}}\bigg], \nonumber
\end{align} where $\theta \in (0, Hp-1)$ is some fixed constant.

If we choose $\beta$ such that
\begin{align}
0 < \beta < \frac{Hp-\theta-1}{2p},\nonumber
\end{align} then
\begin{align}
\sum_{m=1}^{\infty} P\bigg(I(\mathbf{B}^{H,\lambda;m}, \mathbf{B}^{H,\lambda;m+1}) > \frac{1}{2^{m\beta}}\bigg) < \infty.
\nonumber
\end{align}
By the Borel-Cantelli lemma, we have
\begin{align}
P\bigg(\limsup_{m \to \infty} \bigg\{\omega : I(\mathbf{B}^{H,\lambda;m}, \mathbf{B}^{H,\lambda;m+1}) > \frac{1}{2^{m\beta}}\bigg\}\bigg) = 0, \nonumber
\end{align}
and the result follows from the inclusion \eqref{eq2.8}.

\end{proof}

\noindent{\section{Existence-uniqueness theorem and solution norm estimates}}

To establish the existence and uniqueness of the solution to \eqref{eq0.1}, we first introduce the pure rough differential equation given by
\begin{align}\label{eq4.1}
dy_t = g(y_t) dB^{H,\lambda}_{t}, \quad \forall t \in [a,b], y_a \in \mathbb{R}^d.
\end{align}
In particular, we first verify the differentiability of the solution to \eqref{eq4.1} with respect to its initial condition, and then transform the original system \eqref{eq0.1} into an ordinary differential equation by applying the Doss-Sussmann technique \cite{Sussmann}.

It should be noted that the existence, uniqueness, and continuity of the solution to \eqref{eq4.1} can be established by arguments analogous to those in \cite{Gubinelli}. We thus omit the detailed proof here and focus on verifying the differentiability of the solution $y_t(\mathbf{B}^{H,\lambda}, y_a)$ to \eqref{eq4.1} with respect to the initial value $y_a$. To this end, we first present the following a priori estimates.

\begin{Proposition}
The solution $y_t(\mathbf{B}^{H,\lambda}, y_a)$ to \eqref{eq4.1} is uniformly continuous
w.r.t. $y_a$, i.e. for any two solutions $y(\mathbf{B}^{H,\lambda}, y_a)$ and $\bar{y}(\mathbf{B}^{H,\lambda}, \bar{y}_a)$ the following estimates hold
\begin{align}
    \| \bar{y} - y \|_{\infty,[a,b]} \leq \| \bar{y}_a - y_a \| e^{(\log 2) \bar{N}_{[a,b]}(\mathbf{B}^{H,\lambda})},\nonumber
\end{align}
and
\begin{align}\label{eq4.2}
&||| \bar{y} - y, R^{(\bar{y}-y)\sharp} ,
R^{(\bar{y}-y)\sharp\sharp}|||_{p-\text{var},[a,b]}
\nonumber\\
&\leq \| \bar{y}_a - y_a \|\bar{N}_{[a,b]}^{\frac{p-1}{p}}(\mathbf{B}^{H,\lambda})
e^{(\log 2) \bar{N}_{[a,b]}(\mathbf{B}^{H,\lambda})}
- \| \bar{y}_a - y_a \|,
\end{align}
where $||| y, R^{y\sharp}, R^{y\sharp\sharp}|||_{p-\text{var},[s,t]} := ||| y |||_{p-\text{var},[s,t]}+
||| R^{y\sharp} |||_{\frac{p}{3}-\text{var},[s,t]^2}
+||| R^{y\sharp\sharp} |||_{\frac{p}{2}-\text{var},[s,t]^2}
$,
$\bar{N}_{[a,b]}(\mathbf{B}^{H,\lambda})$ is the maximal index of the maximal stopping time in the
greedy sequence
\begin{align}\label{eq4.3}
    \tau_0 = a, \tau_{k+1} := \inf \left\{ t > \tau_k :
     ||| \mathbf{B}^{H,\lambda} |||_{p-\text{var},[\tau_k, t]} = \min\bigg\{\frac{1}{12\Theta_{3}}, \bigg(\frac{1}{12\Theta_{3}}\bigg)^{\frac{1}{3}}\bigg\} \right\} \wedge b,
\end{align}
that lies in the interval $[a,b]$, with $\Theta_{3}$ given by \eqref{eq4.42-2}.

\end{Proposition}

\begin{proof}
The proof proceeds in two steps.

\noindent \textbf{Step 1:} Recall from \eqref{eq0-2.15} that
\begin{align}\label{eq4.4}
&\bigg\| \int_s^t g(y_u) db^{H,\lambda}_{u} - g(y_s) \otimes B^{H,\lambda;1}_{s,t} -
[g(y)]'_s B^{H,\lambda;2}_{s,t}
-[g(y)]''_s B^{H,\lambda;3}_{s,t}
\bigg\| \nonumber\\
&\leq C_\alpha |t - s|^{4\alpha} \bigg( ||| B^{H,\lambda;1} |||_{\alpha,[s,t]} ||| R^{g(y)\sharp} |||_{3\alpha,[s,t]^2} +
||| R^{g(y)\sharp \sharp} |||_{2\alpha,[s,t]^2}
||| B^{H,\lambda;2}|||_{2\alpha,[s,t]^2}
\nonumber\\
&~~~~~~~+
||| [g(y)]'' |||_{\alpha,[s,t]} ||| B^{H,\lambda;3} |||_{3\alpha,[s,t]^2} \bigg).
\end{align}
It follows that $y$ is controlled by $b^{H,\lambda}$ with $y' = g(y)$. In view of \eqref{eq0-2.11}, we have
\begin{align}\label{eq4.5}
g(y_t) - g(y_s) &=
\int_0^1 Dg(y_s + \eta y_{s,t}) y_{s,t} d\eta \nonumber\\
&=Dg(y_s) y'_s \otimes B^{H,\lambda;1}_{s,t}+
Dg(y_s) y''_s \otimes B^{H,\lambda;2}_{s,t}
\nonumber\\
&+\int_0^1 [ Dg(y_s + \eta y_{s,t}) - Dg(y_s) ]
y'_s \otimes B^{H,\lambda;1}_{s,t} d\eta
\nonumber\\
&+\int_0^1 [ Dg(y_s + \eta y_{s,t}) - Dg(y_s) ]
y''_s \otimes B^{H,\lambda;2}_{s,t} d\eta
\nonumber\\
&+\int_0^1 Dg(y_s + \eta y_{s,t}) R^{y\sharp}_{s,t} d\eta
\nonumber\\
&=Dg(y_s) y'_s \otimes B^{H,\lambda;1}_{s,t}+
Dg(y_s) y''_s \otimes B^{H,\lambda;2}_{s,t}
+D^{2}g(y_s)(y'_s)^{2}B^{H,\lambda;2}_{s,t}
\nonumber\\
&+\int_0^1 [ Dg(y_s + \eta y_{s,t}) - Dg(y_s) ]
y''_s \otimes B^{H,\lambda;2}_{s,t} d\eta
\nonumber\\
&+\int_0^1 Dg(y_s + \eta y_{s,t}) R^{y\sharp}_{s,t} d\eta
\nonumber\\
&+\int_0^1\int_0^1\int_0^1
D^{3}g(y_s+\xi\eta\tau y_{s,t})\xi\eta y_{s,t}
d\tau \eta y_{s,t}d\xi y'_s \otimes B^{H,\lambda;1}_{s,t} d\eta
\nonumber\\
&+\frac{1}{2}D^{2}g(y_s)y''_sy'_s
B^{H,\lambda;2}_{s,t}\otimes
B^{H,\lambda;1}_{s,t}
+\frac{1}{2}D^{2}g(y_s)R^{y\sharp}_{s,t}
y'_s\otimes B^{H,\lambda;1}_{s,t},
\end{align}
where we have used
\begin{align}\label{eq4.5-1}
&\int_0^1 [ Dg(y_s + \eta y_{s,t}) - Dg(y_s) ]
y'_s \otimes B^{H,\lambda;1}_{s,t} d\eta
\nonumber\\
&=\int_0^1\int_0^1D^{2}g
(y_s + \eta \xi y_{s,t})\eta y_{s,t}d\xi
y'_s\otimes B^{H,\lambda;1}_{s,t} d\eta\nonumber\\
&=\int_0^1\int_0^1[
D^{2}g(y_s + \xi\eta y_{s,t}) - D^{2}g(y_s) ]
\eta y_{s,t}d\xi y'_s\otimes B^{H,\lambda;1}_{s,t}d\eta
+\frac{1}{2}D^{2}g(y_s)y_{s,t}y'_s \otimes B^{H,\lambda;1}_{s,t}
\nonumber\\
&=\int_0^1\int_0^1[
D^{2}g(y_s + \xi\eta y_{s,t}) - D^{2}g(y_s) ]
\eta y_{s,t}d\xi y'_s\otimes B^{H,\lambda;1}_{s,t}d\eta
+D^{2}g(y_s)(y'_s)^{2}B^{H,\lambda;2}_{s,t}
\nonumber\\
&~~~~+\frac{1}{2}D^{2}g(y_s)y''_s
y'_s B^{H,\lambda;2}_{s,t}\otimes B^{H,\lambda;1}_{s,t}
+\frac{1}{2}D^{2}g(y_s)R^{y\sharp}_{s,t}y'_s\otimes B^{H,\lambda;1}_{s,t}.
\end{align}
Then by \eqref{eq0-2.11}, we obtain that
\begin{align}\label{eq4.6}
R^{g(y)\sharp}_{s,t}&=
\int_0^1 \int_0^1
D^{2}g(y_s + \eta\xi y_{s,t})\eta y_{s,t}d\xi
y''_s \otimes B^{H,\lambda;2}_{s,t} d\eta
\nonumber\\
&+\int_0^1 Dg(y_s + \eta y_{s,t}) R^{y\sharp}_{s,t} d\eta
\nonumber\\
&+\int_0^1\int_0^1[
D^{2}g(y_s + \xi\eta y_{s,t}) - D^{2}g(y_s) ]
\eta y_{s,t}d\xi y'_s\otimes B^{H,\lambda;1}_{s,t}d\eta
\nonumber\\
&+\frac{1}{2}D^{2}g(y_s)y''_sy'_s
B^{H,\lambda;2}_{s,t}\otimes B^{H,\lambda;1}_{s,t}
+\frac{1}{2}D^{2}g(y_s)R^{y\sharp}_{s,t}
y'_s\otimes B^{H,\lambda;1}_{s,t}.
\end{align}
Using $y'_{s} = g(y_{s})$ and $y''_{s}=[ g(y) ]'_s = Dg(y_s)g(y_s)$, we estimate
\begin{align}\label{eq4.7}
\| R^{g(y)\sharp}_{s,t} \|
&\leq C_{g}^{3}
\int_0^1\int_0^1 \eta \|y_{s,t}\|\|B^{H,\lambda;2}_{s,t}\| d\xi d\eta
+C_{g}\int_0^1 \|R^{y\sharp}_{s,t}\| d\eta
+\frac{1}{2}C_{g}^{4}\|B^{H,\lambda;2}_{s,t}\|
\|B^{H,\lambda;1}_{s,t}\|\nonumber\\
&+C_{g}^{2}
\|y_{s,t}\|\| B^{H,\lambda;1}_{s,t} \|
+\frac{1}{2}C_{g}^{2}\|R^{y\sharp}_{s,t}\|
\|B^{H,\lambda;1}_{s,t}\|\nonumber\\
&\leq
\frac{1}{2}C_{g}^{3}\|y_{s,t}\|\|B^{H,\lambda;2}_{s,t}\|
+C_{g}\|R^{y\sharp}_{s,t}\|
+\frac{1}{2}C_{g}^{4}\|B^{H,\lambda;2}_{s,t}\|
\|B^{H,\lambda;1}_{s,t}\|
\nonumber\\
&
+C_{g}^{2}
\| B^{H,\lambda;1}_{s,t} \|\|y_{s,t}\|
+\frac{1}{2}C_{g}^{2}|||B^{H,\lambda;1}|||_{p-\text{var},[s,t]}
\|R^{y\sharp}_{s,t}\|
.
\end{align}
This, together with the H\"older inequality and the Minkowski inequality, yields
\begin{align}\label{eq4.9}
&||| R^{g(y)\sharp} |||_{\frac{p}{3}-\text{var},[s,t]^2}
\nonumber\\
&\leq
\frac{1}{2}C_{g}^{3}|||y|||_{p-\text{var},[s,t]}
|||B^{H,\lambda;2}|||_{\frac{p}{2}-\text{var},[s,t]^{2}}
+C_{g}
|||R^{y\sharp}|||_{\frac{p}{3}-\text{var},[s,t]^{2}}
\nonumber\\
&+C_{g}^{2}(t-s)^{\frac{2}{p}}
||| B^{H,\lambda;1} |||_{p-\text{var},[s,t]}
|||y|||_{p-\text{var},[s,t]}
+\frac{1}{2}C_{g}^{4}
|||B^{H,\lambda;2}|||_{\frac{p}{2}-\text{var},[s,t]^{2}}
|||B^{H,\lambda;1}|||_{p-\text{var},[s,t]}\nonumber\\
&
+\frac{1}{2}C_{g}^{2}
|||B^{H,\lambda;1}|||_{p-\text{var},[s,t]}
|||R^{y\sharp}|||_{\frac{p}{3}-\text{var},[s,t]^{2}}
.
\end{align}
\begin{align}\label{eq4.8}
||| [g(y)]' |||_{p-\text{var},[s,t]} \leq 2 C_g^2 ||| y |||_{p-\text{var},[s,t]}, \quad \| [g(y)]' \|_{\infty,[s,t]} \leq C_g^2.
\end{align}
Due to
\begin{align}\label{eq4.10}
[g(y_s)]''=[Dg(y_s)]^{2}g(y_s)+D^{2}g(y_s)[g(y_s)]^{2},
\end{align}
by the Minkowski inequality, we have
\begin{align}\label{eq4.11}
|||[g(y)]''|||_{p-\text{var},[s,t]}
&\leq\bigg(
\sup_{\Pi([s,t])} \sum_{i=1}^n
\|[Dg(y_{t_{i+1}})]^{2}g(y_{t_{i+1}})
-[Dg(y_{t_{i}})]^{2}g(y_{t_{i}})
\|^p \bigg)^{1/p}\nonumber\\
&+
\bigg(
\sup_{\Pi([s,t])} \sum_{i=1}^n
\|D^{2}g(y_{t_{i+1}})[g(y_{t_{i+1}})]^{2}
-D^{2}g(y_{t_{i}})[g(y_{t_{i}})]^{2}
\|^p \bigg)^{1/p}\nonumber\\
&\leq 6C_{g}^{3}|||y|||_{p-\text{var},[s,t]}.
\end{align}
It follows from \eqref{eq0-2.11-1} that
\begin{align}\label{eq4.12}
R^{g(y)\sharp\sharp}_{s,t}
&=[g(y)]'_{t}-[g(y)]'_{s}-[g(y)]''_{s}B^{H,\lambda;1}_{s,t}
\nonumber\\
&=Dg(y_{t})g(y_{t})
-Dg(y_{s})g(y_{s})
-[D^{2}g(y_{s})g^{2}(y_{s})
+[Dg(y_{s})]^{2}g(y_{s})
]B^{H,\lambda;1}_{s,t}\nonumber\\
&=[Dg(y_{t})-Dg(y_{s})]g(y_{t})
+Dg(y_{s})[g(y_{t})-g(y_{s})]
\nonumber\\
&-
[D^{2}g(y_{s})g^{2}(y_{s})
+[Dg(y_{s})]^{2}g(y_{s})
]B^{H,\lambda;1}_{s,t}\nonumber\\
&=\int_{0}^{1}
D^{2}g(y_{s}+\eta y_{s,t} )y_{s,t}d\eta g(y_{t})
+ Dg(y_{s})\int_{0}^{1}
Dg(y_{s}+\eta y_{s,t} )y_{s,t}d\eta
\nonumber\\
&-
[D^{2}g(y_{s})g^{2}(y_{s})
+[Dg(y_{s})]^{2}g(y_{s})
]B^{H,\lambda;1}_{s,t}\nonumber\\
&=\int_{0}^{1}
D^{2}g(y_{s}+\eta y_{s,t} )y'_{s}g(y_{t})B^{H,\lambda;1}_{s,t}
-D^{2}g(y_{s})g^{2}(y_{s})B^{H,\lambda;1}_{s,t}d\eta\nonumber\\
&+\int_{0}^{1}
D^{2}g(y_{s}+\eta y_{s,t} )
(y''_s B^{H,\lambda;2}_{s,t}+
    R^{y\sharp}_{s,t})d\eta g(y_{t})
\nonumber\\
&+\int_{0}^{1}
[Dg(y_{s}+\eta y_{s,t} )
-Dg(y_{s})]Dg(y_{s})g(y_{s})
B^{H,\lambda;1}_{s,t}d\eta\nonumber\\
&+\int_{0}^{1}Dg(y_{s}+\eta y_{s,t} )
(y''_s B^{H,\lambda;2}_{s,t}+
    R^{y\sharp}_{s,t})d\eta Dg(y_{s})
\nonumber\\
&=\int_{0}^{1}
[D^{2}g(y_{s}+\eta y_{s,t} )-D^{2}g(y_{s})]
g(y_{t})g(y_{s})B^{H,\lambda;1}_{s,t}d\eta\nonumber\\
&+\int_{0}^{1}D^{2}g(y_{s})g(y_{s})
[g(y_{t} )-g(y_{s})] B^{H,\lambda;1}_{s,t}d\eta\nonumber\\
&+\int_{0}^{1}D^{2}g(y_{s}+\eta y_{s,t} )
(y''_s B^{H,\lambda;2}_{s,t}+
    R^{y\sharp}_{s,t})d\eta g(y_{t})
\nonumber\\
&+\int_{0}^{1}
[Dg(y_{s}+\eta y_{s,t} )
-Dg(y_{s})]Dg(y_{s})g(y_{s})
B^{H,\lambda;1}_{s,t}d\eta\nonumber\\
&+\int_{0}^{1}Dg(y_{s}+\eta y_{s,t} )
(y''_s B^{H,\lambda;2}_{s,t}+
    R^{y\sharp}_{s,t})d\eta Dg(y_{s}),
\end{align}
which leads to
\begin{align}\label{eq4.13}
\|R^{g(y)\sharp\sharp}_{s,t}\|
\leq 2C^{3}_{g}\|y_{s,t}\|\|B^{H,\lambda;1}_{s,t}\|
+2C^{4}_{g}\|B^{H,\lambda;2}_{s,t}\|
+2C^{2}_{g}\|R^{y\sharp}_{s,t}\|,
\end{align}
and thus
\begin{align}\label{eq4.14}
|||R^{g(y)\sharp\sharp}|||_{
\frac{p}{2}-\text{var},[s,t]^{2}}
&\leq 2C^{3}_{g}|||y|||_{
p-\text{var},[s,t]}
|||B^{H,\lambda;1}|||_{p-\text{var},[s,t]}\nonumber\\
&+2C^{4}_{g}
|||B^{H,\lambda;2}|||_{\frac{p}{2}-\text{var},[s,t]^{2}}
+2C^{2}_{g}|||R^{y\sharp}|||
_{\frac{p}{3}-\text{var},[s,t]^{2}}.
\end{align}
In view of \eqref{eq0-2.17}, \eqref{eq4.9}, \eqref{eq4.11} and \eqref{eq4.14}, we deduce the estimate
\begin{align}\label{eq4.15}
\| y_{s,t} \|
&= \bigg\| \int_s^t g(y_u) db^{H,\lambda}_u \bigg\|
\nonumber\\
&\leq
\| g(y_s) \| \| B^{H,\lambda;1}_{s,t} \| +
\| [g(y)]'_s\| \| B^{H,\lambda;2}_{s,t} \|
+
\| [g(y)]''_s \| \| B^{H,\lambda;3}_{s,t} \|
\nonumber\\
&+
C_p \bigg( ||| B^{H,\lambda;1} |||_{p-\text{var},[s,t]} ||| R^{g(y)\sharp}|||_{\frac{p}{3}-\text{var},[s,t]^2} +
||| B^{H,\lambda;2} |||_{\frac{p}{2}-\text{var},[s,t]^{2}} ||| R^{g(y)\sharp\sharp} |||_{\frac{p}{2}-\text{var},[s,t]^2}
\nonumber\\
&+
||| [g(y)]'' |||_{p-\text{var},[s,t]} ||| B^{H,\lambda;3} |||_{\frac{p}{3}-\text{var},[s,t]^2} \bigg)
\nonumber\\
&\leq
C_p\bigg(
\frac{1}{2}C_{g}^{3}|||y|||_{p-\text{var},[s,t]}
|||B^{H,\lambda;2}|||_{\frac{p}{2}-\text{var},[s,t]^{2}}
||| B^{H,\lambda;1} |||_{p-\text{var},[s,t]}
\nonumber\\
&+
C_{g}
|||R^{y\sharp}|||_{\frac{p}{3}-\text{var},[s,t]^{2}}
||| B^{H,\lambda;1} |||_{p-\text{var},[s,t]}
+C_{g}^{2}(t-s)^{\frac{2}{p}}
||| B^{H,\lambda;1} |||^{2}_{p-\text{var},[s,t]}
|||y|||_{p-\text{var},[s,t]}
\nonumber\\
&+\frac{1}{2}C_{g}^{4}
|||B^{H,\lambda;2}|||_{\frac{p}{2}-\text{var},[s,t]^{2}}
|||B^{H,\lambda;1}|||^{2}_{p-\text{var},[s,t]}
+\frac{1}{2}C_{g}^{2}
|||B^{H,\lambda;1}|||^{2}_{p-\text{var},[s,t]}
|||R^{y\sharp}|||_{\frac{p}{3}-\text{var},[s,t]^{2}}
\nonumber\\
&+
2C^{3}_{g}|||y|||_{
p-\text{var},[s,t]}
|||B^{H,\lambda;1}|||_{p-\text{var},[s,t]}
|||B^{H,\lambda;2}|||_{\frac{p}{2}-\text{var},[s,t]^{2}}\nonumber\\
&+2C^{4}_{g}
|||B^{H,\lambda;2}|||^{2}_{\frac{p}{2}-\text{var},[s,t]^{2}}
+2C^{2}_{g}|||R^{y\sharp}|||
_{\frac{p}{3}-\text{var},[s,t]^{2}}
|||B^{H,\lambda;2}|||_{\frac{p}{2}-\text{var},[s,t]^{2}}
\nonumber\\
&+6C_{g}^{3}|||y|||_{p-\text{var},[s,t]}
|||B^{H,\lambda;3}|||_{\frac{p}{3}-\text{var},[s,t]^{2}}
\bigg)\nonumber\\
&+C_{g}|||B^{H,\lambda;1}|||_{p-\text{var},[s,t]}
+ C^{2}_{g}  |||B^{H,\lambda;2}|||_{\frac{p}{2}-\text{var},[s,t]^{2}}
+2C^{3}_{g} |||B^{H,\lambda;3}|||_{\frac{p}{3}-\text{var},[s,t]^{2}}.
\end{align}
Consequently, this leads to the $p$-variation norm
\begin{align}\label{eq4.16}
&|||y|||_{p-\text{var},[s,t]}\nonumber\\
&\leq
C_p\bigg(
\frac{1}{2}C_{g}^{3}|||y|||_{p-\text{var},[s,t]}
|||B^{H,\lambda;2}|||_{\frac{p}{2}-\text{var},[s,t]^{2}}
||| B^{H,\lambda;1} |||_{p-\text{var},[s,t]}
\nonumber\\
&+
C_{g}
|||R^{y\sharp}|||_{\frac{p}{3}-\text{var},[s,t]^{2}}
||| B^{H,\lambda;1} |||_{p-\text{var},[s,t]}
+C_{g}^{2}(t-s)^{\frac{2}{p}}
||| B^{H,\lambda;1} |||^{2}_{p-\text{var},[s,t]}
|||y|||_{p-\text{var},[s,t]}
\nonumber\\
&+\frac{1}{2}C_{g}^{4}
|||B^{H,\lambda;2}|||_{\frac{p}{2}-\text{var},[s,t]^{2}}
|||B^{H,\lambda;1}|||^{2}_{p-\text{var},[s,t]}
+\frac{1}{2}C_{g}^{2}
|||B^{H,\lambda;1}|||^{2}_{p-\text{var},[s,t]^{2}}
|||R^{y\sharp}|||_{\frac{p}{3}-\text{var},[s,t]^{2}}
\nonumber\\
&+
2C^{3}_{g}|||y|||_{
p-\text{var},[s,t]}
|||B^{H,\lambda;1}|||_{p-\text{var},[s,t]}
|||B^{H,\lambda;2}|||_{\frac{p}{2}-\text{var},[s,t]^{2}}\nonumber\\
&+2C^{4}_{g}
|||B^{H,\lambda;2}|||^{2}_{\frac{p}{2}-\text{var},[s,t]^{2}}
+2C^{2}_{g}|||R^{y\sharp}|||
_{\frac{p}{3}-\text{var},[s,t]^{2}}
|||B^{H,\lambda;2}|||_{\frac{p}{2}-\text{var},[s,t]^{2}}
\nonumber\\
&+6C_{g}^{3}|||y|||_{p-\text{var},[s,t]}
|||B^{H,\lambda;3}|||_{\frac{p}{3}-\text{var},[s,t]^{2}}
\bigg)\nonumber\\
&+C_{g}|||B^{H,\lambda;1}|||_{p-\text{var},[s,t]}
+ C^{2}_{g}  |||B^{H,\lambda;2}|||_{\frac{p}{2}-\text{var},[s,t]^{2}}
+2C^{3}_{g} |||B^{H,\lambda;3}|||_{\frac{p}{3}-\text{var},[s,t]^{2}}.
\end{align}
Combining \eqref{eq0-2.11} and \eqref{eq0-2.17} yields,
for $R^{y\sharp}$, the estimate
\begin{align}\label{eq4.17}
&\|R^{y\sharp}_{s,t}\|
=\|y_{s,t}-y'_s B^{H,\lambda;1}_{s,t}-y''_s B^{H,\lambda;2}_{s,t}\|\nonumber\\
&=\bigg\|\int_{s}^{t}
g(y_{u})db^{H,\lambda}_{u}
-g(y_{s})B^{H,\lambda;1}_{s,t}
-[g(y)]'_{s} B^{H,\lambda;2}_{s,t}\bigg\|\nonumber\\
&\leq C_{p}\bigg(
||| B^{H,\lambda;1} |||_{p-\text{var},[s,t]} ||| R^{g(y)\sharp} |||_{\frac{p}{3}-\text{var},[s,t]^2} +
    ||| B^{H,\lambda;2} |||_{\frac{p}{2}-\text{var},[s,t]^2} ||| R^{g(y)\sharp\sharp} |||_{\frac{p}{2}-\text{var},[s,t]^2}
    \nonumber\\
    &+
    ||| [g(y)]''_{s}|||_{p-\text{var},[s,t]}
    ||| B^{H,\lambda;3} |||_{\frac{p}{3}-\text{var},[s,t]^2} \bigg)
+\|[g(y)]''_{s} \| \|B^{H,\lambda;3}_{s,t} \|.
\end{align}
This implies that $|||R^{y\sharp}|||_{\frac{p}{3}-\text{var},[s,t]^2}$ shares the same estimate as $|||y|||_{p-\text{var},[s,t]}$.
By virtue of \eqref{eq0-2.11} and \eqref{eq0-2.11-1}, we find that
\begin{align}\label{eq4.18}
R^{y\sharp \sharp}_{s,t}
&=g(y_{t})-g(y_{s})-Dg(y_{s})g(y_{s})B^{H,\lambda;1}_{s,t}
\nonumber\\
&=\int_{0}^{1}
Dg(y_{s}+\eta y_{s,t})y_{s,t}d\eta-
Dg(y_{s})g(y_{s})B^{H,\lambda;1}_{s,t}\nonumber\\
&=\int_{0}^{1}
[Dg(y_{s}+\eta y_{s,t})-Dg(y_{s})]
g(y_{s})B^{H,\lambda;1}_{s,t}d\eta
\nonumber\\
&+\int_{0}^{1}Dg(y_{s}+\eta y_{s,t})
y''_s B^{H,\lambda;2}_{s,t}d\eta
+\int_{0}^{1}Dg(y_{s}+\eta y_{s,t})
R^{y\sharp}_{s,t}d\eta\nonumber\\
&=\int_{0}^{1}\int_{0}^{1}
D^{2}g(y_{s}+\xi\eta y_{s,t})\eta y_{s,t}d\xi
g(y_{s})B^{H,\lambda;1}_{s,t}d\eta
\nonumber\\
&+\int_{0}^{1}Dg(y_{s}+\eta y_{s,t})
Dg(y_{s})g(y_{s}) B^{H,\lambda;2}_{s,t}d\eta
+\int_{0}^{1}Dg(y_{s}+\eta y_{s,t})
R^{y\sharp}_{s,t}d\eta,
\end{align}
and consequently
\begin{align}\label{eq4.19}
|||R^{y\sharp \sharp}\|_{\frac{p}{2}-\text{var},[s,t]^{2}}
&\leq \frac{1}{2}
C_{g}^{2}
|||y|||_{p-\text{var},[s,t]}
|||B^{H,\lambda;1}_{s,t}|||_{p-\text{var},[s,t]}\nonumber\\
&+C_{g}^{3}|||B^{H,\lambda;2}|||_{\frac{p}{2}-\text{var},[s,t]^{2}}
+C_{g} |||R^{y\sharp}|||_{\frac{p}{3}-\text{var},[s,t]^{2}}.
\end{align}
Finally, combining \eqref{eq4.16} and \eqref{eq4.19}, we conclude that
\begin{align}\label{eq4.20}
& ||| y, R^{y\sharp}, R^{y\sharp\sharp}|||_{p-\text{var},[s,t]}\nonumber\\
&= ||| y |||_{p-\text{var},[s,t]} +
||| R^{y\sharp} |||_{\frac{p}{3}-\text{var},[s,t]^2}
+||| R^{y\sharp\sharp} |||_{\frac{p}{2}-\text{var},[s,t]^2}
\nonumber\\
&\leq\Theta(1+||| y, R^{y\sharp}, R^{y\sharp\sharp}|||_{p-\text{var},[s,t]}
),
\end{align}
where
\begin{align}
\Theta&:=\Theta_{1}
(1+|||\mathbf{B}^{H,\lambda}|||_{p-\text{var},[s,t]}
\vee|||\mathbf{B}^{H,\lambda}|||^{2}_{p-\text{var},[s,t]}
\vee|||\mathbf{B}^{H,\lambda}|||^{3}_{p-\text{var},[s,t]}),
\nonumber
\end{align}
and
\begin{align}
\Theta_{1}:=C_{p}[2C_{g}+(5+2(b-a)^{\frac{2}{p}})
C^{2}_{g}+17C^{3}_{g}+5C^{4}_{g}]
+2C_{g}+3C^{2}_{g}+5C^{3}_{g}.
\nonumber
\end{align}
If $\Theta\leq \frac{1}{2}$ or $|||\mathbf{B}^{H,\lambda}|||_{p-\text{var},[s,t]}\leq \min
\{\frac{1}{2\Theta_{1}}-1, (\frac{1}{2\Theta_{1}}-1)^{\frac{1}{3}}\}$, then
\begin{align}\label{eq4.22}
||| y, R^{y\sharp}, R^{y\sharp\sharp}|||_{p-\text{var},[s,t]}\leq
\frac{\Theta}{1-\Theta}.
\end{align}
By employing an argument similar to that in \cite[Theorem 2.4]{Duc-Hong} and constructing a greedy sequence of stopping times $\bigg\{ \tau_i\bigg( \min
\{\frac{1}{2\Theta_{1}}-1, (\frac{1}{2\Theta_{1}}-1)^{\frac{1}{3}}\}, [a,b], p \bigg) \bigg\}_{i \in \mathbb{N}}$, we conclude that
\begin{align}\label{eq4.23}
\| y \|_{\infty,[a,b]} &\leq \| y_a \| + \frac{\Theta}{1-\Theta} N_{[a,b]}(\mathbf{B}^{H,\lambda}) \nonumber\\
&\leq \| y_a \| + \frac{\Theta}{1-\Theta} \bigg[ 1 + \bigg(
\min\{\frac{1}{2\Theta_{1}}-1, (\frac{1}{2\Theta_{1}}-1)^{\frac{1}{3}}\}\bigg)^{-p}
 \| \mathbf{B}^{H,\lambda} \|^p_{p-\text{var},[a,b]} \bigg],
\end{align}
and
\begin{align}\label{eq4.24}
&||| y, R^{y\sharp}, R^{y\sharp\sharp}|||_{p-\text{var},[s,t]}
\nonumber\\
&\leq N_{[a,b]}^{\frac{p-1}{p}}(\mathbf{B}^{H,\lambda}) \sum_{k=0}^{N_{[a,b]}(\mathbf{B}^{H,\lambda})-1}
||| y, R^{y\sharp}, R^{y\sharp\sharp}|||_{p-\text{var},[\tau_k, \tau_{k+1}]}\nonumber\\
&\leq\frac{\Theta}{1-\Theta}
N_{[a,b]}^{\frac{2p-1}{p}}(\mathbf{B}^{H,\lambda})\nonumber\\
&\leq\frac{\Theta}{1-\Theta}
2^{\frac{p-1}{p}}
\bigg[ 1 + \bigg(
\min\{\frac{1}{2\Theta_{1}}-1, (\frac{1}{2\Theta_{1}}-1)^{\frac{1}{3}}\}\bigg)^{1-2p}
 ||| \mathbf{B}^{H,\lambda} |||^{2p-1}_{p-\text{var},[a,b]} \bigg],
\end{align}
where in the last inequality we have used \eqref{eq0-2.19}.

\noindent \textbf{Step 2:}
Let us consider two solutions, denoted by  $y_t(\mathbf{B}^{H,\lambda}, y_a)$ and $\bar{y}_t(\mathbf{B}^{H,\lambda}, \bar{y}_a)$, which remain within the bounded range $\frac{\Theta}{1-\Theta}
N_{[a,b]}^{\frac{2p-1}{p}}(\mathbf{B}^{H,\lambda})$. The difference  $z_t = \bar{y}_t - y_t$ then satisfies the rough integral equation
\begin{align}\label{eq4.25}
z_t = z_a + \int_a^t  g(\bar{y}_s) - g(y_s) db^{H,\lambda}_s,
\end{align}
with the initial value $z_a=\bar{y}_a - y_a$.
Applying \eqref{eq4.6}, we find that
\begin{align}\label{eq4.26}
R^{(g(\bar{y})-g(y))\sharp}_{s,t}&=
\int_0^1 \int_0^1
D^{2}g(\bar{y}_s + \eta\xi \bar{y}_{s,t})\eta \bar{y}_{s,t}d\xi
\bar{y}''_s \otimes B^{H,\lambda;2}_{s,t} d\eta
\nonumber\\
&-
\int_0^1 \int_0^1
D^{2}g(y_s + \eta\xi y_{s,t})\eta y_{s,t}d\xi
y''_s \otimes B^{H,\lambda;2}_{s,t} d\eta
\nonumber\\
&+
\int_0^1 Dg(\bar{y}_s + \eta \bar{y}_{s,t}) R^{\bar{y}\sharp}_{s,t} d\eta
-
\int_0^1 Dg(y_s + \eta y_{s,t}) R^{y\sharp}_{s,t} d\eta
\nonumber\\
&+\int_0^1\int_0^1[
D^{2}g(\bar{y}_s + \xi\eta \bar{y}_{s,t}) - D^{2}g(\bar{y}_s) ]
\eta \bar{y}_{s,t}d\xi \bar{y}'_s\otimes B^{H,\lambda;1}_{s,t}d\eta
\nonumber\\
&-
\int_0^1\int_0^1[
D^{2}g(y_s + \xi\eta y_{s,t}) - D^{2}g(y_s) ]
\eta y_{s,t}d\xi y'_s\otimes B^{H,\lambda;1}_{s,t}d\eta
\nonumber\\
&+\frac{1}{2}D^{2}g(\bar{y}_s)\bar{y}''_s
\bar{y}'_s
B^{H,\lambda;2}_{s,t}\otimes B^{H,\lambda;1}_{s,t}
-
\frac{1}{2}D^{2}g(y_s)y''_sy'_s
B^{H,\lambda;2}_{s,t}\otimes B^{H,\lambda;1}_{s,t}
\nonumber\\
&
+\frac{1}{2}D^{2}g(\bar{y}_s)R^{\bar{y}\sharp}_{s,t}
\bar{y}'_s\otimes B^{H,\lambda;1}_{s,t}
-
\frac{1}{2}D^{2}g(y_s)R^{y\sharp}_{s,t}
y'_s\otimes B^{H,\lambda;1}_{s,t}\nonumber\\
&:=II_{1}+II_{2}+II_{3}+II_{4}+II_{5}.
\end{align}
Thanks to $y'_{s}=g(y_{s})$ and $y''_{s}=[g(y)]'_s = Dg(y_s)g(y_s)$, we define $Q(y_s):=Dg(y_s)g(y_s)$.
Then we have $\|Q(\bar{y}_s)-Q(y_s)\|\leq 2C_{g}^{2}\|z_{s}\|$ and
\begin{align}\label{eq4.26-1}
||| Q(\bar{y}) - Q(y) |||_{p-\text{var},[s,t]}
\leq 2 C_g^2 \left( ||| z |||_{p-\text{var},[s,t]} + || z ||_{\infty,[s,t]} ||| y |||_{p-\text{var},[s,t]} \right).
\end{align}
It is straightforward to verify that
\begin{align}\label{eq4.27}
II_{1}&=
\int_0^1 \int_0^1
\bigg(D^{2}g(\bar{y}_s + \eta\xi \bar{y}_{s,t})
-D^{2}g(y_s + \eta\xi y_{s,t})\bigg)
\eta \bar{y}_{s,t}d\xi
\bar{y}''_s \otimes B^{H,\lambda;2}_{s,t} d\eta
\nonumber\\
&+
\int_0^1 \int_0^1
D^{2}g(y_s + \eta\xi y_{s,t})
\eta(\bar{y}_{s,t}-y_{s,t})d\xi
\bar{y}''_s \otimes B^{H,\lambda;2}_{s,t} d\eta
\nonumber\\
&+
\int_0^1 \int_0^1
D^{2}g(y_s + \eta\xi y_{s,t})
\eta y_{s,t}d\xi
(\bar{y}''_s-y''_s)
\otimes B^{H,\lambda;2}_{s,t} d\eta
\nonumber\\
&=
\int_0^1 \int_0^1\int_0^1
D^{3}g\bigg(
(y_s + \eta\xi y_{s,t})+\tau
((\bar{y}_s + \eta\xi \bar{y}_{s,t})-
(y_s + \eta\xi y_{s,t}))\bigg)d\tau
\nonumber\\
&\times\bigg((\bar{y}_s + \eta\xi \bar{y}_{s,t})-
(y_s + \eta\xi y_{s,t})
\bigg)
\eta \bar{y}_{s,t}d\xi
Dg(\bar{y}_s)g(\bar{y}_s) \otimes B^{H,\lambda;2}_{s,t} d\eta
\nonumber\\
&+
\int_0^1 \int_0^1
D^{2}g(y_s + \eta\xi y_{s,t})
\eta z_{s,t}
d\xi Dg(\bar{y}_s)g(\bar{y}_s)
\otimes B^{H,\lambda;2}_{s,t} d\eta
\nonumber\\
&+
\int_0^1 \int_0^1
D^{2}g(y_s + \eta\xi y_{s,t})
\eta y_{s,t}d\xi
\bigg(Dg(\bar{y}_s)g(\bar{y}_s)
-Dg(y_s)g(y_s)\bigg)
\otimes B^{H,\lambda;2}_{s,t} d\eta,
\end{align}
which yields the estimate
\begin{align}\label{eq4.28}
II_{1}&\leq C^{3}_{g}
\bigg(\frac{1}{2}\|z_{s}\|+
\frac{1}{6}\|z_{s,t}\|\bigg)
\|\bar{y}_{s,t}\|
\|B^{H,\lambda;2}_{s,t}\|
+\frac{1}{2}C^{3}_{g}\|z_{s,t}\|
\|B^{H,\lambda;2}_{s,t}\|\nonumber\\
&
+\frac{1}{2}C_{g}\|y_{s,t}\|
\|Q(\bar{y}_s)-Q(y_s)\|\|B^{H,\lambda;2}_{s,t}\|.
\end{align}
Similarly, we deduce that
\begin{align}\label{eq4.29}
II_{2}&\leq
C_{g}\|z_{s}\|\|R_{s,t}^{\bar{y}\sharp}\|
+\frac{1}{2}C_{g}
\|z_{s,t}\|\|R_{s,t}^{\bar{y}\sharp}\|
+C_{g}\|R_{s,t}^{z\sharp}\|,
\end{align}
\begin{align}\label{eq4.30}
II_{3}&\leq
C^{2}_{g}\|\bar{y}_{s,t}\|
\|B^{H,\lambda;1}_{s,t}\|
\|z\|_{\infty}
+
C^{2}_{g}\|y_{s,t}\|
\|B^{H,\lambda;1}_{s,t}\|
\|z\|_{\infty}
\nonumber\\
&+C^{2}_{g}\|z_{s,t}\|\|B^{H,\lambda;1}_{s,t}\|
+\frac{1}{6}C^{2}_{g}
\|\bar{y}_{s,t}\|\|B^{H,\lambda;1}_{s,t}\|
\|z_{s,t}\|,
\end{align}
\begin{align}\label{eq4.31}
II_{4}&\leq
C^{4}_{g}\|z_{s}\|
\|B^{H,\lambda;2}_{s,t}\|\|B^{H,\lambda;1}_{s,t}\|
+\frac{1}{2}C^{2}_{g}
\|Q(\bar{y}_{s})-Q(y_{s})\|
\|B^{H,\lambda;2}_{s,t}\|\|B^{H,\lambda;1}_{s,t}\|,
\end{align}
and
\begin{align}\label{eq4.32}
II_{5}&\leq
\frac{1}{2}C^{2}_{g}
\|B^{H,\lambda;1}_{s,t}\|
\|R_{s,t}^{\bar{y}\sharp}\|\|z_{s}\|
+\frac{1}{2}C^{2}_{g}
\|B^{H,\lambda;1}_{s,t}\|
\|R_{s,t}^{z\sharp}\|
+\frac{1}{2}C^{2}_{g}
\|R_{s,t}^{y\sharp}\|
\|B^{H,\lambda;1}_{s,t}\|\|z_{s}\|,
\end{align}
Using \eqref{eq4.26} and \eqref{eq4.28}-\eqref{eq4.32} results in
\begin{align}\label{eq4.33}
&|||R^{(g(\bar{y})-g(y))\sharp}|||
_{\frac{p}{3}-\text{var},[s,t]^{2}}\leq C^{3}_{g}|||z|||
_{p-\text{var},[s,t]}
|||\bar{y}|||
_{p-\text{var},[s,t]}
|||B^{H,\lambda;2}|||
_{\frac{p}{2}-\text{var},[s,t]^{2}}
\nonumber\\
&+\frac{1}{2}C^{3}_{g}
|||z|||
_{p-\text{var},[s,t]}
|||B^{H,\lambda;2}|||
_{\frac{p}{2}-\text{var},[s,t]^{2}}
+C^{3}_{g}
|||y|||
_{p-\text{var},[s,t]}
|||B^{H,\lambda;2}|||
_{\frac{p}{2}-\text{var},[s,t]^{2}}
||z||_{\infty}
\nonumber\\
&+C_{g}\|z\|_{\infty}
|||R^{\bar{y}\sharp}|||
_{\frac{p}{3}-\text{var},[s,t]^{2}}
+\frac{1}{2}C_{g}
|||z|||
_{p-\text{var},[s,t]}
|||R^{\bar{y}\sharp}|||
_{\frac{p}{3}-\text{var},[s,t]^{2}}+
C_{g}
|||R^{z\sharp}|||
_{\frac{p}{3}-\text{var},[s,t]^{2}}
\nonumber\\
&+
C^{2}_{g}(t-s)^{\frac{2}{p}}
|||\bar{y}|||_{p-\text{var},[s,t]}
||z||_{\infty}
|||B^{H,\lambda;1}|||
_{p-\text{var},[s,t]}
\nonumber\\
&+C^{2}_{g}(t-s)^{\frac{2}{p}}
|||z|||
_{p-\text{var},[s,t]}
|||B^{H,\lambda;1}|||
_{p-\text{var},[s,t]}
\nonumber\\
&+C^{2}_{g}(t-s)^{\frac{2}{p}}
||z||_{\infty}
|||y|||
_{p-\text{var},[s,t]}
|||B^{H,\lambda;1}|||
_{p-\text{var},[s,t]}
\nonumber\\
&+\frac{1}{6}C^{2}_{g}
|||\bar{y}|||^{2}
_{p-\text{var},[s,t]}
|||z|||
_{p-\text{var},[s,t]}
|||B^{H,\lambda;1}|||
_{p-\text{var},[s,t]}
\nonumber\\
&+C^{4}_{g}
\|z\|_{\infty}
|||B^{H,\lambda;1}|||
_{p-\text{var},[s,t]}
|||B^{H,\lambda;2}|||
_{\frac{p}{2}-\text{var},[s,t]^{2}}
\nonumber\\
&+
2C^{4}_{g}
|||B^{H,\lambda;1}|||
_{p-\text{var},[s,t]}
|||B^{H,\lambda;2}|||
_{\frac{p}{2}-\text{var},[s,t]^{2}}
\|z\|_{\infty}
\nonumber\\
&+
\frac{1}{2}C^{2}_{g}
|||B^{H,\lambda;1}|||
_{p-\text{var},[s,t]}
|||R^{\bar{y}\sharp}|||
_{\frac{p}{3}-\text{var},[s,t]^{2}}\|z\|_{\infty}
\nonumber\\
&+\frac{1}{2}C^{2}_{g}
|||B^{H,\lambda;1}|||
_{p-\text{var},[s,t]}
|||R^{z\sharp}|||
_{\frac{p}{3}-\text{var},[s,t]^{2}}
\nonumber\\
&+\frac{C^{2}_{g}}{2}
|||R^{y\sharp}|||
_{\frac{p}{3}-\text{var},[s,t]^{2}}
|||B^{H,\lambda;1}|||
_{p-\text{var},[s,t]}
\|z\|_{\infty}.
\end{align}
Note from \eqref{eq4.12} that
\begin{align}\label{eq4.34}
R^{(g(\bar{y})-g(y))\sharp\sharp}_{s,t}
&=
\int_{0}^{1}
[D^{2}g(\bar{y}_{s}+\eta \bar{y}_{s,t} )-D^{2}g(\bar{y}_{s})]
g(\bar{y}_{t})g(\bar{y}_{s})B^{H,\lambda;1}_{s,t}d\eta\nonumber\\
&-
\int_{0}^{1}
[D^{2}g(y_{s}+\eta y_{s,t} )-D^{2}g(y_{s})]
g(y_{t})g(y_{s})B^{H,\lambda;1}_{s,t}d\eta\nonumber\\
&+D^{2}g(\bar{y}_{s})g(\bar{y}_{s})
[g(\bar{y}_{t} )-g(\bar{y}_{s})] B^{H,\lambda;1}_{s,t}
-
D^{2}g(y_{s})g(y_{s})
[g(y_{t} )-g(y_{s})] B^{H,\lambda;1}_{s,t}\nonumber\\
&+\int_{0}^{1}D^{2}g(\bar{y}_{s}+\eta \bar{y}_{s,t} )
(\bar{y}''_s B^{H,\lambda;2}_{s,t}+
    R^{\bar{y}\sharp}_{s,t})d\eta g(\bar{y}_{s})
\nonumber\\
&-
\int_{0}^{1}D^{2}g(y_{s}+\eta y_{s,t} )
(y''_s B^{H,\lambda;2}_{s,t}+
    R^{y\sharp}_{s,t})d\eta g(y_{s})
\nonumber\\
&+\int_{0}^{1}
[Dg(\bar{y}_{s}+\eta \bar{y}_{s,t} )
-Dg(\bar{y}_{s})]Dg(\bar{y}_{s})g(\bar{y}_{s})
B^{H,\lambda;1}_{s,t}d\eta
\nonumber\\
&-
\int_{0}^{1}
[Dg(y_{s}+\eta y_{s,t} )
-Dg(y_{s})]Dg(y_{s})g(y_{s})
B^{H,\lambda;1}_{s,t}d\eta\nonumber\\
&+
\int_{0}^{1}Dg(\bar{y}_{s}+\eta \bar{y}_{s,t} )
(\bar{y}''_s B^{H,\lambda;2}_{s,t}+
    R^{\bar{y}\sharp}_{s,t})d\eta Dg(\bar{y}_{s})
\nonumber\\
&-
\int_{0}^{1}Dg(y_{s}+\eta y_{s,t} )
(y''_s B^{H,\lambda;2}_{s,t}+
    R^{y\sharp}_{s,t})d\eta Dg(y_{s})\nonumber\\
&:=II_{6}+II_{7}+II_{8}+II_{9}+II_{10}.
\end{align}
In a similar way as \eqref{eq4.27}, we estimate
\begin{align}\label{eq4.35}
II_{6}\leq 4C^{3}_{g}
\|B^{H,\lambda;1}_{s,t}\|\|z_{s}\|
+2C^{3}_{g}
\|B^{H,\lambda;1}_{s,t}\|\|z_{t}\|
+
\frac{1}{2}C^{3}_{g}\|B^{H,\lambda;1}_{s,t}\|\|z_{s,t}\|
\end{align}
\begin{align}\label{eq4.35}
II_{7}\leq 3C^{3}_{g}
\|z_{s}\|\|\bar{y}_{s,t}\|\|B^{H,\lambda;1}_{s,t}\|
+C^{3}_{g}\|B^{H,\lambda;1}_{s,t}\|z_{s,t}\|
(\frac{1}{2}\|\bar{y}_{s,t}\|+1),
\end{align}
\begin{align}\label{eq4.36}
II_{8}+II_{10}&\leq
2C^{2}_{g}(\|z_{s}\|+\frac{1}{2}\|z_{s,t}\|)
(C^{2}_{g}\|B^{H,\lambda;2}_{s,t}\|+
\|R^{\bar{y}\sharp}_{s,t}\|)\nonumber\\
&
+2C^{2}_{g}(2C^{2}_{g}\|z_{s}\|\|B^{H,\lambda;2}_{s,t}\|
+\|R^{(\bar{y}-y)\sharp}_{s,t}\|
)
+2C^{2}_{g}(C^{2}_{g}\|B^{H,\lambda;2}_{s,t}\|+
\|R^{y\sharp}_{s,t}\|
)\|z_{s}\|,
\end{align}
and
\begin{align}\label{eq4.37}
II_{9}
&\leq
C^{3}_{g}\|\bar{y}_{s,t}\|\|B^{H,\lambda;1}_{s,t}\|
(\frac{1}{2}\|z\|_{\infty}+\frac{1}{6}\|z_{s,t}\|)
\nonumber\\
&+\frac{1}{2}C^{3}_{g}\|z_{s,t}\|\|B^{H,\lambda;1}_{s,t}\|
+C^{3}_{g}\|y_{s,t}\|\|B^{H,\lambda;1}_{s,t}\|
\|z\|_{\infty}.
\end{align}
Therefore, we conclude that
\begin{align}\label{eq4.37-200}
&|||R^{(g(\bar{y})-g(y))\sharp\sharp}|||
_{\frac{p}{2}-\text{var},[s,t]^{2}}\nonumber\\
&\leq
8C^{3}_{g}
|||z|||_{p-\text{var},[s,t]}
|||B^{H,\lambda;1}|||_{p-\text{var},[s,t]}
+
4C^{3}_{g}
|||\bar{y}|||
_{p-\text{var},[s,t]}
|||B^{H,\lambda;1}|||_{p-\text{var},[s,t]}
\|z\|_{\infty}
\nonumber\\
&
+C^{3}_{g}|||\bar{y}|||_{p-\text{var},[s,t]}
|||B^{H,\lambda;1}|||_{p-\text{var},[s,t]}
|||z|||_{p-\text{var},[s,t]}
+
8C^{4}_{g}\|z\|_{\infty}|||B^{H,\lambda;2}|||
_{\frac{p}{2}-\text{var},[s,t]^{2}}
\nonumber\\
&+2C^{2}_{g}
\|z\|_{\infty}|||R^{\bar{y}\sharp}|||
_{\frac{p}{3}-\text{var},[s,t]^{2}}
+
C^{4}_{g}
|||z|||_{p-\text{var},[s,t]}
|||B^{H,\lambda;2}|||
_{\frac{p}{2}-\text{var},[s,t]^{2}}
\nonumber\\
&+C^{2}_{g}
|||z|||_{p-\text{var},[s,t]}
|||R^{\bar{y}\sharp}|||
_{\frac{p}{3}-\text{var},[s,t]^{2}}
+2C^{2}_{g}
|||R^{z\sharp}|||
_{\frac{p}{3}-\text{var},[s,t]^{2}}
\nonumber\\
&+
2C^{2}_{g}
|||R^{y\sharp}|||
_{\frac{p}{3}-\text{var},[s,t]^{2}}
\|z\|_{\infty}
+C_g^3
|||y|||_{p-\text{var},[s,t]}
|||B^{H,\lambda;1}|||_{p-\text{var},[s,t]}
\|z\|_{\infty}
.
\end{align}
Notice that
\begin{align}\label{eq4.39}
[g(\bar{y})-g(y)]''_{s,t}
&=\bigg[
\bigg(Dg(\bar{y}_t)Dg(\bar{y}_t)g(\bar{y}_t)
+D^{2}g(\bar{y}_t)g^{2}(\bar{y}_t)\bigg)
\nonumber\\
&~~-
\bigg(Dg(y_t)Dg(y_t)g(y_t)
+D^{2}g(y_t)g^{2}(y_t)\bigg)\bigg]
\nonumber\\
&-
\bigg[
\bigg(Dg(\bar{y}_s)Dg(\bar{y}_s)g(\bar{y}_s)
+D^{2}g(\bar{y}_s)g^{2}(\bar{y}_s)\bigg)
\nonumber\\
&~~-
\bigg(Dg(y_s)Dg(y_s)g(y_s)
+D^{2}g(y_s)g^{2}(y_s)\bigg)\bigg].
\end{align}
Following an argument similar to that used in \eqref{eq4.27}, we obtain the estimate
\begin{align}\label{eq4.40}
&|||[g(\bar{y})-g(y)]''|||
_{p-\text{var},[s,t]}\nonumber\\
&\leq 12C^{3}_{g}||| z |||_{p-\text{var},[s,t]}
+6C^{3}_{g}
\|z\|_{\infty}
(|||\bar{y}|||_{p-\text{var},[s,t]}
+|||y|||_{p-\text{var},[s,t]}).
\end{align}
Recalling \eqref{eq0-2.17}, we observe that
\begin{align}\label{eq4.41}
\|z_{s,t}\|&=\bigg\|\int_s^t  g(\bar{y}_u) - g(y_u) db_u^{H,\lambda}\bigg\|\nonumber\\
&\leq C_p \bigg( ||| B^{H,\lambda;1} |||_{p-\text{var},[s,t]}
||| R^{(g(\bar{y}) - g(y))\sharp} |||_{\frac{p}{3}-\text{var},[s,t]^2} \nonumber\\
&
+
||| B^{H,\lambda;2} |||_{\frac{p}{2}-\text{var},[s,t]^2}
||| R^{(g(\bar{y}) - g(y))\sharp\sharp} |||_{\frac{p}{2}-\text{var},[s,t]^2}
\nonumber\\
&+
||| [g(\bar{y})-g(y)]'' |||_{p-\text{var},[s,t]}
||| B^{H,\lambda;3} |||_{\frac{p}{3}-\text{var},[s,t]^2} \bigg)\nonumber\\
&+\|g(\bar{y}_{s})-g(y_{s})\|\|B^{H,\lambda;1}_{s,t}\|
+\|[g(\bar{y})-g(y)]'_{s}\|\|B^{H,\lambda;2}_{s,t}\|
+\|[g(\bar{y})-g(y)]''_{s}\|\|B^{H,\lambda;3}_{s,t}\|.
\end{align}
Substituting \eqref{eq4.33}, \eqref{eq4.37-200} and \eqref{eq4.40} into \eqref{eq4.41} yields
\begin{align}\label{eq4.42}
||| z |||_{p-\text{var},[s,t]}
&\leq \Theta_{2}
\bigg[ \| z \|_{\infty,[s,t]}
+||| z, R^{z\sharp},R^{z\sharp\sharp}|||_{p-\text{var},[s,t]}
\bigg]\nonumber\\
&\leq 2\Theta_{2}
\bigg[ \| z_{s} \|
+||| z, R^{z\sharp},R^{z\sharp\sharp}|||_{p-\text{var},[s,t]}
\bigg],
\end{align}
where
\begin{align}\label{eq4.42-1}
\Theta_{2}&:=\Theta_{3}
(||| \mathbf{B}^{H,\lambda} |||_{p-\text{var},[s,t]}
\vee ||| \mathbf{B}^{H,\lambda} |||^{2}_{p-\text{var},[s,t]}
\vee ||| \mathbf{B}^{H,\lambda} |||^{3}_{p-\text{var},[s,t]} ),
\end{align}
\begin{align}\label{eq4.42-2}
\Theta_{3}&:=
MC_{p}[C_{g}+(1+(t-s)^{\frac{2}{p}})
C^{2}_{g}+C^{3}_{g}+C^{4}_{g}]
\bigg[1+||| y, R^{y\sharp},R^{y\sharp\sharp}|||_{p-\text{var},[s,t]}
\nonumber\\
&
+||| \bar{y}, R^{\bar{y}\sharp},R^{\bar{y}\sharp\sharp}
|||_{p-\text{var},[s,t]}
+||| \bar{y}, R^{\bar{y}\sharp},R^{\bar{y}\sharp\sharp}
|||_{p-\text{var},[s,t]}^{2}
\bigg]
\end{align}
for some positive integer $M$,
and
\begin{align}
&\max\bigg\{||| y, R^{y\sharp}, R^{y\sharp\sharp}|||_{p-\text{var},[s,t]}
\vee
||| \bar{y}, R^{\bar{y}\sharp},R^{\bar{y}\sharp\sharp}
|||_{p-\text{var},[s,t]}\bigg\}
\nonumber\\
&\leq\frac{\Theta}{1-\Theta}
2^{\frac{p-1}{p}}
\bigg[ 1 + \bigg(
\min\{\frac{1}{2\Theta_{1}}-1, (\frac{1}{2\Theta_{1}}-1)^{\frac{1}{3}}\}\bigg)^{1-2p}
 ||| \mathbf{B}^{H,\lambda} |||^{2p-1}_{p-\text{var},[a,b]} \bigg],
 \nonumber
\end{align}
due to \eqref{eq4.24}.
From \eqref{eq0-2.11} and \eqref{eq0-2.15}, we observe that $\| R^{z\sharp} \|_{q-\text{var},[s,t]}$ admits an estimate similar to that of $||| z |||_{p-\text{var},[s,t]}$.
For $R^{z\sharp\sharp}$, we have
\begin{align}\label{eq4.43}
R^{z\sharp\sharp}_{s,t}
&=\|z'_{s,t}-z''_{s}B^{H,\lambda;1}_{s,t}\|=
\|(\bar{y}-y)'_{s,t}-(\bar{y}-y)''_{s}B^{H,\lambda;1}_{s,t}\|
\nonumber\\
&=\|(g(\bar{y})-g(y))_{s,t}-(g(\bar{y})-g(y))
'_{s}B^{H,\lambda;1}_{s,t}\|\nonumber\\
&=\|(g(\bar{y})-g(y))''_{s}B^{H,\lambda;2}_{s,t}
+R^{(g(\bar{y})-g(y))\sharp}_{s,t}\|
\nonumber\\
&\leq \|(g(\bar{y})-g(y))''_{s}\|\|B^{H,\lambda;2}_{s,t}\|
+\|R^{(g(\bar{y})-g(y))\sharp}_{s,t}\|.
\end{align}
Following an argument analogous to \eqref{eq4.39}, we obtain that
\begin{align}\label{eq4.44}
\|(g(\bar{y})-g(y))''_{s}\|
\leq 6C_{g}^{3}\|z_{s}\|,
\end{align}
which, together with \eqref{eq4.33}, yields the estimate for  $|||R^{z\sharp\sharp}|||
_{\frac{p}{2}-\text{var},[s,t]^{2}}$. Consequently, we conclude that
\begin{align}\label{eq4.45}
||| z, R^{z\sharp},R^{z\sharp\sharp}|||_{p-\text{var},[s,t]}
&\leq 6\Theta_{2}
\bigg[ \| z_{s} \|
+||| z, R^{z\sharp},R^{z\sharp\sharp}|||_{p-\text{var},[s,t]}
\bigg].
\end{align}
When $12\Theta_{3}
(||| \mathbf{B}^{H,\lambda} |||_{p-\text{var},[s,t]}
\vee ||| \mathbf{B}^{H,\lambda} |||^{2}_{p-\text{var},[s,t]}
\vee ||| \mathbf{B}^{H,\lambda} |||^{3}_{p-\text{var},[s,t]} )\leq 1$, then
\begin{align}\label{eq4.46}
||| z, R^{z\sharp},R^{z\sharp\sharp}|||_{p-\text{var},[s,t]}
\leq \| z_{s} \|.
\end{align}
Therefore, \eqref{eq4.2} is followed directly from the usage of the greedy sequence
of stopping times \eqref{eq4.3}, which is similar to \eqref{eq4.23} and \eqref{eq4.24}. In particular
$\bar{N}$ can be estimated by
\begin{align}\label{eq4.47}
\bar{N}_{[a,b]}(\mathbf{B}^{H,\lambda}) &\leq 1 +
\left[\min\bigg\{\frac{1}{12\Theta_{3}}, \bigg(\frac{1}{12\Theta_{3}}\bigg)^{\frac{1}{3}}\bigg\}\right]^{-p}
||| \mathbf{B}^{H,\lambda} |||^{p}_{p-\text{var},[a,b]},
\end{align}
where $\Theta_{3}$ is given by  \eqref{eq4.42-2}.
\end{proof}

\begin{theorem}
The solution $y_t(\mathbf{B}^{H,\lambda}, y_a)$ of \eqref{eq4.1} is differentiable w.r.t. initial
condition $y_a$, moreover, its derivatives $\frac{\partial y_t}{\partial y_a}(\mathbf{B}^{H,\lambda}, y_a)$ is the matrix solution of
the linearized rough differential equation
\begin{align}\label{eq4.48}
    d\xi_t = Dg(y_t) \xi_t dB^{H,\lambda}
\end{align}
\end{theorem}

\begin{proof}
The proof proceeds in several steps.

\noindent \textbf{Step 1:}  For a fixed solution $y_t(\mathbf{B}^{H,\lambda}, y_a)$ on $[a,b]$, we first establish the existence and uniqueness of the solution to the linearized rough differential equation \eqref{eq4.48}, whose time-dependent coefficient is given by $\Sigma_t := Dg(y_t)$.
To this end, we follow Gubinelli's approach by analyzing the solution mapping $H_t = \xi_a + \int_a^t \Sigma_s \xi_s dB^{H,\lambda}_s$ on the set $\mathcal{D}^{2\alpha}([a,b], \xi_a)$ of the controlled
paths $\xi_t$ with the fixed initial conditions $\xi_a$. Note that $\Sigma_t =
Dg(y_t)$ is also controlled by $B^{H,\lambda}$ with
\begin{align}
&\Sigma_t-\Sigma_s=D g(y_{t})-D g(y_{s})
\nonumber\\
&=\int_{0}^{1}
D^{2}g(y_{s}+\eta y_{s,t})(y'_{s}B^{H,\lambda;1}_{s,t}+
y''_{s}B^{H,\lambda;2}_{s,t}+R_{s,t}^{y\sharp})d\eta.\nonumber
\end{align}
Due to
\begin{align}
\Sigma'_s=D^{2}g(y_{s})g(y_{s}),\nonumber
\end{align}
and
\begin{align}
\Sigma''_s=D^{3}g(y_{s})g^{2}
(y_{s})+D^{2}g(y_{s})Dg(y_{s})g(y_{s}),\nonumber
\end{align}
in light of \eqref{eq0-2.11}, we obtain that
\begin{align}
\|R_{s,t}^{\Sigma\sharp}\|
&\leq \frac{1}{6}C_{g}^{3}\|y_{s,t}\|\|B^{H,\lambda;1}_{s,t}\|
\|B^{H,\lambda;1}_{s,t}\|
+\frac{1}{4}C_{g}^{4}\|B^{H,\lambda;1}_{s,t}\|
\|B^{H,\lambda;2}_{s,t}\|\nonumber\\
&+\frac{1}{2}C_{g}^{2}\|R_{s,t}^{y\sharp}\|\|B^{H,\lambda;1}_{s,t}\|
+\frac{1}{2}C_{g}^{3}\|y_{s,t}\|\|B^{H,\lambda;2}_{s,t}\|
+C_{g}\|R_{s,t}^{y\sharp}\|.\nonumber
\end{align}
Furthermore, $\Sigma_t \xi_t$ is also controlled by $B^{H,\lambda}$ with
\[[\Sigma.\xi.]'_s = \Sigma'_s \xi_s + \Sigma_s \xi_s',\]
\[[\Sigma.\xi.]''_s = \Sigma''_s \xi_s + 2\Sigma'_s \xi'_s+ \Sigma_s \xi''_s,\]
and
\begin{align}
\|R_{s,t}^{(\Sigma\xi)\sharp}\| &\leq
\|\xi_s\| \|R_{s,t}^{\Sigma\sharp}\| +
\|\Sigma_s\| \|R_{s,t}^{\xi\sharp}\|+
\|R_{s,t}^{\Sigma\sharp}\| \|\xi_{s,t}\|\nonumber\\
&+
C_{g}^{2}\|B^{H,\lambda;1}_{s,t}
\|(\|\xi''_s\|\|B^{H,\lambda;2}_{s,t}\|+
R_{s,t}^{\xi\sharp})
+2C_{g}^{3}\|B^{H,\lambda;2}_{s,t}\|\|\xi_{s,t}\|.\nonumber
\end{align}
In view of \eqref{eq0-2.15}, it holds that
\begin{align}
&\| H_{s,t} - \Sigma_s \xi_s \otimes B^{H,\lambda;1}_{s,t} -
[\Sigma.\xi.]'_s B^{H,\lambda;2}_{s,t}
-[\Sigma.\xi.]''_s B^{H,\lambda;3}_{s,t}\| \nonumber\\
&\leq C_\alpha (t-s)^{4\alpha}
\bigg( |||B^{H,\lambda;1}|||_{\alpha,[s,t]} |||R^{(\Sigma\xi)\sharp}|||_{3\alpha,[s,t]}
+ |||B^{H,\lambda;2}|||_{2\alpha,[s,t]}
|||R^{(\Sigma\xi)\sharp\sharp}|||_{2\alpha,[s,t]}
\nonumber\\
&+
|||[\Sigma.\xi.]''|||_{\alpha,[s,t]}
|||B^{H,\lambda;3}|||_{3\alpha,[s,t]}
\bigg),\nonumber
\end{align}
where
\begin{align}
|||[\Sigma.\xi.]'|||_{\alpha,[s,t]}
&\leq ||\Sigma'||_{\infty}|||\xi|||_{\alpha,[s,t]}
+|||\Sigma'|||_{\alpha,[s,t]}
||\xi||_{\infty}\nonumber\\
&+||\Sigma||_{\infty}
|||\xi'|||_{\alpha,[s,t]}
+|||\Sigma|||_{\alpha,[s,t]}
||\xi'||_{\infty},\nonumber
\end{align}
\begin{align}
|||R^{(\Sigma\xi)\sharp}|||_{3\alpha,[s,t]}
&\leq ||\xi||_{\infty}
|||R^{\Sigma\sharp}|||_{3\alpha,[s,t]}
+||\Sigma||_{\infty}
|||R^{\xi\sharp}|||_{3\alpha,[s,t]}\nonumber\\
&+C_{g}^{2}|||B^{H,\lambda;1}|||_{\alpha,[s,t]}
(||\xi''||_{\infty}
|||B^{H,\lambda;2}|||_{2\alpha,[s,t]}
+(t-s)^{\alpha}|||R^{\xi\sharp}|||_{3\alpha,[s,t]})
\nonumber\\
&+2C_{g}^{3}|||B^{H,\lambda;2}|||_{2\alpha,[s,t]}
|||\xi|||_{\alpha,[s,t]}
+|||R^{\Sigma\sharp}|||_{3\alpha,[s,t]}
|||\xi|||_{\alpha,[s,t]}(t-s)^{\alpha},\nonumber
\end{align}
\begin{align}
||\xi'||_{\infty}\leq
||\Sigma||_{\infty}
||\xi_{a}||+(t-s)^{\alpha}|||\xi'|||_{\alpha,[s,t]},\nonumber
\end{align}
\begin{align}
|||\xi|||_{\alpha,[s,t]}\leq
||\Sigma||_{\infty}
||\xi_{a}|| |||B^{H,\lambda;1}|||_{\alpha,[s,t]}
+(t-s)^{\alpha}(|||B^{H,\lambda;1}|||_{\alpha,[s,t]}\vee (t-s)^{\alpha})
|||(\xi, \xi')|||_{2\alpha,[s,t]},\nonumber
\end{align}
\begin{align}
|||\xi'|||_{\alpha,[s,t]}\leq
|||\Sigma|||_{\alpha,[s,t]}||\xi||_{\infty}
+||\Sigma||_{\infty}|||\xi|||_{\alpha,[s,t]},\nonumber
\end{align}
\begin{align}
||\xi||_{\infty}
&\leq ||\xi_{a}||+(t-s)^{\alpha}
(||\Sigma||_{\infty}||\xi_{a}||
|||B^{H,\lambda;1}|||_{\alpha,[s,t]}\nonumber\\
&+(t-s)^{\alpha}(|||B^{H,\lambda;1}|||_{\alpha,[s,t]}\vee (t-s)^{\alpha})
|||(\xi, \xi')|||_{2\alpha,[s,t]}),\nonumber
\end{align}
\begin{align}
||\xi''||_{\infty}
\leq \|\Sigma'\|_{\infty}\|\xi\|_{\infty}
+\|\Sigma\|_{\infty}\|\xi'\|_{\infty},\nonumber
\end{align}
and
\begin{align}
|||\xi'|||_{\alpha,[s,t]}
\leq ||\xi''||_{\infty}|||B^{H,\lambda;1}|||_{\alpha,[s,t]}
+|||R^{\xi\sharp\sharp}|||_{2\alpha,[s,t]}
(t-s)^{\alpha}.\nonumber
\end{align}
Consequently, a direct computation shows that
\begin{align}
|||R^{(\Sigma\xi)\sharp}|||_{3\alpha,[s,t]}
&\leq \bigg[
||\xi_{a}||+(t-s)^{\alpha}
\bigg((t-s)^{\alpha}(|||B^{H,\lambda;1}|||_{\alpha,[s,t]}\vee (t-s)^{\alpha})
|||(\xi, \xi', \xi'')|||_{2\alpha,[s,t]}
\nonumber\\
&+
\|\Sigma\|_{\infty}
||\xi_{a}||
|||B^{H,\lambda;1}|||_{\alpha,[s,t]}
\bigg)\bigg]
|||R^{\Sigma\sharp}|||_{3\alpha,[s,t]}
+\|\Sigma\|_{\infty}|||R^{\Sigma\sharp}|||_{3\alpha,[s,t]}
\nonumber\\
&+C_{g}^{2}|||B^{H,\lambda;1}|||_{\alpha,[s,t]}
\bigg[\|\Sigma'\|_{\infty}
\bigg(||\xi_{a}||+(t-s)^{\alpha}\bigg(\|\Sigma\|_{\infty}
||\xi_{a}|||||B^{H,\lambda;1}|||_{\alpha,[s,t]}
\nonumber\\
&+(t-s)^{\alpha}(|||B^{H,\lambda;1}|||_{\alpha,[s,t]}\vee (t-s)^{\alpha})
|||(\xi, \xi', \xi'')|||_{2\alpha,[s,t]}\bigg)\bigg)
+\|\Sigma\|_{\infty}(\|\Sigma\|_{\infty}||\xi_{a}||
\nonumber\\
&+
(t-s)^{\alpha}|||\xi'|||_{\alpha,[s,t]})\bigg]
|||B^{H,\lambda;2}|||_{2\alpha,[s,t]}
+C_{g}^{2}|||B^{H,\lambda;1}|||_{\alpha,[s,t]}(t-s)^{\alpha}
|||R^{\xi\sharp}|||_{3\alpha,[s,t]}\nonumber\\
&+2C_{g}^{3}|||B^{H,\lambda;2}|||_{2\alpha,[s,t]}
|||\xi|||_{\alpha,[s,t]}
+|||R^{\Sigma\sharp}|||_{3\alpha,[s,t]}
|||\xi|||_{\alpha,[s,t]}
(t-s)^{\alpha},\nonumber
\end{align}
\begin{align}
|||R^{(\Sigma\xi)\sharp\sharp}|||_{2\alpha,[s,t]}
&\leq \|\Sigma'\|_{\infty}
(\|\xi''\|_{\infty}|||B^{H,\lambda;2}|||_{2\alpha,[s,t]}
+(t-s)^{\alpha}|||R^{\xi\sharp}|||_{3\alpha,[s,t]}
)\nonumber\\
&
+(\|\Sigma''\|_{\infty}|||B^{H,\lambda;2}|||_{2\alpha,[s,t]}+
|||R^{\Sigma\sharp}|||_{3\alpha,[s,t]})\|\xi'\|_{\infty}
\nonumber\\
&+|||\Sigma|||_{\alpha,[s,t]}
(\|\xi''\|_{\infty}|||B^{H,\lambda;1}|||_{\alpha,[s,t]}
\nonumber\\
&+(t-s)^{\alpha}|||R^{\xi\sharp\sharp}|||_{2\alpha,[s,t]})
+\|\Sigma\|_{\infty}|||R^{\xi\sharp\sharp}|||_{2\alpha,[s,t]},
\nonumber
\end{align}
\begin{align}
|||(\Sigma\xi)''|||_{\alpha,[s,t]}&\leq
|||\Sigma''|||_{\alpha,[s,t]}\|\xi\|_{\infty}
+||\Sigma''||_{\infty}|||\xi|||_{\alpha,[s,t]}\nonumber\\
&+2|||\Sigma'|||_{\alpha,[s,t]}\|\xi'\|_{\infty}
+2||\Sigma'||_{\infty}|||\xi|||_{\alpha,[s,t]}
\nonumber\\
&+|||\Sigma|||_{\alpha,[s,t]}
\|\xi''\|_{\infty}
+||\Sigma||_{\infty}|||\xi''|||_{\alpha,[s,t]}
+||\Sigma_{s}|||||\xi''|||_{\alpha,[s,t]},\nonumber
\end{align}
and
\begin{align}
||(\Sigma\xi)''||_{\infty}\leq
||\Sigma''||_{\infty}
\|\xi\|_{\infty}
+2||\Sigma'||_{\infty}
\|\xi'\|_{\infty}
+||\Sigma||_{\infty}
\|\xi''\|_{\infty}.\nonumber
\end{align}
Therefore,
\begin{align}
||||H'|||_{\alpha,[s,t]}\leq ||\Sigma||_{\infty}  |||\xi|||_{\alpha,[s,t]}
+||\xi||_{\infty}  |||\Sigma|||_{\alpha,[s,t]}.\nonumber
\end{align}
Notice that
\begin{align}
R^{H\sharp\sharp}_{s,t}
=\Sigma_{t}\xi_{t}-\Sigma_{s}\xi_{s}-
(\Sigma'_{s}\xi_{s}+\Sigma_{s}\xi'_{s})B^{H,\lambda;1}_{s,t}
=(\Sigma_{s}\xi_{s})''B^{H,\lambda;2}_{s,t}+
R^{(\Sigma\xi)\sharp}_{s,t}.\nonumber
\end{align}
As a result, we conclude that there exist constants
\[
M_1 = M_1(\Sigma, [a,b], C_{g}, B^{H,\lambda;1}, B^{H,\lambda;2}, B^{H,\lambda;3})\]
and
\[M_2 = M_2(\Sigma, [a,b], C_{g}, B^{H,\lambda;1}, B^{H,\lambda;2}, B^{H,\lambda;3}),
\]
such that
\begin{align}\label{eq4.68}
&|||H''|||_{\alpha,[s,t]} + |||R^{H\sharp}|||_{3\alpha ,[s,t]}
+|||R^{H\sharp\sharp}|||_{2\alpha,[s,t]}\nonumber\\
&\leq M_2\|\xi_\alpha\|+ M_1
\bigg[(t-a)^\alpha + |||B^{H,\lambda;1}|||_{\alpha,[s,t]}+ |||B^{H,\lambda;2}|||_{2\alpha,[s,t]}\nonumber\\
&+ |||B^{H,\lambda;3}|||_{3\alpha,[s,t]}\bigg]
|||(\xi,\xi',\xi'')|||_{2\alpha,[s,t]}.
\end{align}
This implies that on every interval $[\bar{\tau}_i,\bar{\tau}_{i+1}]$ of the greedy sequence of stopping times
$\{\bar{\tau}_i(\frac{1}{2M_1},I,\alpha)\}_{i\in\mathbb{N}}$ as in \eqref{eq0-2.18},
the solution mapping is a contraction from the set
\[
\bigg\{\mathcal{D}^{2\alpha}([a,b],\xi_a):
|||(\xi,\xi',\xi'')|||_{2\alpha,[\bar{\tau}_i,\bar{\tau}_{i+1}]}\leq 2M_2\|\xi_{\bar{\tau}_i}\|\bigg\}
\]
into itself. Consequently, a solution to \eqref{eq4.48} exists on each interval $[\bar{\tau}_i,\bar{\tau}_{i+1}]$. By linearity, it follows from estimate \eqref{eq4.68} with
$\|\xi_a-\bar{\xi}_a\|=0$ that $\|(\xi-\bar{\xi},\xi'-\bar{\xi}', \xi''-\bar{\xi}'')\|_{2\alpha}=0$,
which establishes uniqueness. Finally, concatenating the solutions over the intervals
$[\bar{\tau}_i,\bar{\tau}_{i+1}]$ yields the existence and uniqueness of the solution to \eqref{eq4.48} on $[a,b]$.

\noindent\textbf{Step 2:}
Let $\Phi(t,B^{H,\lambda},y_a)$ denote the solution matrix of the linearized system
\eqref{eq4.48}. Then $\xi_t = \Phi(t,B^{H,\lambda},y_a)(\bar{y}_a - y_a)$ is the solution of \eqref{eq4.48} with initial condition  $\xi_a = \bar{y}_a - y_a$. Define $r_t:=\bar{y}_t - y_t - \xi_t$, then $r_a = 0$ and
\begin{align}\label{eq4.69}
r_t
&= \int_a^t \int_0^1 \bigl[ Dg(y_s + \eta(\bar{y}_s - y_s)) - Dg(y_s) \bigr] (\bar{y}_s - y_s) d\eta dB^{H,\lambda}_s + \int_a^t Dg(y_s) r_s dB^{H,\lambda}_s, \nonumber\\
&= e_{a,t} + \int_a^t Dg(y_s) r_s dB^{H,\lambda}_s,
\end{align}
where
\[
e_{a,t} = \int_a^t \int_0^1 \bigl[ Dg(y_s + \eta(\bar{y}_s - y_s)) - Dg(y_s) \bigr] (\bar{y}_s - y_s) d\eta dB^{H,\lambda}_s,
\]
and $e$ is also controlled by $B^{H,\lambda}$ with $e_{a,a}=0$.
We are going to estimate $\|r\|_{\infty,[a,b]}$
and $\|r,R^{r\sharp},R^{r\sharp\sharp}\|_{p-\mathrm{var},[a,b]}$. We observe that
\begin{align}
r_{s,t}&=e_{s,t}+\int_{s}^{t}Dg(y_{u})r_{u}dB^{H,\lambda}_{u}\nonumber\\
&=e'_{s}\otimes x_{s,t}+e''_{s}\otimes B^{H,\lambda;2}_{s,t}+R^{e\sharp}_{s,t}
+\int_{s}^{t}Dg(y_{u})r_{u}dB^{H,\lambda}_{u}.\nonumber
\end{align}
Notice that $r'_{s}=e'_{s}+Dg(y_{s})r_{s}$ and
\begin{align}
r''_{s}
&=\int_{0}^{1}\int_{0}^{1}D^{3}g(y_{s}+\eta \xi (\bar{y}_{s}-y_{s}))
[g(y_{s})+\eta \xi (g(\bar{y}_{s})-g(y_{s}))]\eta (\bar{y}_{s}-y_{s})^{2}d\xi d\eta\nonumber\\
&+\int_{0}^{1}\int_{0}^{1}
D^{2}g(y_{s}+\eta \xi (\bar{y}_{s}-y_{s}))
\eta (g(\bar{y}_{s})-g(y_{s})) (\bar{y}_{s}-y_{s}) d\xi d\eta
\nonumber\\
&+ Dg(y_{s})\int_{0}^{1}[Dg(y_{s}+\eta(\bar{y}_{s}-y_{s}) )-
Dg(y_{s})](\bar{y}_{s}-y_{s}) d\eta
\nonumber\\
&+D^{2}g(y_{s})g(y_{s})r_{s}+Dg(y_{s})g(y_{s})r_{s},\nonumber
\end{align}
then
\begin{align}
||R^{r\sharp}_{s,t}||&\leq ||R^{e\sharp}_{s,t}||
+\|[Dg(y_{s}r_{s})]''\|\|B^{H,\lambda;3}_{s,t}\|\nonumber\\
&+C_p \bigg( ||| B^{H,\lambda} |||_{p-\text{var},[s,t]}
||| R^{(Dg(y)r)\sharp} \|_{\frac{p}{3}-\text{var},[s,t]^2}
\nonumber\\
&+||| B^{H,\lambda;2} |||_{\frac{p}{2}-\text{var},[s,t]}
||| R^{(Dg(y)r)\sharp\sharp} \|_{\frac{p}{2}-\text{var},[s,t]^2}
\nonumber\\
&
+||| (Dg(y)r)'' \|_{p-\text{var},[s,t]} ||| B^{H,\lambda;3} |||_{\frac{p}{3}-\text{var},[s,t]^2} \bigg).\nonumber
\end{align}
It is easy to see that
\begin{align}
&\|Dg(y_{t})r_{t}-Dg(y_{s})r_{s}-[Dg(y)r]'_{s}B^{H,\lambda;1}_{s,t}
-[Dg(y)r]''_{s}B^{H,\lambda;2}_{s,t}\|
\nonumber\\
&\leq \frac{1}{3}C_{g}^{3}\|y_{s,t}\|\|B^{H,\lambda;2}_{s,t}\|\|r_{s}\|
+\frac{1}{2}C_{g}^{2}(C_{g}^{2}\|B^{H,\lambda;2}_{s,t}\|+
\|R^{y\sharp}_{s,t}\|)
\|B^{H,\lambda;1}_{s,t}\|\|r_{s}\|
\nonumber\\
&+\frac{1}{2}C_{g}^{3}\|y_{s,t}\|\|B^{H,\lambda;2}_{s,t}\|\|r_{s}\|
+C_{g}\|R^{y\sharp}_{s,t}\|\|r_{s}\|
+C_{g}\|R^{r\sharp}_{s,t}\|
\nonumber\\
&+C_{g}^{2}\|B^{H,\lambda;1}_{s,t}\|
(\|r''_{s}\|\|B^{H,\lambda;2}_{s,t}\|+\|R^{r\sharp}_{s,t}\|)
+C_{g}(C_{g}^{2}\|B^{H,\lambda;2}_{s,t}\|+
\|R^{y\sharp}_{s,t}\|)\|r_{s,t}\|
\nonumber\\
&+C_{g}^{2} \|y_{s,t}\|
\|B^{H,\lambda;2}_{s,t}\| \|r'_{s}\|.\nonumber
\end{align}
In addition, we deduce that
\begin{align}
&|||R^{(Dg(y)r)\sharp}|||_{\frac{p}{3}-\mathrm{var}}\nonumber\\
&\leq C_{g}^{3}|||y|||_{p-\mathrm{var}}
|||B^{H,\lambda;2}|||_{\frac{p}{2}-\mathrm{var}}
\|r\|_{\infty}
+C_{g}|||R^{y\sharp}|||_{\frac{p}{3}-\mathrm{var}}\|r\|_{\infty}
+C_{g}|||R^{r\sharp}|||_{\frac{p}{3}-\mathrm{var}}
\nonumber\\
&
+\frac{1}{2}C_{g}^{2}(C_{g}^{2}
|||B^{H,\lambda;2}|||_{\frac{p}{2}-\mathrm{var}}
+|||R^{y\sharp}|||_{\frac{p}{3}-\mathrm{var}})
|||B^{H,\lambda;1}|||_{p-\mathrm{var}}\|r\|_{\infty}
\nonumber\\
&+C_{g}^{2}|||B^{H,\lambda;1}|||_{p-\mathrm{var}}
(\|r'\|_{\infty}|||B^{H,\lambda;2}|||_{\frac{p}{2}-\mathrm{var}}+
|||R^{r\sharp}|||_{\frac{p}{3}-\mathrm{var}})
\nonumber\\
&+C_{g}|||r|||_{p}(C_{g}^{2}
|||B^{H,\lambda;2}|||_{\frac{p}{2}-\mathrm{var}}
+|||R^{y\sharp}|||_{\frac{p}{3}-\mathrm{var}})
+C_{g}^{2} |||y|||_{p-\mathrm{var}}
|||B^{H,\lambda;2}|||_{\frac{p}{2}-\mathrm{var}}\|r'\|_{\infty},
\nonumber
\end{align}
and
\begin{align}
&|||R^{(Dg(y)r)\sharp\sharp}|||_{\frac{p}{2}-\mathrm{var}}
\nonumber\\
&\leq 2C_{g}^{3}|||y|||_{p-\mathrm{var}}
|||B^{H,\lambda;1}|||_{p-\mathrm{var}}\|r\|_{\infty}\nonumber\\
&+2C_{g}^{2}\|r\|_{\infty}
(C_{g}^{2}|||B^{H,\lambda;2}|||_{\frac{p}{2}-\mathrm{var}}+
|||R^{y\sharp}|||_{\frac{p}{3}-\mathrm{var}})\nonumber\\
&+\frac{1}{2}C_{g}^{2}|||y|||_{p-\mathrm{var}}
|||B^{H,\lambda;1}|||_{p-\mathrm{var}}\|r'\|_{\infty}
+C_{g}^{2}\|r'\|_{\infty}
(C_{g}^{2}|||B^{H,\lambda;2}|||_{\frac{p}{2}-\mathrm{var}}+
|||R^{y\sharp}|||_{\frac{p}{3}-\mathrm{var}})\nonumber\\
&+2C_{g}^{2}|||y|||_{p-\mathrm{var}}
|||r|||_{p-\mathrm{var}}
+C_{g}^{2}(\|r''\|_{\infty}
|||B^{H,\lambda;2}|||_{\frac{p}{2}-\mathrm{var}}+
|||R^{r\sharp}|||_{\frac{p}{3}-\mathrm{var}})
\nonumber\\
&+C_{g}|||r'|||_{p-\mathrm{var}}|||y|||_{p}+C_{g}
|||R^{r\sharp\sharp}|||_{\frac{p}{2}-\mathrm{var}}.\nonumber
\end{align}
Similarly, we estimate that
\begin{align}
\|[Dg(y)r]''\|_{\infty}\leq 2C_{g}^{3}\|r\|_{\infty}
+2C_{g}^{2}\|r'\|_{\infty}+C_{g}\|r''\|_{\infty},\nonumber
\end{align}
and
\begin{align}
|||[Dg(y)r]''|||_{p-\mathrm{var}}&\leq 6C_{g}^{3}|||y|||_{p-\mathrm{var}}\|r\|_{\infty}
+2C_{g}^{3}|||r|||_{p-\mathrm{var}}+
4C_{g}^{2}|||y|||_{p-\mathrm{var}}\|r'\|_{\infty}
\nonumber\\
&+2C_{g}^{2}|||r'|||_{p-\mathrm{var}}+
C_{g}|||y|||_{p-\mathrm{var}}\|r''\|_{\infty}
+C_{g}|||r''|||_{p-\mathrm{var}}.\nonumber
\end{align}
Combining the above estimates, it follows that there exists a constant $M_{1}>1$ such that
\begin{align}
|||R^{r\sharp}|||_{\frac{p}{3}-\mathrm{var}}
&\leq M_{1}
(|||\mathbf{B}^{H,\lambda}|||_{p-\mathrm{var}}\vee |||\mathbf{B}^{H,\lambda}|||^{2}_{p-\mathrm{var}}\vee |||\mathbf{B}^{H,\lambda}|||^{3}_{p-\mathrm{var}})
(||r_{s}||+|||r, R^{r\sharp}, R^{r\sharp\sharp}|||_{p-\mathrm{var}})\nonumber\\
&+M_{1}(\|e'\|_{\infty, [a,b]}+|||e'|||_{p-\mathrm{var}}+
|||R^{e\sharp}|||_{\frac{p}{3}-\mathrm{var}}+
\|e''\|_{\infty, [a,b]}).\nonumber
\end{align}
Similarly,
\begin{align}
|||r|||_{p}\leq \|e'\|_{\infty}
|||B^{H,\lambda;1}|||_{p-\mathrm{var}}+
\|e''\|_{\infty}|||B^{H,\lambda;2}|||_{\frac{p}{2}-\mathrm{var}}
+C_{g}\|r\|_{\infty}|||B^{H,\lambda;1}|||_{p-\mathrm{var}}+
|||R^{e\sharp}|||_{\frac{p}{3}-\mathrm{var}},
\nonumber
\end{align}
and
\begin{align}
&|||R^{r\sharp\sharp}|||_{\frac{p}{2}-\mathrm{var}}
\nonumber\\
&\leq
|||R^{e\sharp\sharp}|||_{\frac{p}{2}-\mathrm{var}}+
\|r\|_{\infty}
\bigg[2C_{g}^{3}|||B^{H,\lambda;2}|||_{\frac{p}{2}-\mathrm{var}}
+\bigg(\frac{1}{6}C_{g}^{3}|||y|||_{p-\mathrm{var}}
|||B^{H,\lambda;1}|||_{p-\mathrm{var}}^{2}
\nonumber\\
&+\frac{1}{2}C_{g}^{4}|||B^{H,\lambda;2}|||_{\frac{p}{2}-\mathrm{var}}
|||B^{H,\lambda;1}x|||_{p-\mathrm{var}}
+\frac{1}{2}C_{g}^{2}|||R^{y\sharp}|||_{\frac{p}{3}-\mathrm{var}}
|||B^{H,\lambda;1}|||_{p-\mathrm{var}}
\nonumber\\
&+\frac{1}{2}C_{g}^{3}|||y|||_{p-\mathrm{var}}
|||B^{H,\lambda;2}|||_{\frac{p}{2}-\mathrm{var}}
+C_{g}|||R^{y\sharp}|||_{\frac{p}{3}-\mathrm{var}}\bigg)\bigg]
+C_{g}|||y|||_{p-\mathrm{var}}|||r|||_{p-\mathrm{var}}
\nonumber\\
&+C_{g}
|||B^{H,\lambda;2}|||_{\frac{p}{2}-\mathrm{var}}
(\|e''\|_{\infty}+C_{g}^{2}\|r\|_{\infty}
+C_{g}\|e'\|_{\infty}+C_{g}^{2}\|r\|_{\infty}
+C_{g}|||R^{r\sharp}|||_{\frac{p}{3}-\mathrm{var}}).
\nonumber
\end{align}
As a result, we conclude that there exists a constant $M_{2}>1$ such that
\begin{align}
&|||r, R^{r\sharp}, R^{r\sharp\sharp}|||_{p-\mathrm{var}}
\nonumber\\
&\leq
M_{2}
(C_{g}|||\mathbf{B}^{H,\lambda}|||_{p-\mathrm{var}}\vee C_{g}^{2}|||\mathbf{B}^{H,\lambda}|||^{2}_{p-\mathrm{var}}
\vee C_{g}^{3}|||\mathbf{B}^{H,\lambda}|||^{3}_{p-\mathrm{var}})
(||r_{s}||+|||r, R^{r\sharp}, R^{r\sharp\sharp}|||_{p-\mathrm{var}})\nonumber\\
&+M_{2}(\|e'\|_{\infty}+|||e'|||_{p-\mathrm{var}}+
|||R^{e\sharp}|||_{\frac{p}{3}-\mathrm{var}}+\|e''\|_{\infty}+
|||R^{e\sharp\sharp}|||_{\frac{p}{2}-\mathrm{var}}).\nonumber
\end{align}
when $M_{2}C_{g}|||\mathbf{B}^{H,\lambda}|||_{p-\mathrm{var}}\leq \frac{1}{2}$, we have
\begin{align}
|||r, R^{r\sharp}, R^{r\sharp\sharp}|||_{p-\mathrm{var}}
\leq ||r_{s}||+2M_{2}(\|e'\|_{\infty}+|||e'|||_{p-\mathrm{var}}+
|||R^{e\sharp}|||_{\frac{p}{3}-\mathrm{var}}+\|e''\|_{\infty}+
|||R^{e\sharp\sharp}|||_{\frac{p}{2}-\mathrm{var}}).\nonumber
\end{align}
By employing estimates analogous to \eqref{eq4.23} and \eqref{eq4.24}, and utilizing the greedy sequence of stopping times $\{\tau_{k}(\frac{1}{2M_2 C_{g}})\}_{k\in N}$, we obtain
\begin{align}\label{eq4.70}
&\|r\|_{\infty,[a,b]} \vee |||r, R^{r\sharp}, R^{r\sharp\sharp}|||_{p-\mathrm{var},[a,b]} \nonumber\\
&\leq N_{\frac{1}{2M_2 C_g},[a,b]}
(\mathbf{B}^{H,\lambda})^{\frac{p-1}{p}} e^{(\log 2) N_{\frac{1}{2M_2 C_g},[a,b]}(\mathbf{B}^{H,\lambda})}\nonumber\\
&\times \bigg\{ \|r_a\| + M_2 \bigg(
\|e'\|_{\infty, [a,b]}+|||e'|||_{p-\mathrm{var}}+
|||R^{e\sharp}|||_{\frac{p}{3}-\mathrm{var}}+\|e''\|_{\infty, [a,b]}+
|||R^{e\sharp\sharp}|||_{\frac{p}{2}-\mathrm{var}}
\bigg) \bigg\} \nonumber\\
&\leq M_2 \bigg(
\|e'\|_{\infty, [a,b]}+|||e'|||_{p-\mathrm{var}}+
|||R^{e\sharp}|||_{\frac{p}{3}-\mathrm{var}}+\|e''\|_{\infty, [a,b]}+
|||R^{e\sharp\sharp}|||_{\frac{p}{2}-\mathrm{var}}
\bigg)\nonumber\\
&~~\times
e^{(1+\log 2)
N_{\frac{1}{2M_2 C_g},[a,b]}(\mathbf{B}^{H,\lambda})}
,
\end{align}
where we use the fact that $r_a = 0$.

\noindent\textbf{Step 3:}
Setting $z_s = \bar{y}_s - y_s$, it follows from \eqref{eq4.69} that
\begin{align}
e_s'
&= \int_0^1 \bigl[ Dg(y_s + \eta z_s) - Dg(y_s) \bigr] z_s d\eta
\nonumber\\
&= \int_0^1 \int_0^1 D^2 g\bigl((1-\mu\eta)y_s + \mu\eta\bar{y}_s\bigr) [z_s,z_s] \eta d\mu d\eta = \bar{e}_s,
\nonumber
\end{align}
which is also controlled by $B^{H,\lambda}$. As a result, $\|e_s'\| \leq \frac12 C_g \|z_s\|^2$ and
\begin{align}\label{eq4.80}
\|e'\|_{\infty,[a,b]}
&\leq \frac{1}{2} C_g \|z\|_{\infty,[a,b]}^2,
\end{align}
\begin{align}
\|e''\|_{\infty,[a,b]}
\leq 2C_g^{2}\|z\|_{\infty,[a,b]}^{2},\nonumber
\end{align}
and
\begin{align}
|||e'|||_{p-\mathrm{var}}
&\leq C_g \bigl( |||y|||_{p-\mathrm{var}} \vee |||\bar{y}|||_{p-\mathrm{var}} \bigr)
\|z\|_{\infty}^2 \nonumber\\
&+ 2C_g \|z\|_{\infty} |||z|||_{p-\mathrm{var}}.\nonumber
\end{align}
On the other hand, we estimate that
\begin{align}\label{eq4.81}
|||R^{e\sharp}|||_{\frac{p}{3}-\mathrm{var}}
&\leq \|\bar{e}''\|_{\infty}
|||B^{H,\lambda;3}|||_{\frac{p}{3}-\mathrm{var}}
+C_g\bigg(
|||B^{H,\lambda;1}|||_{p-\mathrm{var}}
|||R^{\bar{e}\sharp}|||_{\frac{p}{3}-\mathrm{var}}
\nonumber\\
&+|||B^{H,\lambda;2}|||_{\frac{p}{2}-\mathrm{var}}
|||R^{\bar{e}\sharp\sharp}|||_{\frac{p}{2}-\mathrm{var}}
+|||\bar{e}''|||_{p-\mathrm{var}}
|||B^{H,\lambda;3}|||_{\frac{p}{3}-\mathrm{var}}
\bigg).
\end{align}
A direct calculation yields
\begin{align}
||\bar{e}''||_{\infty}\leq
C_g^{3}\|z\|_{\infty,[a,b]}^{2}
+6C_g^{3}\|z\|_{\infty,[a,b]}+
2C_g^{3}|||z|||_{p-\mathrm{var}},\nonumber
\end{align}
and
\begin{align}
|||R^{\bar{e}\sharp}|||_{\frac{p}{3}-\mathrm{var}}
&\leq |||R^{(D^{2}g)\sharp}|||_{\frac{p}{3}-\mathrm{var}}
\|z\|_{\infty,[a,b]}^{2}
+\frac{C_g}{2}|||R^{z^{2}\sharp}|||_{\frac{p}{3}-\mathrm{var}}
+\frac{1}{2}C_g^{2}
\bigg(|||B^{H,\lambda;1}|||_{p-\mathrm{var}}
|||z|||_{p-\mathrm{var}}^{2}
\nonumber\\
&+|||B^{H,\lambda;1}|||_{p-\mathrm{var}}
|||z|||_{p-\mathrm{var}}
(2 C_g^{2}|||B^{H,\lambda;2}|||_{\frac{p}{2}-\mathrm{var}}
+
|||R^{z\sharp}|||_{\frac{p}{3}-\mathrm{var}})\bigg).\nonumber
\end{align}
Therefore, there exists a generic constant D such that
\begin{align}\label{eq4.82}
&||\bar{e}'||_{\infty}\vee |||\bar{e}, R^{\bar{e}\sharp}, R^{\bar{e}\sharp\sharp}|||_{p-\mathrm{var},[a,b]} \nonumber\\
&\leq D(||z||_{\infty}+||z'||_{\infty}
+|||z|||_{p-\mathrm{var}}+
|||z'|||_{p-\mathrm{var}}+
|||R^{z\sharp}|||_{\frac{p}{3}-\mathrm{var}}+
|||R^{z\sharp\sharp}|||_{\frac{p}{2}-\mathrm{var}})^{2}
\end{align}
By replacing \eqref{eq4.80}, \eqref{eq4.81}, \eqref{eq4.82} into \eqref{eq4.70}, and using \eqref{eq4.2}, we conclude
that there exists a generic constant $D = D(p, [a,b], |||\mathbf{B}^{H, \lambda}|||_{p-\mathrm{var}})$ such that
\begin{align}\label{eq4.82-1}
\|\bar{y}(\mathbf{B}^{H, \lambda},\bar{y}_a)
- y(\mathbf{B}^{H, \lambda},y_a) - \xi(\mathbf{B}^{H, \lambda},\bar{y}_a - y_a)\|_{\infty,[a,b]}
\leq D\|\bar{y}_a - y_a\|^2.
\end{align}
This, combined with the linearity of $\xi$ with respect to $\bar{y}_a - y_a$, shows the differentiability
of $y_t(\mathbf{B}^{H, \lambda},y_a)$ with respect to $y_a$.
\end{proof}

\begin{corollary}
Consider the backward rough differential equation
\begin{align}\label{eq4.83}
h_b = h_t + \int_t^b g(h_u)
dB^{H,\lambda}_u,\quad \forall t \in [a,b].
\end{align}
Then the solution $h_t(\mathbf{B}^{H,\lambda},h_b)$ of \eqref{eq4.83} is differentiable with respect to initial condition
$h_b$, moreover, its derivatives
$\frac{\partial h_t}{\partial h_b}(\mathbf{B}^{H,\lambda},h_b)$ is the matrix solution of the linearized
backward rough differential equation
\begin{align}\label{eq4.84}
\eta_b = \eta_t + \int_t^b Dg(h_u)\eta_u
dB^{H,\lambda}_u,\quad \forall t \in [a,b].
\end{align}
Moreover, the estimates \eqref{eq4.2} also hold for the solution $h_t$ of the backward
equation \eqref{eq4.83}.

\end{corollary}

\begin{corollary}
Denote by $\varphi(t,\mathbf{B}^{H,\lambda},y_a)$ the solution mapping of the rough
equation \eqref{eq4.1}. Then $\frac{\partial \varphi}{\partial y}(t,\mathbf{B}^{H,\lambda},y_a)$ is globally Lipschitz continuous w.r.t. $y_a$.
\end{corollary}

\begin{proof}
In view of \eqref{eq4.2} and \eqref{eq4.82-1}, we find that
\begin{align}
\|\xi(\mathbf{B}^{H,\lambda},\bar{y}_a - y_a)\|_{\infty,[a,b]}
&\leq \|\bar{y}(\mathbf{B}^{H,\lambda},\bar{y}_a) - y(\mathbf{B}^{H,\lambda},y_a)\| + D\|\bar{y}_a - y_a\|^2 \nonumber\\
&\leq \|\bar{y}_a - y_a\| e^{(\log 2)\bar{N}_{[a,b]}(\mathbf{B}^{H,\lambda})} + D\|\bar{y}_a - y_a\|^2.\nonumber
\end{align}
By fixing $y_a$, dividing both sides by $\|\bar{y}_a - y_a\|$, and then taking the limit as $\|\bar{y}_a - y_a\|$
to zero, we obtain
\begin{align}
\bigg\|\frac{\partial \varphi}{\partial y}(t,\mathbf{B}^{H,\lambda},y_a)\bigg\|_{\infty,[a,b]}
\leq e^{(\log 2)\bar{N}_{[a,b]}(\mathbf{B}^{H,\lambda})},\nonumber
\end{align}
and
\begin{align}\label{eq4.85}
\bigg
\|\frac{\partial \varphi}{\partial y}(t,\mathbf{B}^{H,\lambda},y_a)\bigg\|_{p-\mathrm{var},[a,b]}
\leq\bar{N}_{[a,b]}^{\frac{p-1}{p}}(\mathbf{B}^{H,\lambda})
e^{(\log 2)\bar{N}_{[a,b]}(\mathbf{B}^{H,\lambda})}.
\end{align}
Similarly if we fix $\bar{y}_a$ then
\begin{align}
\bigg\|\frac{\partial \varphi}{\partial y}(t,\mathbf{B}^{H,\lambda},\bar{y}_a)\bigg\|_{\infty,[a,b]}
\leq e^{(\log 2)\bar{N}_{[a,b]}(\mathbf{B}^{H,\lambda})},\nonumber
\end{align}
and
\begin{align}\label{eq4.86}
\bigg\|\frac{\partial \varphi}{\partial y}(t,\mathbf{B}^{H,\lambda},\bar{y}_a)\bigg\|_{p-\mathrm{var},[a,b]}
\leq \bar{N}_{[a,b]}^{\frac{p-1}{p}}(\mathbf{B}^{H,\lambda}) e^{(\log 2)\bar{N}_{[a,b]}(\mathbf{B}^{H,\lambda})}.
\end{align}
Next for $\xi_t = \frac{\partial \varphi}{\partial y}(t,\mathbf{B}^{H,\lambda},y_a)\xi_a$ and $\bar{\xi}_t = \frac{\partial \varphi}{\partial y}(t,\mathbf{B}^{H,\lambda},\bar{y}_a)\xi_a$, we consider the difference
$r_t = \bar{\xi}_t - \xi_t$, which satisfies $r_a = 0$ and
\begin{align}\label{eq4.87}
r_{s,t}
&= \int_s^t \bigl( Dg(\bar{y}_u)\bar{\xi}_u - Dg(y_u)\xi_u \bigr) dB^{H,\lambda}_u \nonumber\\
&= \int_s^t\bigl( Dg(\bar{y}_u) - Dg(y_u) \bigr)
\bar{\xi}_u dB^{H,\lambda}_u + \int_s^t Dg(y_u) r_u dB^{H,\lambda}_u \nonumber\\
&= e_{s,t} + \int_s^t Dg(y_u) r_u dB^{H,\lambda}_u.
\end{align}
Equation \eqref{eq4.87} has the same structure as \eqref{eq4.69}. Therefore, by employing the arguments from Step 2 of Theorem 4.2, we obtain an estimate similar to \eqref{eq4.70}. To be precise, there exists a constant $M$ such that
\begin{align}\label{eq4.88}
&\|r\|_{\infty,[a,b]} \vee |||r, R^{r\sharp}, R^{r\sharp\sharp}|||_{p-\mathrm{var},[a,b]}
\leq M e^{(1+\log 2) N_{\frac{1}{2M C_g},[a,b]}(\mathbf{B}^{H,\lambda})}\nonumber\\
&\times \bigg(
\|e'\|_{\infty, [a,b]}+
|||e'|||_{p-\mathrm{var}}+
|||R^{e\sharp}|||_{\frac{p}{3}-\mathrm{var}}+
\|e''\|_{\infty, [a,b]}+
|||R^{e\sharp\sharp}|||_{\frac{p}{2}-\mathrm{var}}
\bigg),
\end{align}
where we use the fact that $r_a = 0$.
From \eqref{eq4.87}, we derive the following estimates for the terms on the right-hand side of \eqref{eq4.88}:
\begin{align}
e_s' = \bigl( Dg(\bar{y}_s) - Dg(y_s) \bigr)\bar{\xi}_s = \int_0^1 D^2 g\bigl((1-\eta)y_s + \eta\bar{y}_s\bigr) [z_s,\bar{\xi}_s] d\eta.\nonumber
\end{align}
From \eqref{eq4.86} and the estimates derived in Step 3 of Theorem 4.2, it follows that there exists a generic constant $M = M(p, [a,b], \|\mathbf{B}^{H,\lambda}\|_{p-\mathrm{var},[a,b]}) > 1$
such that
\begin{align}\label{eq4.89}
&\|e'\|_{\infty,[a,b]}+\|e''\|_{\infty,[a,b]}
+|||e'|||_{p-\mathrm{var},[a,b]}
+|||R^{e\sharp}|||_{\frac{p}{3}-\mathrm{var}}\nonumber\\
&+
|||R^{e\sharp\sharp}|||_{\frac{p}{2}-\mathrm{var}}
\leq M(||z_{a}||+|||z, R^{z\sharp}, R^{z\sharp\sharp}\|_{p-\mathrm{var},[a,b]}).
\end{align}
Replacing \eqref{eq4.89} into \eqref{eq4.88}, we conclude that there exists a
generic constant $M > 1$ such that
\begin{align}\label{eq4.90}
\|r\|_\infty \vee |||r,R^{r\sharp}, R^{r\sharp\sharp}|||_{p-\mathrm{var},[a,b]}
\leq M\bigl(
\|z_a\| + |||z,R^{z\sharp},R^{r\sharp\sharp}|||_{p-\mathrm{var},[a,b]} \bigr) \leq M\|z_a\|,
\end{align}
where the last inequality is due to \eqref{eq4.2}. This proves the global Lipschitz continuity of $\frac{\partial \varphi}{\partial y}(t,\mathbf{B}^{H,\lambda},y_a)$ with respect to $y_a$.

\end{proof}

\begin{theorem}
There exists a solution to \eqref{eq0.1} on the interval
$[a,b]$ with the initial value $y_a$.
\end{theorem}

\begin{proof}
We apply the Doss-Sussmann transformation \cite{Sussmann} to the rough integral, thereby reducing the problem to proving the existence of a solution for the transformed system. The proof proceeds in several steps:

\noindent\textbf{Step 1.}
Let $\varphi(t,\mathbf{B}^{H, \lambda},y_a)$ and $\psi(t,\mathbf{B}^{H, \lambda},h_b)$ denote the solution maps of the forward equation \eqref{eq4.1} and the backward equation \eqref{eq4.83}, respectively. Then
$\varphi(t,\mathbf{B}^{H, \lambda})\circ\psi(t,\mathbf{B}^{H, \lambda}) = I^{d\times d}$, and due to Theorem 4.2 and Corollary 4.3, $\varphi,\psi$ are continuously differentiable with respect to $y_a$ and $h_b$ respectively. Furthermore,
\begin{align}\label{eq4.91}
\frac{\partial \varphi}{\partial y}(t,\mathbf{B}^{H, \lambda},y_a)\frac{\partial \psi}{\partial h}(t,\mathbf{B}^{H, \lambda},y_b)
= \frac{\partial \psi}{\partial h}(t,\mathbf{B}^{H, \lambda},y_b)\frac{\partial \varphi}{\partial y}(t,\mathbf{B}^{H, \lambda},y_a)
= I^{d\times d}.
\end{align}
Arguing as in \eqref{eq4.85} and \eqref{eq4.86}, we obtain
\begin{align}\label{eq4.92}
\left\|\frac{\partial \psi}{\partial h}(t,\mathbf{B}^{H, \lambda},y_b)\right\|_{\infty,[a,b]} \leq e^{(\log 2)\bar{N}_{[a,b]}(\mathbf{B}^{H, \lambda})}.
\end{align}
Now consider the ordinary differential equation
\begin{align}\label{eq4.93}
\dot{z}_t = \frac{\partial \psi}{\partial h}\bigl(t,\mathbf{B}^{H, \lambda},\varphi(t,B^{H, \lambda},z_t)\bigr) f\bigl(\varphi(t,B^{H, \lambda},z_t)\bigr) = F(t,z_t),\quad t \in [a,b].
\end{align}

As in Corollary 4.4, the map $\frac{\partial \psi}{\partial h}(t,\mathbf{B}^{H, \lambda},z)$ is globally Lipschitz continuous in $z$.
Moreover, $\varphi(t,x,z)$  is also globally Lipschitz continuous with respect to $z$.
Consequently, the composition $\frac{\partial \psi}{\partial h}\bigl(t,\mathbf{B}^{H, \lambda},\varphi(t,x,z_t)\bigr)$ inherits this global Lipschitz continuity in $z$. Given that
$f$ is also globally Lipschitz continuous and of linear growth, it follows that the function
$F(t,z)$ appearing on the right-hand side of \eqref{eq4.93} is locally Lipschitz continuous and satisfies a linear growth condition. Therefore, equation \eqref{eq4.93} admits a unique solution.

\noindent\textbf{Step 2.}
To this end, consider the transformation $y_t = \varphi(t,\mathbf{B}^{H, \lambda},z_t)$ for $t \in [a,b]$. We shall then prove that $y_t$ satisfies the equation
\begin{align}\label{eq4.94}
dy_t = f(y_t)dt + g(y_t)dB^{H, \lambda}_t.
\end{align}
We first obtain that
\begin{align}\label{eq4.95}
y_{s,t}
&= \varphi(t,\mathbf{B}^{H, \lambda},z_t) - \varphi(t,\mathbf{B}^{H, \lambda},z_s) + \varphi(t,\mathbf{B}^{H, \lambda},z_s) - \varphi(s,\mathbf{B}^{H, \lambda},z_s)
\nonumber\\
&= \bigl[ \varphi(t,\mathbf{B}^{H, \lambda},z_t) - \varphi(t,\mathbf{B}^{H, \lambda},z_s) \bigr] + \int_s^t g\bigl(\varphi(u,\mathbf{B}^{H, \lambda},z_s)\bigr) dB^{H, \lambda}_u.
\end{align}
In view of  \eqref{eq4.82-1} and Theorem 4.2,
the term in the square bracket satisfies
\begin{align}\label{eq4.96}
&\bigl\| \varphi(t,\mathbf{B}^{H, \lambda},z_t) - \varphi(t,\mathbf{B}^{H, \lambda},z_s) - \frac{\partial \varphi}{\partial z}(s,\mathbf{B}^{H, \lambda},z_s) z_{s,t} \bigr\| \nonumber\\
&\leq \bigl\| \varphi(t,\mathbf{B}^{H, \lambda},z_t) - \varphi(t,\mathbf{B}^{H, \lambda},z_s) - \frac{\partial \varphi}{\partial z}(t,\mathbf{B}^{H, \lambda},z_s) z_{s,t} \bigr\| \nonumber\\
&\quad + \bigl\| \frac{\partial \varphi}{\partial z}(t,\mathbf{B}^{H, \lambda},z_s) - \frac{\partial \varphi}{\partial z}(s,\mathbf{B}^{H, \lambda},z_s) \bigr\| \|z_{s,t}\| \nonumber\\
&\leq D\|z_{s,t}\|^2 + \bigl\| \int_s^t Dg\bigl(\varphi(u,\mathbf{B}^{H, \lambda},z_s)\bigr) \frac{\partial \varphi}{\partial z}(u,\mathbf{B}^{H, \lambda},z_s) dx_u \bigr\| \|z_{s,t}\| \nonumber\\
&\leq D(t-s)^{1+\alpha}
\end{align}
for some constant $D$.
On the other hand, it follows from \eqref{eq0-2.15} that
\begin{align}\label{eq4.97}
&\int_s^t g\bigl(\varphi(u,\mathbf{B}^{H, \lambda},z_s)\bigr)
dB^{H, \lambda}_u\nonumber\\
&= g\bigl(\varphi(s,\mathbf{B}^{H, \lambda},z_s)\bigr) \otimes B^{H, \lambda;1}_{s,t}
+ [g\bigl(\varphi(s,\mathbf{B}^{H, \lambda},z_s)]'
B^{H, \lambda;2}_{s,t} \nonumber\\
&
+
[g\bigl(\varphi(s,\mathbf{B}^{H, \lambda},z_s)]''
B^{H, \lambda;3}_{s,t}
+D(t-s)^{4\alpha} \nonumber\\
&= g(y_s) \otimes
B^{H, \lambda;1}_{s,t} +
Dg(y_s) g(y_s) \otimes
B^{H, \lambda;2}_{s,t}\nonumber\\
&
+ [D^{2}g(y_s) g^{2}(y_s)+Dg(y_s)Dg(y_s)g(y_s)] \otimes
B^{H, \lambda;3}_{s,t}
D(t-s)^{4\alpha}.
\end{align}
By combining \eqref{eq4.95}-\eqref{eq4.97}, we now can write
\begin{align}\label{eq4.98}
y_{s,t} &= f(y_s)(t-s) + g(y_s) \otimes
B^{H, \lambda;1}_{s,t} + Dg(y_s) g(y_s)
B^{H, \lambda;2}_{s,t} \nonumber\\
&
+
[D^{2}g(y_s) g^{2}(y_s)+Dg(y_s)Dg(y_s)g(y_s)] B^{H, \lambda;3}_{s,t}
+
\mathcal{O}\bigl((t-s)^{4\alpha}\bigr).
\end{align}
Equation \eqref{eq4.98} implies that $y$ is controlled by the driving path $B^{H,\lambda}$, with derivative
$y_s' = g\bigl(\varphi(s,\mathbf{B}^{H, \lambda},z_s)\bigr) = g(y_s)$. Consequently, $g(y)$ is also controlled by $B^{H,\lambda}$, and its derivative is given by $[g(y)]_s' = Dg(y_s) y_s' =
Dg(y_s) g(y_s)$ for all $s \in [a,b]$. Now, consider any finite partition $\Pi$ of $[s,t]$
such that its mesh size satisfies $|\Pi| = \max_{[u,v]\in\Pi} |v-u| \ll 1$. Then, from \eqref{eq4.98} we obtain
\begin{align}\label{eq4.99}
&y_{s,t}
= \sum_{[u,v]\in\Pi} y_{u,v} \nonumber\\
&= \sum_{[u,v]\in\Pi} \bigg( g(y_u) \otimes
B^{H, \lambda;1}_{u,v} +
Dg(y_u) g(y_u)
B^{H, \lambda;2}_{u,v}
+[D^{2}g(y_u) g^{2}(y_u)+
Dg(y_u)Dg(y_u)g(y_u)] B^{H, \lambda;3}_{u,v}
\bigg)\nonumber\\
&
+\sum_{[u,v]\in\Pi} f(y_u)(v-u)
+ \sum_{[u,v]\in\Pi}
\mathcal{O}\bigg((v-u)^{4\alpha}\bigg)
\nonumber\\
&=
\sum_{[u,v]\in\Pi} \bigl( g(y_u) \otimes
B^{H, \lambda;1}_{u,v} + Dg(y_u) g(y_u)
B^{H, \lambda;2}_{u,v}
+[D^{2}g(y_u) g^{2}(y_u)+Dg(y_u)Dg(y_u)g(y_u)]
B^{H, \lambda;3}_{u,v}
\bigg)\nonumber\\
&
+\sum_{[u,v]\in\Pi} f(y_u)(v-u) + \mathcal{O}\bigl(|\Pi|^{4\alpha-1}\bigr).
\end{align}
Let $|\Pi| \to 0$, the first term in \eqref{eq4.99} converges to $\int_s^t f(y_u) du$ while the second
term converges to the Gubinelli rough integral $\int_s^t g(y_u) dB^{H,\lambda}_u$. We conclude
that
\[
y_{s,t} = \int_s^t f(y_u) du + \int_s^t g(y_u) dB^{H,\lambda}_u,
\]
which proves the existence part.

\end{proof}

We now present the uniqueness theorem for the solution, whose proof is analogous to that of Theorem 3.9 in \cite{Duc-2020} and is therefore omitted here.

\begin{theorem}
The solution $y_{t}(\mathbf{B}^{H, \lambda}, y_{a})$ of \eqref{eq0.1} is uniformly continuous with respect to $y_{a}$. In particular, there exists a unique solution with respect to the initial condition $y_{a}$.

\end{theorem}

\appendix
\noindent{\section{Appendix}}

\begin{lemma}
Assume the conditions in Lemma 2.1 hold, then we have for $l>r$,
\begin{align}
&\bigg[
\mathbb{E}
\bigg(
\Delta_{2l-1}^{m+1}B^{H,\lambda}
\Delta_{2r-1}^{m+1}B^{H,\lambda}\bigg)\bigg]^{2}
\nonumber\\
&~~-
\mathbb{E}\bigg(
\Delta_{2l-1}^{m+1}B^{H,\lambda}
\Delta_{2r}^{m+1}B^{H,\lambda}\bigg)
\mathbb{E}\bigg(
\Delta_{2l}^{m+1}B^{H,\lambda}
\Delta_{2r-1}^{m+1}B^{H,\lambda}\bigg)\nonumber\\
&\leq \frac{c_{H, \lambda}}{2^{(2H+4)m}}
+\frac{c_{H, \lambda}}{2^{(4H+2)m}}.\nonumber
\end{align}

\end{lemma}

\begin{proof}
Thanks to \eqref{eq2.1-1}, it holds that
\begin{align}\label{eqA.1}
&\bigg[
\mathbb{E}
\bigg(
\Delta_{2l-1}^{m+1}B^{H,\lambda}
\Delta_{2r-1}^{m+1}B^{H,\lambda}\bigg)\bigg]^{2}
\nonumber\\
&~~-
\mathbb{E}\bigg(
\Delta_{2l-1}^{m+1}B^{H,\lambda}
\Delta_{2r}^{m+1}B^{H,\lambda}\bigg)
\mathbb{E}\bigg(
\Delta_{2l}^{m+1}B^{H,\lambda}
\Delta_{2r-1}^{m+1}B^{H,\lambda}\bigg)\nonumber\\
&=\bigg[C_{\frac{2l-2r+1}{2^{m+1}}}^{2}
\bigg|\frac{2l-2r+1}{2^{m+1}}\bigg|^{2H}
+
C_{\frac{2l-2r-1}{2^{m+1}}}^{2}
\bigg|\frac{2l-2r-1}{2^{m+1}}\bigg|^{2H}
-2C_{\frac{2l-2r}{2^{m+1}}}^{2}
\bigg|\frac{2l-2r}{2^{m+1}}\bigg|^{2H}\bigg]^{2}
\nonumber\\
&~~-
\bigg[C_{\frac{2l-2r}{2^{m+1}}}^{2}
\bigg|\frac{2l-2r}{2^{m+1}}\bigg|^{2H}
+
C_{\frac{2l-2r-2}{2^{m+1}}}^{2}
\bigg|\frac{2l-2r-2}{2^{m+1}}\bigg|^{2H}
-2C_{\frac{2l-2r-1}{2^{m+1}}}^{2}
\bigg|\frac{2l-2r-1}{2^{m+1}}\bigg|^{2H}\bigg]
\nonumber\\
&~~\times
\bigg[C_{\frac{2l-2r+2}{2^{m+1}}}^{2}
\bigg|\frac{2l-2r+2}{2^{m+1}}\bigg|^{2H}
+
C_{\frac{2l-2r}{2^{m+1}}}^{2}
\bigg|\frac{2l-2r}{2^{m+1}}\bigg|^{2H}
-2C_{\frac{2l-2r+1}{2^{m+1}}}^{2}
\bigg|\frac{2l-2r+1}{2^{m+1}}\bigg|^{2H}\bigg]
\nonumber\\
&=\bigg[C_{t+h}^{2}|t+h|^{2H}
+
C_{t-h}^{2}|t-h|^{2H}
-2C_{t}^{2}|t|^{2H}\bigg]^{2}
\nonumber\\
&~~-
\bigg[C_{t}^{2}|t|^{2H}
+
C_{t-2h}^{2}|t-2h|^{2H}
-2C_{t-h}^{2}|t-h|^{2H}\bigg]
\nonumber\\
&~~\times
\bigg[C_{t+2h}^{2}|t+2h|^{2H}
+
C_{t}^{2}|t|^{2H}
-2C_{t+h}^{2}|t+h|^{2H}\bigg],
\end{align}
where for simplicity of notation we denote $t:=\frac{2l-2r}{2^{m+1}}$ and $h:=\frac{1}{2^{m+1}}$. Using the Taylor expansion, we have
\begin{align}
&\bigg[C_{t+h}^{2}|t+h|^{2H}
+
C_{t-h}^{2}|t-h|^{2H}
-2C_{t}^{2}|t|^{2H}\bigg]^{2}
\nonumber\\
&=
\bigg([C_{t}^{2}|t|^{2H}]''\bigg)^{2}h^{4}
+
[C_{t}^{2}|t|^{2H}]''
[C_{t}^{2}|t|^{2H}]^{(4)}\frac{h^{6}}{6}+\cdots,
\nonumber
\end{align}
and
\begin{align}
&\bigg[C_{t}^{2}|t|^{2H}
+
C_{t-2h}^{2}|t-2h|^{2H}
-2C_{t-h}^{2}|t-h|^{2H}\bigg]
\nonumber\\
&~~\times
\bigg[C_{t+2h}^{2}|t+2h|^{2H}
+
C_{t}^{2}|t|^{2H}
-2C_{t+h}^{2}|t+h|^{2H}\bigg]\nonumber\\
&=\bigg([C_{t}^{2}|t|^{2H}]''\bigg)^{2}h^{4}
+\bigg[(1+\frac{1}{6})
[C_{t}^{2}|t|^{2H}]''
[C_{t}^{2}|t|^{2H}]^{(4)}
-\bigg([C_{t}^{2}|t|^{2H}]^{(3)}\bigg)^{2}\bigg]
h^{6}+\cdots.
\nonumber
\end{align}
Furthermore,
\begin{align}\label{eqA.2}
&\bigg[C_{t+h}^{2}|t+h|^{2H}
+
C_{t-h}^{2}|t-h|^{2H}
-2C_{t}^{2}|t|^{2H}\bigg]^{2}
\nonumber\\
&~~-
\bigg[C_{t}^{2}|t|^{2H}
+
C_{t-2h}^{2}|t-2h|^{2H}
-2C_{t-h}^{2}|t-h|^{2H}\bigg]
\nonumber\\
&~~\times
\bigg[C_{t+2h}^{2}|t+2h|^{2H}
+
C_{t}^{2}|t|^{2H}
-2C_{t+h}^{2}|t+h|^{2H}\bigg]\nonumber\\
&=\bigg[\bigg([C_{t}^{2}|t|^{2H}]^{(3)}\bigg)^{2}-
[C_{t}^{2}|t|^{2H}]''
[C_{t}^{2}|t|^{2H}]^{(4)}\bigg]h^{6}+\cdots.
\end{align}
Set $A:=\frac{2\Gamma(2H)}
{(2\lambda)^{2H}}$ and $D:=\frac{2\Gamma(H+\frac{1}{2})}
{\sqrt{\pi}(2\lambda)^{H}}$, then by \eqref{eq2.1-1}, $C_{t}^{2}|t|^{2H}
=A-Dt^{H}K_{H}(\lambda t)$. Let $z:=\lambda t$ and $g(z):=z^{H}K_{H}(z)$, it follows that
\begin{align}\label{eqA.3}
&\bigg([C_{t}^{2}|t|^{2H}]^{(3)}\bigg)^{2}-
[C_{t}^{2}|t|^{2H}]''
[C_{t}^{2}|t|^{2H}]^{(4)}\nonumber\\
&=D^{2}\lambda^{6-2H}
\bigg[[g^{(3)}(z)]^{2}-g^{(2)}(z)g^{(4)}(z)
\bigg].
\end{align}
According to the relation $\frac{d}{dx}(z^{v}K_{v}(z))=-z^{v}K_{v-1}(z)$ for all $v\in \mathbb{R}$ (see e.g., Appendix in \cite{Gaunt}), we can  immediately deduce that
\begin{align}
g'(z)=-z^{H}K_{H-1}(z),\nonumber
\end{align}
\begin{align}\label{eqA.3-10}
g''(z)=-z^{H-1}K_{H-1}(z)
+z^{H}K_{H-2}(z),
\end{align}
\begin{align}
g^{(3)}(z)=3z^{H-1}K_{H-2}(z)
-z^{H}K_{H-3}(z),\nonumber
\end{align}
and
\begin{align}
g^{(4)}(z)=3z^{H-2}K_{H-2}(z)
-6z^{H-1}K_{H-3}(z)+z^{H}K_{H-4}(z).\nonumber
\end{align}
Hence,
\begin{align}\label{eqA.4}
&[g^{(3)}(z)]^{2}-g^{(2)}(z)g^{(4)}(z)\nonumber\\
&=3z^{2H-3}K_{H-1}(z)K_{H-2}(z)+
z^{2H-2}(6K^{2}_{H-2}(z)-
6K_{H-1}(z)K_{H-3}(z))\nonumber\\
&~~~~+z^{2H-1}K_{H-1}(z)K_{H-4}(z)
+z^{2H}(K^{2}_{H-3}(z)-K_{H-2}(z)K_{H-4}(z)).
\end{align}
In view of $K_{v-1}(z)=K_{v+1}(z)-\frac{2v}{z}K_{v}(z)$ for $v\in \mathbb{R}$, we see that
\begin{align}\label{eqA.5}
K_{H-2}(z)=K_{H}(z)-\frac{2(H-1)}{z}K_{H-1}(z),
\end{align}
\begin{align}\label{eqA.6}
K_{H-3}(z)&=K_{H-1}(z)-\frac{2(H-2)}{z}K_{H-2}(z)
\nonumber\\
&=\bigg(1+\frac{4(H-1)(H-2)}{z^{2}}\bigg)K_{H-1}(z)
-\frac{2(H-2)}{z}K_{H}(z),
\end{align}
\begin{align}\label{eqA.7}
K_{H-4}(z)&=K_{H-2}(z)-\frac{2(H-3)}{z}K_{H-3}(z)
\nonumber\\
&=\bigg(1+\frac{4(H-2)(H-3)}{z^{2}}\bigg)K_{H}(z)
\nonumber\\
&~~~~-\bigg(\frac{4H-8}{z}+\frac{8(H-1)(H-2)(H-3)}{z^{3}}
\bigg)K_{H-1}(z).
\end{align}
Inserting \eqref{eqA.5}-\eqref{eqA.7} into \eqref{eqA.4} results in
\begin{align}
&[g^{(3)}(z)]^{2}-g^{(2)}(z)g^{(4)}(z)
=\bigg[-z^{2H}+(4H-2)z^{2H-2}\bigg]K^{2}_{H}(z)
\nonumber\\
&~~+\bigg[(2H-1)z^{2H-1}+(-12H^{2}+16H-5)z^{2H-3}
\bigg]K_{H}(z)K_{H-1}(z)\nonumber\\
&~~+\bigg[z^{2H}+(2-4H)z^{2H-2}+2(H-1)(2H-1)^{2}
z^{2H-4}\bigg]K^{2}_{H-1}(z)\nonumber\\
&~~\leq \bigg[z^{2H}+(2-4H)z^{2H-2}\bigg]
K^{2}_{H-1}(z),\nonumber
\end{align}
where in the last inequality we have used the condition $H\in (\frac{1}{4}, \frac{1}{2})$. Taking into account of \eqref{eqA.2}-\eqref{eqA.3}, we have
\begin{align}\label{eqA.9}
&\bigg[C_{t+h}^{2}|t+h|^{2H}
+
C_{t-h}^{2}|t-h|^{2H}
-2C_{t}^{2}|t|^{2H}\bigg]^{2}
\nonumber\\
&~~-
\bigg[C_{t}^{2}|t|^{2H}
+
C_{t-2h}^{2}|t-2h|^{2H}
-2C_{t-h}^{2}|t-h|^{2H}\bigg]
\nonumber\\
&~~\times
\bigg[C_{t+2h}^{2}|t+2h|^{2H}
+
C_{t}^{2}|t|^{2H}
-2C_{t+h}^{2}|t+h|^{2H}\bigg]\nonumber\\
&\leq
D^{2}\lambda^{6-2H}
\bigg[z^{2H}+(2-4H)z^{2H-2}\bigg]
K^{2}_{H-1}(z)h^{6}
\end{align}
where $D=\frac{2\Gamma(H+\frac{1}{2})}
{\sqrt{\pi}(2\lambda)^{H}}$. By Corollary 3.4 in \cite{Yang-Zheng} that
\begin{align}\label{eqA.10-10}
K_{v}(x)< \sqrt{\frac{\pi}{2}}
\frac{(1+1/x)^{v}}{\sqrt{x+1}}e^{-x},~~~v\in (\frac{1}{2}, \frac{3}{2}), x\in (0,\infty),
\end{align}
in view of $z=\lambda t$, $t=\frac{2l-2r}{2^{m+1}}$ and $h=\frac{1}{2^{m+1}}$, we can see that
\begin{align}\label{eqA.10}
&D^{2}\lambda^{6-2H}
\bigg[z^{2H}+(2-4H)z^{2H-2}\bigg]
K^{2}_{H-1}(z)h^{6}\nonumber\\
&\leq \bigg(\frac{c_{H, \lambda}}{2^{(2H-2)m}}
+\frac{c_{H, \lambda}}{2^{(4H-4)m}}\bigg)
\frac{1}{2^{6m}}\nonumber\\
&=\frac{c_{H, \lambda}}{2^{(2H+4)m}}
+\frac{c_{H, \lambda}}{2^{(4H+2)m}}.
\end{align}
Therefore the assertion can be obtained immediately by substituting \eqref{eqA.10} into \eqref{eqA.9} and taking into account of \eqref{eqA.1}.

\end{proof}

\begin{lemma}
Assume the conditions in Lemma 2.1 hold, then we have for any $l,r \in \bigl[2^{(m-n)}(k-1)+1,\,2^{(m-n)}k\bigr]$,
\[
\mathbb{E}\bigl(\triangle_{2l}^{m+1} B^{H,\lambda} \triangle_{2r}^{m+1} B^{H,\lambda}\bigr)
\leq
\begin{cases}
c_{H,\lambda} |l-r|^{2H-2}\frac{1}{2^{2mH}}
, & l \neq r; \\[10pt]
2C_{\frac{1}{2^m}}^2 \dfrac{1}{2^{2mH}}, & l = r.
\end{cases}
\]
\end{lemma}

\begin{proof}
When \(l \neq r\), set \(t := \dfrac{2l-2r}{2^{m+1}}\) and \(h := \dfrac{1}{2^{m+1}}\), then
\begin{align}
\mathbb{E}\bigl(\triangle_{2l}^{m+1} B^{H,\lambda} \triangle_{2r}^{m+1} B^{H,\lambda}\bigr)
&= C_{t+h}^2|t+h|^{2H} + C_{t-h}^2|t-h|^{2H} - 2C_t^2|t|^{2H}
\nonumber\\
&= \bigl[C_t^2|t|^{2H}\bigr]'' h^2 + \bigl[C_t^2|t|^{2H}\bigr]^{(4)} \frac{h^4}{12} + \dots.
\end{align}
We denote \(D := \dfrac{2\Gamma\bigl(H+\frac{1}{2}\bigr)}{\sqrt{\pi}(2\lambda)^H}\) and \(z := \lambda t\). In view of \eqref{eqA.3-10} and \eqref{eqA.5}, it holds that
\begin{align}
\bigl[C_t^2|t|^{2H}\bigr]''
&= D\lambda^{2-H}\bigl[-z^{H-1}K_{H-1}(z) + z^H K_{H-2}(z)\bigr]
\nonumber\\
&= D\lambda^{2-H}\bigl[(1-2H)z^{H-1}K_{H-1}(z) + z^H K_H(z)\bigr].
\end{align}
Thanks to \eqref{eqA.10-10}, $z^{H-1}K_{H-1}(z)\leq c z^{2H-2}$. By Corollary 3.4 in \cite{Yang-Zheng} that
\begin{align}
K_v(x) < 2^{v-1}\Gamma(v) \frac{(1+1/x)^v}{\sqrt{x+1}} e^{-x},\quad v \in \bigl(0,\tfrac12\bigr),\,x \in (0,\infty),
\end{align}
then \(z^H K_H(z) \leq cz^{H-\frac{1}{2}}\). Then
\begin{align}
\bigl[C_t^2|t|^{2H}\bigr]'' h^2
&\leq c_{H,\lambda}
\left[
\left(\frac{|l-r|}{2^{m+1}}\right)^{2H-2}+
\left(\frac{|l-r|}{2^{m+1}}\right)^{H-\frac{1}{2}}
\right] \frac{1}{2^{2m}}\nonumber\\
&\leq c_{H,\lambda} |l-r|^{2H-2}\frac{1}{2^{2mH}},
\end{align}
and thus
\begin{align}
\mathbb{E}\bigl(\triangle_{2l}^{m+1} B^{H,\lambda} \triangle_{2r}^{m+1} B^{H,\lambda}\bigr)
\leq c_{H,\lambda} |l-r|^{2H-2}\frac{1}{2^{2mH}}.
\end{align}

When \(l = r\), it is easy to see that
\[
\mathbb{E}\bigl(\triangle_{2l}^{m+1} B^{H,\lambda} \triangle_{2r}^{m+1} B^{H,\lambda}\bigr)
= 2C_h^2|h|^{2H} \leq 2C_{\frac{1}{2^m}}^2 \frac{1}{2^{2mH}}.
\]

\end{proof}

\end{document}